\documentclass[11pt,a4paper]{amsart}

\usepackage{amssymb,amsthm,tikz}
\usetikzlibrary{positioning}
\usepackage[colorlinks=true]{hyperref}
  
\tikzset{sgplattice/.style={inner sep=1pt,norm/.style={red!50!blue},char/.style={blue!50!black},
  lin/.style={black!50}},cnj/.style={black!50,yshift=-2.5pt,left=-1pt of #1,scale=0.5,fill=white}}
\usepackage{stmaryrd} % \longmapsfrom
\usepackage{mathrsfs} % \mathscr
\usepackage{colonequals}
\usepackage{geometry}
\newcommand{\repo}[2]{\href{https://github.com/nt-lib/local-heights-X0Nstar/blob/v1.0.0/#1}{\texttt{#2}}}
\newcommand{\ghfile}[1]{\repo{#1}{#1}}
\newcommand{\ghdir}[1]{\href{https://github.com/nt-lib/local-heights-X0Nstar/tree/v1.0.0/#1}{\texttt{#1}}}

\usepackage{longtable}
\usepackage{booktabs}

\usepackage{comment}
\usepackage{soul}

\usepackage[english]{babel}
\DeclareFontFamily{U}{wncy}{}
\DeclareFontShape{U}{wncy}{m}{n}{<->wncyr10}{}
\DeclareSymbolFont{mcy}{U}{wncy}{m}{n}
\DeclareMathSymbol{\Sha}{\mathord}{mcy}{"58}

\usepackage{csquotes}
\usepackage{microtype}

\usepackage{enumerate}

\usepackage[
backend=biber,
style=alphabetic,
sorting=nyt,
backref=true,
maxitems=10,
minitems=10,
maxbibnames=10,
minbibnames=10,
maxcitenames=4,
mincitenames=4,
giveninits=true,
maxalphanames=5,
minalphanames=5,
doi=false,
isbn=false,
url=false
]{biblatex}
\AtEveryBibitem{\clearlist{language}}

\usepackage{aliascnt}
\usepackage{tikz-cd}
\usetikzlibrary{intersections,calc}
\usepackage{wasysym}
\usepackage{enumitem}

\newif\iftodonotes
\todonotestrue   % <-- set to \todonotesfalse to disable comments

\iftodonotes
  \usepackage{todonotes}
  \makeatletter
  \providecommand\@dotsep{5}
  \renewcommand{\listoftodos}[1][\@todonotes@todolistname]{%
    \@starttoc{tdo}{#1}}
  \makeatother
\else
  \usepackage[disable]{todonotes}
\fi

\usepackage{cleveref}
\renewcommand{\epsilon}{\varepsilon}
\renewcommand{\phi}{\varphi}

\def\Alphabet{A,B,C,D,E,F,G,H,I,J,K,L,M,N,O,P,Q,R,S,T,U,V,W,X,Y,Z}%  Capitalized Alphabet
\def\alphabet{a,b,c,d,e,f,g,h,i,j,k,l,m,n,o,p,q,r,s,t,u,v,w,x,y,z}%	lowercase alphabet
\def\endpiece{xxx}%									marks end of list
\def\makeAlphabet[#1]{\expandafter\makeA#1,xxx,}%		Ex. \makeAlphabet[A,B]
\def\makealphabet[#1]{\expandafter\makea#1,xxx,}%		Ex. \makealphabet[c,d]
\def\makeA#1,{\def\temp{#1}\ifx\temp\endpiece\else%
	\mkbb{#1}\mkfrak{#1}\mkbf{#1}\mkcal{#1}\mkscr{#1}\mkbs{#1}\expandafter\makeA\fi}%
\def\makea#1,{\def\temp{#1}\ifx\temp\endpiece\else\mkfrak{#1}\mkbf{#1}\mkbs{#1}\expandafter\makea\fi}%
\def\mkbb#1{\expandafter\def\csname bb#1\endcsname{\mathbb{#1}}}%      Define bb
\def\mkfrak#1{\expandafter\def\csname fr#1\endcsname{\mathfrak{#1}}}%    Define frak
\def\mkbf#1{\expandafter\def\csname b#1\endcsname{\mathbf{#1}}}%           Define bold letters
\def\mkcal#1{\expandafter\def\csname c#1\endcsname{\mathcal{#1}}}%       Define calligraphy
\def\mkscr#1{\expandafter\def\csname s#1\endcsname{\mathscr{#1}}}%       Define script
\def\mkbs#1{\expandafter\def\csname bs#1\endcsname{{\boldsymbol{#1}}}}%       Define bold symbol
\def\makeop[#1]{\xmakeop#1,xxx,}%					Ex. \makeop[Hom,Spec]
\def\mkop#1{\expandafter\def\csname #1\endcsname{{\operatorname{#1}}}} %
\def\xmakeop#1,{\def\temp{#1}\ifx\temp\endpiece\else\mkop{#1}\expandafter\xmakeop\fi}%
\def\makeup[#1]{\xmakeup#1,xxx,}%					Ex. \makeop[Hom,Spec]
\def\mkup#1{\expandafter\def\csname #1\endcsname{{\mathrm{#1}\,}}} %
\def\xmakeup#1,{\def\temp{#1}\ifx\temp\endpiece\else\mkup{#1}\expandafter\xmakeup\fi}%
\makeAlphabet[\Alphabet]%				Define bb, frak, bf, cal for Capitalized Alphabet
\makealphabet[\alphabet]%  				Define frak and bf for uncapitalized alphabet
\makeop[Hom,Aut,End,Mor,SL,GL,ord,Irr,Ell,Gal,Cl,Pic,NS,d,Re,Im,res,Symb,Ev,Char,Ram,SU,BSD,Res,cork,Tam,id,Div,div]
\makeup[Spec,Proj,dR,new,old,AJ,tr,ker,im,coker,tors,trd,nrd]

\newcommand{\Q}{{\mathbb{Q}}}

\newcommand{\Z}{{\mathbb{Z}}}

\newcommand{\F}{{\mathbb{F}}}
\newcommand{\Fbar}{{\overline{\F}}}
\newcommand{\A}{{\mathbb{A}}}
\newcommand{\T}{{\mathbb{T}}}
\newcommand{\R}{{\mathbb{R}}}
\newcommand{\C}{{\mathbb{C}}}
\newcommand{\OO}{{\mathcal{O}}}

\renewcommand{\P}{{\mathbb{P}}}
\newcommand{\characteristic}{\operatorname{char}}

\newcommand{\graph}{\mathcal{G}}

\newcommand{\inj}{\hookrightarrow}

\newcommand{\defn}{\emph}

\newcommand{\legendre}[2]{\left(\frac{#1}{#2}\right)}

\newcommand{\disc}{\operatorname{disc}}

\newcommand{\Fellbar}{\overline{\F}_\ell}

\newcommand{\orderO}{\mathcal{O}}

\newcommand{\ds}{\displaystyle}
\newcommand{\set}[1]{\left\lbrace #1 \right\rbrace}

\renewcommand{\i}{\sqrt{-1}}

\DeclareMathOperator{\Frob}{Frob}
\DeclareMathOperator{\Cls}{Cls}

\newcommand{\LMFDBLabel}[1]{\textnormal{\href{https://www.lmfdb.org/EllipticCurve/Q/#1/}{\texttt{#1}}}}
\newcommand{\LMFDBLabelMF}[1]{\textnormal{\href{https://www.lmfdb.org/ModularForm/GL2/Q/holomorphic/#1/}{\texttt{#1}}}}
\newcommand{\newaliastheorem}[2]{%
  \newaliascnt{#1}{theorem}%
  \newtheorem{#1}[#1]{#2}%
  \aliascntresetthe{#1}%
}

\theoremstyle{plain}
\newtheorem{theorem}{Theorem}[section]
\newaliastheorem{lemma}{Lemma}
\newaliastheorem{corollary}{Corollary}
\newaliastheorem{proposition}{Proposition}
\newaliastheorem{conjecture}{Conjecture}
\newaliastheorem{question}{Question}
\newaliastheorem{assumption}{Assumption}

\theoremstyle{definition}
\newaliastheorem{definition}{Definition}
\newaliastheorem{definition-proposition}{Definition-Proposition}
\newaliastheorem{algorithm}{Algorithm}
\newaliastheorem{setup}{Setup}

\theoremstyle{remark}
\newaliastheorem{example}{Example}
\newaliastheorem{examples}{Examples}
\newaliastheorem{remark}{Remark}
\newaliastheorem{notation}{Notation}
\newaliastheorem{acknowledgements}{Acknowledgements}
\newaliastheorem{convention}{Convention}
\newaliastheorem{analysis}{Analysis}

\crefname{theorem}{Theorem}{Theorems}
\crefname{lemma}{Lemma}{Lemmata}
\crefname{corollary}{Corollary}{Corollaries}
\crefname{proposition}{Proposition}{Propositions}
\crefname{definition}{Definition}{Definitions}
\crefname{definition-proposition}{Definition-Proposition}{Definition-Propositions}
\crefname{conjecture}{Conjecture}{Conjectures}
\crefname{question}{Question}{Questions}
\crefname{example}{Example}{Examples}
\crefname{algorithm}{Algorithm}{Algorithms}
\crefname{remark}{Remark}{Remarks}
\crefname{assumption}{Assumption}{Assumptions}
\crefname{setup}{Setup}{Setups}

\numberwithin{equation}{subsection}

\definecolor{darkgreen}{rgb}{0,0.5,0}
\definecolor{darkblue}{rgb}{0,0,0.8}
\definecolor{darkred}{rgb}{0.8,0,0}

\renewcommand{\geq}{\geqslant}
\renewcommand{\leq}{\leqslant}
\renewcommand{\le}{\leqslant}
\renewcommand{\ge}{\geqslant}

\newcommand{\Czero}{10-4\log 3}

\makeatletter
\let\@wraptoccontribs\wraptoccontribs
\makeatother

\title[Semi-stable models, local heights and quadratic Chabauty for $X_0(N)^*$]{Semi-stable models, local heights\\ and quadratic Chabauty for $X_0(N)^*$}

\author{Nikola Ad\v{z}aga}
\address{Nikola Ad\v{z}aga\\
University of Zagreb Faculty of Civil Engineering, Department of Mathematics\\
10000 Zagreb, Croatia}
\email{nikola.adzaga@grad.unizg.hr}
\urladdr{\url{https://www.grad.unizg.hr/en/nikola.adzaga}}
\author{Maarten Derickx}
\address{Maarten Derickx\\
University of Zagreb Faculty of Science, Department of Mathematics\\
10000 Zagreb, Croatia}
\email{maarten@mderickx.nl}
\urladdr{\url{http://www.maartenderickx.nl/}}
\author{Timo Keller}
\address{Timo Keller, IBM Deutschland Research \& Development, IBM Campus 1, 71139 Ehningen, Germany;
Institut für Mathematik, Universität Würzburg, Emil-Fischer-Strasse 30, 97074, Würz\-burg, Germany;
Rijksuniversiteit Groningen, Bernoulli Institute, Bernoulliborg, Nijenborgh 9, 9747 AG Groningen, The Netherlands;
Leibniz Universität Hannover, Institut für Algebra, Zahlentheorie und Diskrete Mathematik, Welfengarten 1, 30167 Hannover, Germany}
\email{math@kellertimo.de}
\urladdr{\url{https://www.timo-keller.de}}



\begin{document}
\begin{abstract}
Quadratic Chabauty computations for $X_0(N)^*$ are complicated by local
height contributions at primes of bad reduction. For squarefree $N$, we explain a strategy for constructing a global $p$-adic height for which all the contributions at the primes of bad reduction vanish. As our first main result, we show that such a
$p$-adic height exists for all squarefree levels $N>714$. In order to do this, we first give an explicit description of the minimal regular model of $X_0(N)^*$ at primes of bad reduction, and show this model is  semi-stable. Each component of this model gives rise to a linear condition which ensures that the contribution at that component vanishes.
Using embeddings of quadratic orders into quaternion algebras, we obtain an
upper bound for the number of these conditions; when the genus of $X_0(N)^*$ exceeds this bound by more than one, there is enough freedom to choose a suitable
correspondence, simplifying quadratic Chabauty computations.
As an application, we determine the rational points on several curves $X_0(N)^*$ for which this was not done before.
\end{abstract}

\maketitle

%\tableofcontents

\section{Introduction}
In this paper, we study $X_0(N)^*$, the quotient of the modular curve $X_0(N)$ by the full group of Atkin--Lehner involutions.
Carrying out the quadratic Chabauty program to compute the rational points on $X_0(N)^*$ requires computing a certain global $p$-adic height function. This height is built up as a sum of the local heights at the primes of bad reduction and a local height at $p$. Computing the height contributions at the primes of bad reduction has been a bottleneck in extending the recent quadratic Chabauty computations in \cites{BDMTV2,AABCCKW,ANTStar} to other levels $N$. This motivates the following definition.
\begin{definition}
We call a $p$-adic height {\it supported at $p$} if all the local height contributions at the primes of bad reduction are zero.
\end{definition}
It should be clear that $p$-adic heights {\it supported at $p$} are much easier to work with both theoretically and practically, as for them there is no need to compute the height contributions at the primes of bad reduction. A question that naturally arises is whether such $p$-adic heights actually exist. Our first main result answers this affirmatively for $X_0(N)^*$.
\begin{theorem}\label{thm:main1}
Let $N>714$ be squarefree and $p\nmid N$ a prime. Then there exists a non-trivial $p$-adic height on $X_0(N)^*$ that is supported at $p$.
\end{theorem}

In fact, for squarefree $N\leq 714$ there are exactly $36$ curves of the form $X_0(N)^*$ of genus at least two where no $p$-adic height supported at $p$ exists. The list of these $36$ integers can be found in \Cref{thm:existence_of_supported_at_p_heights}.

Note that $p$-adic heights on a curve depend on the choice of a basepoint. When $N$ is squarefree $X_0(N)^*$ has a unique cusp and that cusp is $\Q$-rational. Unless otherwise specified, we always use that cusp as our basepoint.

\begin{theorem}\label{thm:main1.1}
Let $N$ be squarefree and $p\nmid N$ a prime. Then the number of linearly independent $p$-adic heights supported at $p$ on $X_0(N)^*$ is at least $g(X_0(N)^*)-O(\sqrt N\log N)$, with an absolute implied constant. In particular, for every $\varepsilon>0$ it is at least $N^{1-\varepsilon}$ for all sufficiently large $N$.
\end{theorem}
For a more precise and explicit version, see~\Cref{lem:threshold_reduction} and the text below it.

See \Cref{def:padic_height} and \Cref{rem:linear_independence} for a clarification on what we mean precisely by the existence of linearly independent $p$-adic heights.
Knowing how many linearly independent heights there are is useful if, for example, one wants to use $p$-adic heights supported at $p$ on a quotient $A$ of $J_0(N)^*$ where the Mordell--Weil rank of $A$ equals its dimension as in \cite{DograLeFourn}.

A crucial ingredient in proving the above theorems is describing the stable model of $X_0(N)^*$.

\begin{theorem}\label{thm:main2}
Let $N$ be an integer and $\ell$ be a prime that exactly divides $N$. Assume that $X_0(N)^*$ has genus at least $2$. Then the curve $X_0(N)^*$ has stable reduction over $\Z_\ell$, and hence its minimal regular model is semi-stable. 
\end{theorem}

A more precise description of the stable model of $X_0(N)^*$ over $\Z_\ell$ and that of other Atkin--Lehner quotients of $X_0(N)$ can be found in \Cref{thm:semi-stable_star}. Our description of the stable model of $X_0(N)^*$ directly translates to a description of the dual graph of the special fibre at $\ell$ of  $X_0(N)^*$, see \Cref{def:dual_graph_quotient} and \Cref{cor:dual_graph_quotient}.

The minimal desingularisation of a stable curve gives the minimal regular model of that curve and our work also gives a precise description of the minimal regular model of $X_0(N)^*$, see \Cref{cor:minimal_regular_AL_quot}.

Aside from showing that there are often $p$-adic heights on $X_0(N)^*$ supported only at $p$, we actually give an explicit way to construct such $p$-adic heights when they are guaranteed to exist.

\begin{proposition}\label{prop:correspondence_algorithm}
Let $N$ be a squarefree integer, put $g\colonequals g(X_0(N)^*)$, and let $\T\subseteq \End \, J_0(N)^*$ be the Hecke algebra generated by all the Hecke operators $T_q$ with $q\nmid N$. Assume there exists a single prime $q\nmid N$ such that $T_q$ generates $\T\otimes\Q$ as a $\Q$-algebra. Then there exists an algorithm that takes $N$ and $q$ as input, and outputs a sequence of polynomials $f_1,\ldots,f_n\in\Q[x]$ with the following properties: 
\begin{enumerate}
    \item For all primes $p$ not dividing $N$ the $p$-adic heights associated to the correspondences $f_1(T_q),\ldots,f_n(T_q)$ form a basis of the $p$-adic heights supported at $p$.
    \item Each $f_i$ has degree at most $g-1$.
    \item Each correspondence $f_i(T_q)$ has trace $0$.
\end{enumerate}
\end{proposition}
One interesting thing to note is that the same set of polynomials $f_1,\ldots,f_n$ works for all primes $p \nmid N$. A different way to say this is that the subspace of correspondences in $\T \otimes \Q$ giving rise to $p$-adic heights supported at $p$ does not depend on the prime $p$. 

The above algorithm can also be used to determine whether there exists a $p$-adic height supported at $p$ or not. Indeed, such a height exists if and only if the above algorithm returns a non-empty sequence.  We do not know whether $\T\otimes \Q$ can always be generated by a single $T_q$. However, for any given $T_q$ it is easy to check whether it is a generator, see \Cref{rem:Tq_generator}. For our applications, it is not a big restriction to require $T_q$ to be a generator of the Hecke algebra, as we use these polynomials as input to \texttt{QCMod}, and at the time of writing \texttt{QCMod} requires $q=p$ and $T_p$ to be a generator of $\T\otimes \Q$. Additionally, the algorithm can be modified to the general case where $\T \otimes \Q$ is generated by a finite set of Hecke operators $T_{q_1}, \ldots, T_{q_k}$, see \Cref{rem:single_gen}.

The algorithm of \Cref{prop:correspondence_algorithm} is model-free. Indeed, apart from computing traces of Hecke operators on $S_2(\Gamma_0(N))^{W_N}$ via modular symbols, the computations in the algorithm are done in terms of the action of Hecke operators on Brandt modules, which only requires ideal arithmetic in quaternion algebras. In particular, the algorithm does not require any algebraic equations for $X_0(N)^*$. This makes the algorithm fast in practice. For example, it happily computes the basis for the genus 117 curve $X_0(5001)^*=X_0(3\cdot1667)^*$ in around 6 minutes.\footnote{This computation was carried out on an AMD EPYC 9175F 16-Core Processor using only a single thread.}
\begin{comment}This comment is intentionally left here.
    
load "src/StarQuotientMeasures.m";
for i in [5000..5010] do
  if not IsSquarefree(i) then continue; end if;
  time result := StarQuotientHasGoodCorrespondence(i);
  print i, result;
end for;
\end{comment}
Note that genus 117 is way beyond the range where quadratic Chabauty computations are currently feasible. Hence, determining the height contributions at the primes of bad reduction is no longer a bottleneck for doing explicit quadratic Chabauty computations on $X_0(N)^*$.

The output of the algorithm in \Cref{prop:correspondence_algorithm} is such that it can be used directly by the \texttt{QCMod} package \cite{QCMod} to perform explicit quadratic Chabauty computations on quotients of modular curves. Hence, as an application of this, we are able to explicitly determine the rational points on $X_0(N)^*$ for three levels $N$ where this was previously open.
\begin{proposition}\label{prop:examples}
Let $N=187,247$ or $319$. Then the $\Q$-rational points on $X_0(N)^*$ are cusps or CM-points.
\end{proposition}

For a more detailed description of the rational CM points on these curves and their CM discriminants, see \Cref{sec:examples}.

\begin{comment}This comment is intentionally left here.
load "src/genera_and_bounds.m";
maxg := 10;
maxLevel := 10^6;
counts := [0 : i in [0..maxg]];   // counts[g+1] = #{N < maxLevel : genus(X0(N)*) = g}
maxN   := [0 : i in [0..maxg]];   // maxN[g+1]   = largest N < maxLevel with genus g (0 if none)
for N in [1..maxLevel] do
    if not IsSquarefree(N) then continue; end if;
    g := GenusX0Nstar(N);
    if g le maxg then
        counts[g+1] +:= 1;
        maxN[g+1] := N;           // N is increasing, so the last hit is the max
    end if;
end for;
counts;
maxN;

/* output
> counts;
[ 44, 38, 39, 31, 36, 39, 27, 33, 32, 32, 32 ]
> maxN;
[ 119, 238, 390, 510, 570, 910, 858, 897, 1190, 1365, 1326 ]
*/
\end{comment}

Using the techniques of this paper we expect that the computation of the rational points on $X_0(N)^*$ can be extended to many more values of $N$ aside from the three mentioned above. 
Recently, Steffen M\"uller announced joint work with Jennifer Balakrishnan and Jan Vonk determining the rational points on the genus 8 curve $X_{\mathrm{ns}}^+(19)$ at the ChaBONNty conference (MPIM Bonn, June 2026), showing that in principle it is possible to use \texttt{QCMod} to determine the rational points on genus $\leq 8$ curves. So it should now be feasible to use \texttt{QCMod} to compute the rational points on $X_0(N)^*$ for the hundreds of squarefree levels $N$ such that $X_0(N)^*$ has genus $\leq 8$. However, at the time of writing it is not possible to use \texttt{QCMod} in an automated way. For example, using \texttt{QCMod} requires as input a planar model of the curve suitable for use with Balakrishnan and Tuitman's algorithm from \cite{BalTui}. At the moment, running quadratic Chabauty for a single $X_0(N)^*$ requires a lot of manual labour and tweaking just to get a suitable model. 

By \Cref{lem:local_height_vanishes_long_edges}, the $p$-adic height attached to a correspondence is supported at $p$ if the correspondence satisfies one linear condition for each \emph{long loop}, i.e.\ each loop of length $>1$, in the dual graph of $X_0(N)^*$ at each prime $\ell\mid N$. So \Cref{thm:main1} comes down to bounding the total number of long loops from above and the genus of $X_0(N)^*$ from below. By \Cref{thm:semi-stable_star}, long loops come from supersingular points of $X_0(N)_{\Fbar_\ell}$ that have extra automorphisms or are fixed by some Atkin--Lehner involution $w_d$. This leads us to study $\nu_\ell(N,d)$, the number of supersingular points of $X_0(N)_{\overline{\F}_\ell}$ fixed by $w_d$.

Write $N=\ell M$. By Deuring's correspondence, the supersingular points of $X_0(N)_{\Fbar_\ell}$ correspond to the ideal classes of a fixed Eichler order of level $M$ in $B_{\ell,\infty}$, the quaternion algebra over $\Q$ ramified at $\ell$ and $\infty$, and their endomorphism rings are the corresponding left orders. We show that each $w_d$-fixed supersingular point corresponds to exactly two conjugacy classes of embeddings of $\Z[\sqrt{-d}]$ into these left orders; for $d=2$ one also has to allow embeddings of $\Z[i]$ (\Cref{thm:2to1}). Eichler's formula for optimal embeddings then gives closed formulae for $\nu_\ell(N,d)$, in terms of the class numbers and Kronecker symbols.

We observe and explain a close similarity between our formulae and Kluit's formulae \cite{Klu77} for the number of $w_d$-fixed points of $X_0(N)$ in characteristic $0$ -- these formulae can be interpreted as counting the number of embeddings of $\Z[\sqrt{-d}]$ into Eichler orders in different quaternion algebras (for more details, see \Cref{rem:eichler_analogy}).
\subsection{Structure of the proof of \Cref{thm:main1.1}}
Let $X$ be a nice\footnote{smooth, projective, geometrically integral} curve over $\Q$, $b \in X(\Q)$ a choice of basepoint and let $\rho(X)$ denote the rank of the Néron--Severi group of the Jacobian $J(X)$ of $X$. Then there will be $\rho(X)-1$ linearly independent $p$-adic heights on $X$. Now let $\mathcal X$ be the minimal regular model over $\Z$ of $X$ and $\ell$ be a prime of bad reduction. Then the height contribution at $\ell$ of a point $x \in X(\Q) = \mathcal X(\Z)$ depends only on the component of $\mathcal X_{\F_\ell}$ that $x_{\F_\ell}$ lies on (\Cref{thm:BDMTV_localheight}). Let $C_\ell$ denote the number of components of $\mathcal X_{\F_\ell}$.
\begin{lemma}\label{lem:intro_heights_count}
    There exist at least $\rho(X)-1 - \sum_{\ell}(C_\ell-1)$ linearly independent $p$-adic heights on $X$ supported at $p$.
\end{lemma}
\begin{proof}
The height contribution on the component containing $b_{\F_\ell}$ is zero by definition, and every other component forces a linear condition on the height in order for the height contribution at $\ell$ to vanish on that component. In particular, the claim follows.
\end{proof}
Our strategy to prove \Cref{thm:main1.1} is to bound $\rho(X)-1 - \sum_{\ell}(C_\ell-1)$ from below in the case $X=X_0(N)^*$.
\begin{remark}\label{rem:fewer_conditions}
    Actually, different components over $\F_\ell$ of the minimal regular model of $X_0(N)^*$ might impose the same linear condition on the height function. This happens often in practice. Indeed, the dual graph of $X_0(N)^*$ is a single central vertex with several loops attached to it. And all components lying on the same loop impose the same linear condition (\Cref{lem:quadratic_along_edge}). In our actual proof, we replace $C_\ell-1$ by a slightly smaller integer $L_\ell$, the number of loops of the dual graph containing at least one component of the minimal regular model, i.e.\ the number of all loops of length $>1$. However, this is just an optimisation to get slightly better bounds, and not something that is essential to the core idea of the argument.
\end{remark}

Note that $\rho(X_0(N)^*)=g(X_0(N)^*)$ (\Cref{lem:hecke_is_End}), so we need a good lower bound on the genus $g(X_0(N)^*)$. This lower bound is obtained in \Cref{subsec:genusBounds} by applying the Riemann--Hurwitz formula to $X_0(N) \to X_0(N)^*$ and using the formulae for the number of fixed points of Atkin--Lehner operators on $X_0(N)(\C)$ in terms of class numbers.
We also need to give good upper bounds on $C_\ell$ (actually, $L_\ell$ from \Cref{rem:fewer_conditions}). This upper bound is \Cref{thm:Bpnw} and is obtained using our description of the minimal regular model of $X_0(N)^*$ over $\Z_{\ell}$ in \Cref{sec:semi-stable}. The main idea is that this quantity can be bounded in terms of the number of supersingular fixed points of Atkin--Lehner operators plus an asymptotically negligible contribution from the points with $j=0,1728$.
Our formulae for these supersingular fixed-point counts are again in terms of class numbers. Using upper bounds for class numbers in both cases, we show that the genus of $X_0(N)^*$ grows faster than $\sum_{\ell\mid N}L_\ell(N)$, which gives \Cref{thm:main1.1}; together with a finite computation, this also gives \Cref{thm:main1}.
%In \Cref{sec:fixedpoint_formulas} we derive formulae for the number of fixed points of Atkin--Lehner operators acting on the supersingular points of $X_0(N)(\overline{\F}_\ell)$ in terms of class numbers. These formulae are very similar to the formulae for fixed points over $\C$.
%In subsection \ref{subsec:longCycleBound} we use the class number formula together with bounds on special values of $L$-functions to bound $\sum_{\ell}(C_\ell-1)$.

\subsection{Overview of the paper}

In \Cref{sec:background}, we recall the moduli interpretation of $X_0(N)^*$, dual graphs of semi-stable models and corresponding metric graphs, the description of local heights away from $p$ from \cite{BDMTV2}, and the facts on quaternion algebras we need.
In \Cref{sec:semi-stable} we describe the semi-stable model over $\Z_\ell$ of every Atkin--Lehner quotient $X_0(N)/W$ with $\ell\parallel N$, together with its dual graph (\Cref{thm:semi-stable_star,cor:dual_graph_quotient}). For $X_0(N)^*$ this model is stable, which proves \Cref{thm:main2}. We also show how to compute the action of Hecke operators on the homology of these dual graphs (\Cref{cor:computing_Tq}) from their action on the dual graph of $X_0(N)$, which is given by Brandt matrices \cite{Ribet1990,KohelHecke}.
In \Cref{sec:criterion} we show that the local height at $\ell$ vanishes on $X_0(N)^*(\Q_\ell)$ once the correspondence satisfies one linear condition per long loop.
In \Cref{sec:fixedpoint_formulas}, we relate supersingular $w_d$-fixed points to embeddings of quadratic orders into Eichler orders in $B_{\ell,\infty}$ and prove the formulae for $\nu_\ell(N,d)$.
In \Cref{sec:bounds} we bound the number of long loops from above and the genus of $X_0(N)^*$ from below. \Cref{sec:existence} combines these bounds with a finite computation to prove \Cref{thm:main1}.
Finally, \Cref{sec:examples} works out dual graphs in examples (including $X_0(290)^*$ at $\ell=2$) and determines $X_0(N)^*(\Q)$ for $N=187,247,319$, proving \Cref{prop:examples}.

\subsection{Future work}
As already mentioned above, one of the main obstacles to extending the determination of the rational points on $X_0(N)^*$ to multiple levels $N$ is finding suitable planar models of $X_0(N)^*$ so that they can be used with \texttt{QCMod}. Another way this could potentially be solved is by using the geometric quadratic Chabauty method of \cite{EdixhovenLido} which does not require the computation of $p$-adic Coleman integrals. However, a significant amount of work is still needed to make that method explicit enough in order to use it in an automated way on many modular curves. Note that \Cref{thm:main1} will also be useful in the geometric quadratic Chabauty setting, as the local heights can be related to a quantity that one needs to be able to compute in order to carry out geometric quadratic Chabauty \cite[Lemma 4.13]{DHS23}.

In principle, the methods of this paper can be used not only to construct $p$-adic heights where the height contributions at the primes of bad reduction vanish, but can also be used to explicitly compute the height contributions in the case they do not vanish. However, for explicit computations there are still some challenges to be solved in order to do so. For example, given a point $x \in X_0(N)^*(\Q)$ by coordinates satisfying explicit algebraic equations, how does one determine whether $x_{\F_\ell}$ is supersingular, and, if it is, how does one determine the quaternionic ideal corresponding to $x_{\F_\ell}$ under the Deuring correspondence?

\subsection{Relation to other work}
After our presentations at the ChaBONNty conference in June 2026, an attendee informed us that Beneish and Wen were also working on finding models of $X_0(N)^*$ suitable for use in computing local heights. Their manuscript \cite{BeneishWen2026} appeared on arXiv on the same day as ours and their results were obtained independently of us.

Their work overlaps substantially with \Cref{sec:semi-stable,sec:fixedpoint_formulas}: for squarefree $N$ and a prime $\ell\geq 5$ dividing $N$, they also obtain semi-stable models of all Atkin--Lehner quotients $X_0(N)/W$ by quotienting the Deligne--Rapoport model and describe their special fibres (deducing genus formulae from this), and they count $w_d$-fixed supersingular points via optimal embeddings into maximal orders of $B_{\ell,\infty}$, with formulae that agree with ours for $\nu_\ell(N,d)$ wherever both apply. 
However, there are also many results in this paper that do not appear in the work of Beneish and Wen. For example, we determine the thickness of every double point, and hence the minimal regular model and the metric dual graph, as well as the Hecke action on its homology (\Cref{thm:semi-stable_star,cor:minimal_regular_AL_quot,cor:computing_Tq}), which is precisely what is needed for the local heights at $\ell$; we also allow $\ell\in\{2,3\}$ and non-squarefree $N$ with $\ell\parallel N$, and our formulae for $\nu_\ell(N,d)$ hold without the restrictions imposed in \cite{BeneishWen2026} (e.g.\ $\ell,d\geq 5$, and $d\not\equiv 3\pmod 4$ when $N$ is even). Additionally, all the results related to $p$-adic heights supported at $p$ are unique to this paper.

Our description of the local heights away from $p$ is based on \cite[Theorem 3.2]{BDMTV2} and \cite[Theorem 3.11]{hyperellipticQC}, which build on \cite[Corollary 12.1.3]{BettsDogra} and express $h_\ell$ in terms of the dual graph of a semi-stable model at $\ell$ (see \Cref{thm:BDMTV_localheight}). In the quadratic Chabauty computations for Atkin--Lehner quotients of modular curves in \cite{BDMTV2,AABCCKW,ANTStar}, these contributions were determined curve by curve (based on the explicit algebraic equations). For hyperelliptic curves, Betts, Duque-Rosero, Hashimoto and Spelier \cite{hyperellipticQC} give a method to compute the local heights at the primes of bad reduction starting from an explicit equation of the curve. Our approach, in contrast, is model-free: the dual graph of $X_0(N)^*$ at $\ell\mid N$ and the action of the Hecke operators on its homology are obtained from supersingular points and Brandt matrices (\Cref{thm:Deuring,cor:dual_graph_quotient,cor:computing_Tq}), and no equations for $X_0(N)^*$ are needed.  It would be interesting to combine this with the model-free approach to the local height at $p$ of Rendell \cite{Isabel}.

For the stable reduction of $X_0(N)$ for the case $\ell^2\parallel N$, see Edixhoven \cite{Edixhoven}. Xue \cite{xue2009minimal} constructs, over $\Z[1/6]$, the minimal resolution of the quotient of the Deligne--Rapoport model by a single Atkin--Lehner involution. In \Cref{sec:semi-stable} we do this for an arbitrary Atkin--Lehner quotient $X_0(N)/W$ and $\ell\parallel N$, and describe the dual graph of its semi-stable model and its minimal regular model (\Cref{thm:semi-stable_star,cor:dual_graph_quotient,cor:minimal_regular_AL_quot}).

\subsection{Acknowledgements}
We thank Jennifer Balakrishnan, Netan Dogra and Steffen M\"uller for useful discussions on quadratic Chabauty in general, and also on how to use \texttt{QCMod}. We thank Sarah Arpin, Freddy Saia and John Voight for the useful discussion regarding the supersingular points, quaternion algebras and Brandt modules. We thank Sachi Hashimoto for explaining how to use her publicly available code to confirm the classification of all rational points found.  We thank Amnon Besser and Oana Padurariu for the discussions in an early stage of this project.

N.~A.~acknowledges support from the institutional project of the University of Zagreb Faculty of Civil Engineering – [GEM
– Graphs, elliptic and modular curves], financed through
the National Recovery and Resilience Plan 2021–2026. The project is funded by the European Union – NextGenerationEU. 

M.~D.~was supported by the Croatian Science Foundation under the project number HRZZ IP-2022-10-5008. N.~A.~and M.~D.~were supported by the project ``Implementation of cutting-edge research and its application as part of the Scientific Center of Excellence for Quantum and Complex Systems, and Representations of Lie Algebras'', Grant No.~PK.1.1.10.0004, co-financed by the European Union through the European Regional Development Fund - Competitiveness and Cohesion Programme 2021--2027. 

T.~K.~was partially supported by the 2021 MSCA Postdoctoral Fellowship 01064790 -- Ex\-pli\-cit\-Rat\-Points. 

\subsection{AI usage disclosure}
The core ideas, proof structure, and crucial computer code were developed by the authors in 2025 under the initial assumption that $N$ is coprime to $6$. During a second phase in 2026, significant effort was dedicated to removing this assumption and refining the details. During this second phase, AI tools (Large Language Models) were used primarily to assist in locating relevant references and for language editing, as well as in drafting code and proofs for auxiliary results that appear to be known in the literature, despite our inability to find a proper reference. Figures were created using AI tools based on initial hand-drawn sketches. After using these tools, the authors carefully reviewed, checked, and edited all content and code as needed and take full responsibility for the content of this manuscript.

\subsection{Data availability and reproducibility}
The computations for this article were performed using Magma (version V2.30-1)
\cite{Magma}. Only some of the computation can be reproduced in V2.29-9 or earlier due to a bug in Magma's Embed function fixed in V2.29-10. The code accompanying this paper is available at

\begin{center}
\url{https://github.com/nt-lib/local-heights-X0Nstar/tree/v1.0.0} 
\end{center}

under the GNU General Public License v3.0. All filenames given in the paper
refer to files in this repository. The quadratic Chabauty computations use the Magma
packages \cite[commit \texttt{c939520}]{QCMod} based on the work in \cite{BDMTV2}, and the model finding code of \cite[commit \texttt{b4612c1}]{ModelFinder} based on
\cite{AABCCKW}, both of
which are included as git submodules, as is the code of \cite[commit \texttt{faa2af2}]{ShimuraCurveALQuotients} based on \cite{SachiEran2026} which was used to confirm that the rational points found were CM points. We also rely on precomputed class number data downloaded from the LMFDB
\cite{LMFDB}; these data were originally computed in \cite{classGroupTabulation}.

The computations have been run on a server at the University of Zagreb with an AMD
EPYC 9175F CPU with 16 cores (32 threads) and 384 GiB of RAM running Ubuntu 24.04.4 LTS.
Run sequentially, all computations together take about 23 hours; run in parallel, they
finish in about 12 hours. No single computation used more than 2.4 GB of RAM, and
running them in parallel used at most 7 GB of RAM in total.

Most of the computation time was spent on verifying the existence of $p$-adic heights supported at $p$ below our theoretic upper bound in \Cref{thm:existence_of_supported_at_p_heights} (about 12 hours, at most 2 GB of RAM). After that the most time was spent on the explicit quadratic
Chabauty computations for $X_0(187)^*$, $X_0(247)^*$ and $X_0(319)^*$ (about 8.5 hours in
total, at most 2.4 GB of RAM for a single computation). All other computations took
about 2.5 hours in total and used less than 500 MB of RAM each.

Logs containing the output of the computations are available in the \ghdir{logs}
subfolder. More information on how to reproduce these results can be found in the
\ghfile{README.md} file.

\section{Background}\label{sec:background}

\subsection{Motivation: rational points on \(X_0(p)^+\) and \( X_0(N)^*\) and their moduli interpretation}

The modular curve $X_0(p)^+ \colonequals X_0(p)/w_p$ and, more generally, the star quotient $X_0(N)^* \colonequals X_0(N)/W_N$ by the full Atkin--Lehner group $W_N$ inherit a moduli interpretation from $X_0(N)$: a point corresponds to a pair $(E,G)$ consisting of an elliptic curve $E$ together with a cyclic subgroup $G\subseteq E$ of order $N$, taken up to the action of $W_N$. A rational point of $X_0(N)^*$ is then a $\mathrm{Gal}(\overline{\Q}/\Q)$-stable such orbit.  Any non-cuspidal rational point on $X_0(N)^*$ hence corresponds to a $\Q$-curve~\cite{Elkies}, i.e.\ an elliptic curve isogenous to all its Galois conjugates. By~\cite{Ribet2004} and the now-proven Serre's Modularity Conjecture \cite{KhareWintenberger1,KhareWintenberger2}, all non-cuspidal rational points on $X_0(N)^*$ correspond to modular elliptic curves. 

The curve $X_0(N)^*$ always has rational points coming from cusps, and for infinitely many values of $N$ it also has rational points coming from CM elliptic curves. Following Elkies \cite{Elkies} we call the other points \emph{exceptional}. \emph{Elkies' conjecture} predicts that for every $N \gg 0$, $X_0(N)^*$ has no exceptional $\Q$-points. Let $J_0(N)^*$ denote the Jacobian of $X_0(N)^*$. If $N$ is squarefree then, assuming BSD, one has that every simple factor of $J_0(N)^*$ has Mordell--Weil rank at least its dimension. So there is no hope of trying to prove Elkies' conjecture using classical Chabauty or Mazur's formal immersion methods. One strategy for trying to prove Elkies' conjecture is using the quadratic Chabauty method. This has been carried out for several values of $N$ in a series of papers \cites{BDMTV2,AABCCKW,ANTStar}. Carrying out quadratic Chabauty requires computing the height contributions at the primes of bad reduction of a certain $p$-adic height function. A bottleneck for extending the progress on Elkies' conjecture to other values of $N$ has been computing the height contributions. The main result \Cref{thm:main1} states that one can make these height contributions vanish for squarefree $N > 714$. For $N \leq 714$ one can often still make the height contributions vanish. We explicitly determine the list of $N$ for which this is possible in \Cref{thm:existence_of_supported_at_p_heights}. Furthermore, our strategy for making these height contributions vanish is really explicit. As an application of our strategy to make these height contributions vanish we are able to determine $X_0(N)^*(\Q)$ for several example values of $N$ in~\Cref{sec:examples}. CM theory lets us certify each rational point we compute as CM or cuspidal: we simply confirm that there are enough rational CM points. 

%For a list of star quotients whose rational points we already know see~\cites{BDMTV2,AABCCKW,ANTStar}. The provable determination of rational points on hyperelliptic $X_0(N)^*$ is completed in \cite{ANTStar} using a combination of Chabauty-style techniques together with the Mordell--Weil sieve. The authors also give a classification (moduli interpretation) of the $\Q$-points on hyperelliptic $X_0(N)^*$ for squarefree levels $N$.

Recent work and work-in-progress of Hashimoto--Keller--Le Fourn~\cite{LFHK} makes significant progress on the non-squarefree levels, which is also one of the motivations to stick to the squarefree case in this paper. The case of non-squarefree $N$ is complementary: in most of the cases, \emph{loc.\ cit.}\ construct a rank $0$ quotient of the Jacobian (and use this to give a bound on the denominator of the $j$-invariant of $\Q$-points). However, as we mentioned before, assuming BSD, such rank 0 quotients do not exist when $N$ is squarefree, causing their methods to not be applicable, nor would the next step (work-in-progress by the authors of \emph{loc.\ cit.}), Runge's method, because for squarefree $N$, $X_0(N)^*$ does not have enough cusps, namely exactly one. On the other hand, the description of a semi-stable model is easier in the squarefree case.

Note that Elkies' conjecture for all squarefree $N \gg 0$ implies it for all $N\cdot N'^2$ ``quadratically above $N$''. This is because there are degeneracy maps $X_0(N\ell^2)^* \to X_0(N)^*$ that are \emph{modular}, i.e., they preserve the property of a point being exceptional, CM or a cusp (see~\cite{LFHK}). It is important here that the levels differ by $\ell^2$, because the degeneracy maps $X_0(N\ell) \to X_0(N)$ are not compatible with the action of Atkin--Lehner at $\ell$, and hence do not induce a map between their quotients $X_0(N\ell)^*$ and $X_0(N)^*$. If $X_0(NN'^2)^*$ has an exceptional $\Q$-point, its image in $X_0(N)^*(\Q)$ is exceptional as well, which Elkies' conjecture for $N \gg 0$ squarefree rules out.

\subsection{Reduction of algebraic curves}\label{subsec:reduction}

We first recall the semi-stable models and their dual graphs.

\begin{definition}[thickness]\label{def:thickness}
Let $R$ be a DVR with closed point $s \in \Spec R$. Let $X \to \Spec R$ be a semi-stable curve, and let
$x \in X_s$ be an ordinary double point.
Étale locally at $x$ (possibly after passing to a finite étale extension of
$\Spec R$), the completed local ring has the form
\[
\widehat{\mathcal{O}}_{X,x} \cong
\widehat{\mathcal{O}}_{R,s}\llbracket u,v \rrbracket/(uv-c),
\]
%Z might need to be replaced
for some $c \in M_s$ (maximal ideal of the local ring).
The \emph{thickness} $i_x$ of $x$ is defined as
$i_x \colonequals v_s(c)$ with respect to the normalised valuation of
$\mathcal{O}_{R,s}$.
\end{definition}
This definition is independent of all choices, see \cite[Corollary~10.3.22]{Liu2002}.

This notion of thickness is useful as it allows one to easily describe the minimal desingularisation of a semi-stable curve.

\begin{corollary}[regular model of semi-stable curve]\label{cor:regular_from_semistable_model}
Let $R$ be a DVR with closed point $s$.    Let $X \to \Spec R$ be a semi-stable projective curve 
with smooth generic fibre $X_\eta$. Let $\pi\colon X' \to X$ be the minimal desingularisation, 
$x \in X_s$ a split ordinary double point, of thickness $i$ in $X$. Then $\pi^{-1}(x)$ is 
made up of a chain of $i-1$ projective lines over $k(s)$ that meet transversally at rational 
points. These lines are of multiplicity $1$ in $X_s'$, and have self-intersection $-2$ in $X'$. 
\end{corollary}
\Cref{fig:split-node-dising} contains a picture of how this desingularisation looks for $X_0(N)^*$.
\begin{proof}
See~\cite[Corollary 10.3.25]{Liu2002}.
\end{proof}

The following is something that is well known to the experts, and is just a small variation of \cite[Theorem 10.3.34(a)]{Liu2002}, but we were unable to find a clean reference for it.

\begin{lemma}[minimal regular model of a stable curve]\label{lem:minimal_regular_from_stable_model}
    Let $\mathcal X$ be a stable model over a discrete valuation ring $R$. Let $K$ be the field of fractions of $R$ and $X := \mathcal X_K$. Let $\mathcal X'$ be the minimal desingularisation of $\mathcal X$. Then $\mathcal X'$ is the minimal regular model of $X$, and in particular $\mathcal X'$ is semi-stable.
\end{lemma}
\begin{proof}
The proof follows the proof in \cite[Theorem 10.3.34(a)]{Liu2002}. In that proof the minimal regular model is obtained from a semi-stable model by first passing to the minimal desingularisation and then contracting all chains of $\P^1$'s that meet the other irreducible components in just a single point. However, a stable model does not have any chain of $\P^1$'s that meet the other components in a single point. And furthermore in the minimal desingularisation of a stable model, there are no chains of $\P^1$'s that meet the other components in just a single point. Indeed, by \Cref{cor:regular_from_semistable_model}, the only chains of $\P^1$'s that are introduced in this desingularisation process are chains of $\P^1$'s that connect two points. This implies that $\mathcal X'$ is the minimal regular model. That $\mathcal X'$ is semi-stable is the content of  \cite[Theorem 10.3.34(a)]{Liu2002}.
\end{proof}

%The components and double points of a semi-stable model are conveniently packaged into a metric graph, which is the object the local heights of the next subsection will be computed on.

\begin{definition}[dual graph]\label{def:reduction_graph}
Let $X$ be a nice curve over $\Q_\ell$ of genus $\geq 2$. Let $K/\Q_\ell$ be a finite extension of ramification index $r$ over which $X/\Q_\ell$ acquires semi-stable reduction. Let $\mathcal X/\orderO_K$ be a semi-stable model of $X$, let $\bar s: \Spec \Fbar_\ell \to \Spec \orderO_K$ be a geometric point and $\mathcal X_{\bar s}$ denote its corresponding (geometric) special fibre. Write $\mathcal X_{\bar s}=\bigcup_w X_w$ as a union of its irreducible components. The \defn{dual graph} (or \defn{reduction graph}) $\graph$ of $\mathcal X$ has one vertex $w\in V(\graph)$ for each irreducible component $X_w$ of $\mathcal X_{\bar s}$ and one edge $e\in E(\graph)$ for each ordinary double point, joining the components on which its two branches lie (so $e$ is a loop if both branches lie on the same component). An ordinary double point of thickness $i_e$ (\Cref{def:thickness}) is assigned \defn{length} $l(e)=i_e/r$. We write $\graph_\Q$ for the set of points of the resulting metric graph lying at rational distance from a vertex (the points with coordinate \(x_e \in \Q \cap [0, l(e)]\) on some edge $e$), and similarly $\graph_{\frac 1r \Z}$ for those points that lie at a distance in $\frac 1r \Z$ from a vertex.
\end{definition}

Note that when $X$ is already semi-stable over $\Q_\ell$ (so $K=\Q_\ell$ and $r=1$) the length of each edge is simply the thickness of the corresponding ordinary double point.

Let $\mathcal X'$ be the minimal desingularisation of $\mathcal X$. Then $\mathcal X'$ is semi-stable again. And furthermore $\mathcal X'_{\bar s}$ is obtained from  $\mathcal X_{\bar s}$ by replacing every ordinary double point by a chain of $\P^1$'s of length $i_e-1$ (see \Cref{cor:regular_from_semistable_model}). On the level of dual graphs this corresponds to subdividing an edge $e$ into $i_e$ segments of length $1/r$. This changes neither $\graph_\Q$ nor any of the objects below, see \Cref{rem:measure_subdivision}. Additionally, the dual graph doesn't depend on the choice of the field $K$ in the following sense: If $K'$ is an extension of $K$ with ramification index $r'$ over $\Q_\ell$, then $\mathcal X_{\OO_{K'}}$ is semi-stable again with the same set of ordinary double points, except that a double point that used to have thickness $i_e$ now has thickness $i_e r'/r$, showing that $l(e)$ remains the same.

\begin{definition}\label{def:reduction_map}
Let $X$ be a curve over $\Q_\ell$ that acquires semi-stable reduction over a field extension $K$ of $\Q_\ell$ of ramification index $r$  . The \defn{reduction map} $\mathrm{red}\colon X(\Q_\ell)\to \graph_{\frac 1r \Z}$ is defined to be the map that sends a point to the component of $\mathcal X'_s$ to which it reduces. 
\end{definition}
Note $X(\Q_\ell) \subseteq X(K) =\mathcal X'(\OO_K)$. Since $\mathcal X'$ is regular, one has that $\OO_K$-points of $\mathcal X'$ and in particular $\Q_\ell$-points of $X$ specialise to smooth points in the special fibre. In particular, each $\Q_\ell$-point reduces to a unique component of $\mathcal X'_s$, which is what is needed to make the above definition of the reduction map well defined.

%We now recall the precise statement. The dual graph $\graph$ of a semi-stable model at $\ell$, its metric realisation $\graph_\Q$, the reduction map and the cycle space $\mathrm{H}_1(\graph,\Q)$ were introduced in \Cref{def:reduction_graph,def:reduction_map,def:cycle_space}. Further notions are needed and in defining those, we refer the reader to \cite{BakerFaber2006} for more details.
\subsection{Metric graphs}
In this short subsection, $\graph$ is a metric graph (a graph with positive lengths assigned to its edges), for instance the dual graph of \Cref{def:reduction_graph}.

\begin{definition}\label{def:cycle_space}
Orient the edges and let $\Q\,E(\graph)$ be the $\Q$-vector space on them, with the pairing $e\cdot e'=\delta_{ee'}\,l(e)$ given by the lengths. Inside it sits the \defn{cycle space} $\mathrm{H}_1(\graph,\Q)=\ker\left(\partial\colon\Q\,E(\graph)\to\Q\,V(\graph)\right)$, the first homology of $\graph$. Let $\pi\colon\Q\,E(\graph)\to\mathrm{H}_1(\graph,\Q)$ be the orthogonal projection for this pairing and $e^*\colon\Q\,E(\graph)\to\Q$ the coordinate functional dual to $e$.
\end{definition}

Each edge $e$ of $\graph$ carries an arclength coordinate $x_e$, running from $0$ at the source $s(e)$ to $l(e)$ at the target $t(e)$; we write $g_e$ for the restriction of a function $g$ to $e$.

\begin{definition}\label{def:pp_measure}
A continuous function $g\colon\graph \to\R$ is \emph{piecewise polynomial} if $g_e$ is a polynomial for every edge $e$. A \defn{piecewise polynomial measure} on $\graph_\Q$ is a formal sum
\[
  \mu=\sum_{e\in E(\graph)}\mu_e\cdot\lvert\mathrm{d}x_e\rvert+\sum_{w\in V(\graph)}\mu_w\cdot\delta_w ,
\]
in which each $\mu_e$ is a polynomial in $x_e$ and each $\mu_w$ lies in $\R$. Here $\lvert\mathrm{d}x_e\rvert$ denotes the Lebesgue measure on $e$, and $\delta_w$ the Dirac measure at $w$. The \defn{total mass} of $\mu$ is
\[
  \lvert \mu\rvert\colonequals\sum_{e\in E(\graph)}\int_0^{l(e)}\mu_e\,\mathrm{d}x_e+\sum_{w\in V(\graph)}\mu_w .
\]
\end{definition}

\begin{definition}\label{def:graph_laplacian}
Let $g\colon\graph \to \R$ be piecewise polynomial. Its \defn{Laplacian} is the piecewise polynomial measure
\[
  \nabla^2(g)\colonequals-\sum_{e\in E(\graph)}(g_e)''\cdot\lvert\mathrm{d}x_e\rvert-\sum_{w\in V(\graph)}\sigma_w(g)\cdot\delta_w,
\]
\[ \text{where } \quad \sigma_w(g)\colonequals\sum_{s(e)=w}(g_e)'(0)-\sum_{t(e)=w}(g_e)'\left(l(e)\right),\] and $(g_e)''$ is the second derivative with respect to $x_e$.
\end{definition}
\noindent Thus $\sigma_w(g)$ is the sum of the slopes of $g$ at $w$ in the directions leading ``away from'' $w$; a loop at $w$ contributes to both sums, once for each of its two ends. %This is the Laplacian of \cite[Definition~5]{Baker2006MetrizedGraphs}, with $\sigma_w(g)$ written in terms of the chosen orientations rather than of tangent directions.

\subsection{Quadratic Chabauty and local heights}

We recall some background theory on quadratic Chabauty, where we follow \cite{BalakrishnanDogra1} and \cite{BDMTV2}.

Let $X/\Q$ be a nice curve of genus $g$ with good reduction at $p$, and suppose its Jacobian $J$ has Mordell--Weil rank equal to $g$. In fact, as in \cite{BDMTV2}, we assume the slightly stronger statement that the $p$-adic logarithm $\log : J(\Q_p) \to H^0(X_{\Q_p},\Omega^1)^\vee$  induces an isomorphism $$\log: J(\Q) \otimes_\Z \Q_p \to H^0(X_{\Q_p},\Omega^1)^\vee.$$ In particular, this means that the kernel of $\log$ intersected with $J(\Q)$ is $J(\Q)_\tors$. 
Note that the rank equal to $g$ condition, or in fact $\log$ inducing an isomorphism is not a big restriction for our applications, since $N=5077$ is the only squarefree $N<10000$ for which the rank condition fails, see \cref{rem:rank-genus-condition}.
We also assume that $X(\Q)\neq \emptyset$, which is certainly true for the curves we are interested in. This allows us to choose a basepoint $b \in X(\Q)$.

Heights play a crucial role in quadratic Chabauty. And for this article we will define everything in terms of the Coleman--Gross height pairing $$h: J(\Q) \times J(\Q) \to \Q_p$$ \cite[Prop.~1.2, Prop.~5.2, \S6]{ColemanGross}, which is often also called a global $p$-adic height. This height pairing depends on 
\begin{enumerate}
    \item the choice of a continuous id\`ele class character $\chi: \A_\Q^\times/\Q^\times \to \Q_p$,
    \item a splitting of the Hodge filtration, or equivalently a subspace $W \subseteq H^1_{dR}(X_{\Q_p})$ complementary to the space of holomorphic $1$-forms, which we take to be isotropic with respect to the cup product.
\end{enumerate}

\begin{definition}[\cite{ColemanGross}]\label{def:CGheight}
    The \emph{\(p \)-adic height pairing} is a symmetric bi-additive pairing
    \(h:\Div^0(X) \times \Div^0(X) \to \Q_p\), \((D_1, D_2) \mapsto h(D_1, D_2),\)
for \(D_1, D_2\) divisors of degree \(0\) with disjoint support, such that:
\begin{enumerate}
    \item We have
    \[ h(D_1, D_2) = \sum_{v} h_v(D_1, D_2) = h_p(D_1, D_2) + \sum_{\ell\neq p} h_\ell(D_1, D_2),\]
    where \(h_p\) can be expressed via Coleman integration, the sum goes over all primes \(\ell\neq p\) (only finitely many terms are non-zero), and \(h_\ell(D_1, D_2)=m_\ell \chi_\ell(\ell)\), where \(m_\ell\in \Q\) can be defined in terms of intersection theory \cite[Equation (1.3)]{ColemanGross}.
    \item We have \(h(D, \div(f))=0\) for all \(f\in \Q(X)^*\).
\end{enumerate}
By \((ii)\), \(h\) induces a symmetric bilinear pairing \(J(\Q) \times J(\Q) \to \Q_p\).
\end{definition}

As this global height is bilinear, it vanishes on the torsion in $J(\Q)$. In particular, there exists a bilinear map $B$ on $H^0(X_{\Q_p},\Omega^1)^\vee$ that makes the diagram below commute.
\[
\begin{tikzcd}[column sep=small, row sep=large]
J(\Q)\times J(\Q)
  \arrow[rr, "{(\log,\,\log)}"]
  \arrow[dr, "h"']
& &
H^0(X_{\Q_p},\Omega^1)^\vee \times H^0(X_{\Q_p},\Omega^1)^\vee
  \arrow[dl, "B", start anchor={[xshift=1.25cm,yshift=-0.35cm]west}] \\
& |[xshift=0.3cm]| \Q_p &
\end{tikzcd}
\]

We choose a correspondence $Z \subseteq X \times X$ of trace $0$. Let $b,z \in X(K)$ where $K$ is either $\Q$, $\Q_p$ or $\Q_\ell$. Then we define the maps $i_1, i_2$ and $i_\Delta$ from $X \to X \times X$, where $i_1$ is the inclusion $\set{b} \times X \hookrightarrow X \times X$, $i_2$ is the inclusion $X \times \set{z} \hookrightarrow X \times X$ and $i_\Delta$ is the diagonal $X \hookrightarrow X \times X$. Then the degree 0 divisor $D_Z(b,z)$ is defined by 
$$D_Z(b,z) :=i_\Delta^*(Z) -i_1^*(Z)-i_2^*(Z).$$ The degree of $D_Z(b,z)$ is $Z \cdot(\Delta(X)-\set{b} \times X -X \times \set{z})$, which indeed equals 0 by the Lefschetz trace formula.
Note that $D_Z(b,z)$ is only well defined if the support of $Z$ does not contain $\set{b} \times X$, $X \times \set{z}$ nor $\Delta(X)$. However, using the moving lemma~\cite[§11.4]{Fulton} one can always replace $Z$ with a rationally equivalent cycle so that $D_Z(b,z)$ is well-defined. A correspondence $Z$ on $X$ induces an endomorphism of its Jacobian $J(X)$. On $\End \, J(X) \otimes \Q$ there is an involution, called the Rosati involution, and we will use $\End^\dagger J(X)$ to denote the subspace of $\End \, J(X) \otimes \Q$ invariant under the Rosati involution. By \cite[Proposition 17.2]{MilneAbelianVarieties}, $\End^\dagger J(X)$ is isomorphic to the Néron--Severi group $\NS(J(X))\otimes \Q$.
Quadratic Chabauty exploits the $p$-adic heights coming from trace $0$ correspondences that are fixed by the Rosati involution, or equivalently those coming from a trace $0$ element of the Néron--Severi group.

\begin{definition}\label{def:padic_height} Let $b \in X(\Q)$ be a basepoint and $Z \otimes a \in \NS(J(X)) \otimes \Q$ be a correspondence of trace $0$. Then the {\it $p$-adic height associated to $Z \otimes a$} is the function $h_{Z,b}\colon X(\Q)\to\Q_p$ given by
\begin{align*}
   h_{Z,b} :   X(\Q) &\to\Q_p \\
    z &\mapsto ah(z-b,D_Z(b,z)).
\end{align*}
\end{definition}

\noindent From now on we will just write $Z$ instead of $Z \otimes a$ when talking about correspondences with rational coefficients. When $Z,b$ are clear from the context we will often just write $h$ instead of $h_{Z,b}$. The height $h$ decomposes as a sum of local contributions $h=h_p+\sum_{\ell\neq p}h_\ell$, where for a prime $\ell$ the height $h_\ell$ is defined by
\begin{align*}
   h_\ell :   X(\Q_\ell) &\to\Q_p \\
    z &\mapsto h_\ell(z-b,D_Z(b,z)).
\end{align*}
If $\ell\neq p$ is a prime of good reduction for $X$, then $h_\ell(X(\Q_\ell)) =0$. In particular, $h_\ell=0$ for all but finitely many values of $\ell$.
With these definitions it is clear that both the global height $h_{Z,b}$ and the local heights $h_{Z,b,\ell}$ depend linearly on $Z$.

\begin{remark}\label{rem:linear_independence}The height from \Cref{def:padic_height} is in terms of the Coleman--Gross height from \Cref{def:CGheight}, and hence also depends on a character $\chi : \A_\Q^\times \to \Q_p$ and a splitting of the Hodge filtration. Changing $\chi$ only changes each local height $h_{\ell}$ by a multiplicative constant. Changing the splitting of the Hodge filtration only changes the local height at $p$. In particular, whether a $p$-adic height is supported at $p$ does not depend on these choices. Hence, after having fixed a basepoint, it is safe to think of these $p$-adic heights as being parametrised by the vector space of trace $0$ elements in $\NS(J(X)) \otimes \Q$. In particular, the linear independence statement from \Cref{thm:main1.1} should be interpreted as taking place in $(\NS(J(X)) \otimes \Q)^{\operatorname{tr}=0}$.
\end{remark}

\begin{definition}
    A trace zero correspondence $Z$ on a curve $X$ over $\Q$ is called a {\it{clean correspondence}} if $h_{Z,b}$ is a $p$-adic height supported at $p$, i.e.\ if the local contributions $h_\ell$ to $h_{Z,b}$ mentioned above are zero for all primes $\ell \neq p$.
\end{definition}

The above formula for $h_\ell$ is also used when $\ell =p$.
Using the bilinear form $B$ one can extend the global height $h$ to a function $h_B$ on all of $X(\Q_p)$ by defining
\begin{align*}
   h_B :   X(\Q_p) &\to\Q_p \\
    z &\mapsto B(\log(z-b),\log(D_Z(b,z))).
\end{align*}

By definition, one has the equality $h=h_B$ on $X(\Q)$. Now $h_B - h_p$ is a locally analytic function on $X(\Q_p)$. Furthermore, on $X(\Q)$ we have the equality $h_B - h_p = h-h_p =\sum_{p \neq \ell} h_\ell$. Define $\Upsilon_\ell := h_\ell(X(\Q_\ell))$, then each $\Upsilon_\ell$ is a finite set by \cite[Corollary 0.2]{KimTamagawa}. If $\ell$ is a prime of good reduction for $X$ then $\Upsilon_\ell = \set 0$, and hence $\Upsilon_\ell$ is trivial for all but finitely many values of $\ell$. In particular $\Upsilon := \sum_{\ell \neq p}\Upsilon_\ell$ is a well defined finite set.

\begin{definition}
The function $\rho_{Z,b} := h_B - h_p :X(\Q_p) \to \Q_p$ is called the \emph{quadratic Chabauty function} associated to the correspondence $Z$ and basepoint $b$.
\end{definition}

\begin{theorem}\cite[Thm 1.2]{BalakrishnanDogra1}
Assume that $X$ is a nice curve over $\Q$ whose Mordell-Weil rank equals its genus, $Z$ a non-zero trace zero correspondence on $X$, $b \in X(\Q)$ a basepoint and $p$ a prime of good reduction for $X$. Then one has that $\rho_{Z,b}(X(\Q)) \subseteq \Upsilon$. Furthermore $\rho_{Z,b}^{-1}(\Upsilon)$ is a finite subset of $X(\Q_p)$.
\end{theorem}

The proof that $\rho_{Z,b}(X(\Q)) \subseteq \Upsilon$ has already been sketched above. The deep part of the theorem is actually that $\rho_{Z,b}^{-1}(\Upsilon)$ is a finite set. And furthermore this finite set is also so explicit that one can hope to compute it, or at least approximate it $p$-adically.

%\noindent The term $h_p$ is a locally analytic function of the point, computed by $p$-adic (Coleman) integration; the finitely many terms $h_\ell$ for $\ell\neq p$ are locally constant. The quadratic Chabauty function $\rho \colonequals h-h_p$ therefore takes values in the finite set $\sum_{\ell\neq p}h_\ell\left(X(\Q_\ell)\right)$, and combining this with the Chabauty conditions cuts out a finite, explicitly computable superset of $X(\Q)$. The input we need is the shape of the bad-prime terms. 

%By \cite[Theorem~3.2]{BDMTV2}, for a prime $\ell\neq p$ of bad reduction the local height $h_\ell\colon X(\Q_\ell)\to\Q_p$ depends only on the component of the special fibre the point reduces to. In terms of the dual graph of a regular semi-stable model of $X$ at $\ell$, it depends only on the vertex to which the point reduces.

Let $V_p(X)$ denote the $p$-adic Tate module of the Jacobian of $X$. As mentioned in \cite{BDMTV2} after the statement of Theorem 3.2 one has that $V_p(X)$ admits a filtration $$V_p(X) = W_0 V_p(X) \supseteq W_1 V_p(X) \supseteq W_2 V_p(X) \supseteq W_3 V_p(X) =0$$ whose successive quotients are $\mathrm{H}_1(\graph,\Q)\otimes \Q_p$, $\bigoplus_{w\in V(\graph)}V_p(X_w)$ and $\mathrm{H}_1(\graph,\Q)^\vee \otimes \Q_p(1)$. The action of any correspondence $Z$ (more generally, of any endomorphism of the Jacobian) on $V_p$ preserves this filtration. This shows the following is well defined.
\begin{definition}[the endomorphism induced by a correspondence]\label{def:corr_endomorphism}
Let $Z$ be a correspondence. Then the endomorphisms $F_Z$ of $\mathrm{H}_1(\graph,\Q)$ and $\mathrm{H}_1(\graph,\Q)\otimes \Q_p$, $\bigoplus_{w\in V(\graph)}V_p(X_w)$ are defined to be those induced by the action of $Z$ on the graded pieces of $V_p(X)$.
\end{definition}

Note that $F_Z$ does not necessarily preserve the summands in $W_1 V_p(X)/W_2 V_p(X) \cong \bigoplus_{w\in V(\graph)}V_p(X_w)$. However, one can still use the composition $$V_p(X_w) \to \bigoplus_{w\in V(\graph)}V_p(X_w) \stackrel{F_Z}{\to} \bigoplus_{w\in V(\graph)}V_p(X_w) \to V_p(X_w)$$ in order to get a well defined action of $F_Z$ on $V_p(X_w)$.

\begin{definition}\label{def:corr_measure}
The \defn{measure} of the endomorphism $F := F_Z$ of \Cref{def:corr_endomorphism} is
\[
  \mu_F=\sum_{e\in E(\graph)}\frac{1}{l(e)}\,e^*F\left(\pi(e)\right)\cdot \lvert\mathrm{d}x_e\rvert+\tfrac12\sum_{w\in V(\graph)}\operatorname{Tr}\left(F\mid V_p(X_w)\right)\cdot \delta_w ,
\]
with $\pi$ and $e^*$ as in \Cref{def:cycle_space}.
\end{definition}

\begin{remark}\label{rem:measure_subdivision}
The measure $\mu_F$ is invariant under subdivision of the underlying metric graph: if $\graph'$ is obtained from $\graph$ by subdividing an edge $e$ into edges $e_1,\dots,e_n$ of lengths $l_1,\dots,l_n$, the new vertices carrying rational components, then $\mu_{F'}=\mu_F$ under the identification $\graph'_\Q=\graph_\Q$, see \Cref{lem:measure_subdivision} in the Appendix.
Moreover, \Cref{lem:laplacian_mass,lem:measure_mass} say that the equation $\nabla^2(g)=\mu_F$ of \Cref{thm:BDMTV_localheight} below is solvable \emph{precisely} when $F$ has trace $0$ on $\mathrm H^1_{\mathrm{dR}}(X_{\Q_\ell})$, and that its solution is then unique up to a constant, which the normalisation at $\mathrm{red}(b)$ removes.
\end{remark}

\begin{theorem}[local heights away from $p$, {\cite[Theorem~3.2]{BDMTV2}, \cite[Theorem~3.11]{hyperellipticQC}}]\label{thm:BDMTV_localheight}
Keep the notation of \Cref{def:reduction_graph,def:cycle_space,def:pp_measure,def:graph_laplacian,def:corr_measure} and fix a basepoint $b\in X(\Q_\ell)$. Then the local height factors through reduction to the graph,
\[
  h_\ell=\kappa\cdot j_\graph\circ\mathrm{red}\colon X(\Q_\ell)\longrightarrow\Q_p,
\]
for a non-zero constant $\kappa \in \Q_p$ independent of the point, where $j_\graph\colon\graph\to\R$ is the unique piecewise-quadratic function determined by
\[
  j_\graph\left(\mathrm{red}(b)\right)=0,\qquad \nabla^2(j_\graph)=\mu_F .
\]
In particular $h_\ell$ is locally constant, its value depends only on the component to which the point reduces, and it attains only finitely many values.
\end{theorem}

\begin{remark}
Note that the function $j_\graph$ is a priori only a real-valued function and there is no well defined way to multiply elements of $\R$ with $\kappa \in \Q_p$. For the theorem above this is not a problem however as we will shortly explain. Indeed, if $X$ attains semi-stable reduction over a field extension $K$ of $\Q_\ell$ of ramification index $r$ then $\mathrm{red}$ only takes values in the finite set $\graph_{\frac 1r  \Z}$, the subset of points in $\graph$ whose distance to a vertex lies in  $\frac 1r \Z$. And on this subset of points (in fact on the larger set of points $\graph_\Q$) one can show that $j_\graph$ only takes on $\Q$-rational values. So $j_\graph \circ \mathrm{red}$ takes values in $\Q$, and after embedding $\Q$ in $\Q_p$ it makes sense to multiply it with $\kappa$. 
\end{remark}
Concretely, $j_\graph$ is the unique continuous, edge-wise quadratic function that vanishes at $\mathrm{red}(b)$, has second derivative $-\tfrac{1}{l(e)}\,e^*F(\pi(e))$ along each edge $e$, and whose outgoing slopes at each vertex $w$ sum to $-\tfrac12\operatorname{Tr}(F\mid V_p(X_w))$. 

The non-zero constant $\kappa$ merely records the normalisation of the $p$-adic height and is immaterial for us, as the focus in this paper is mainly on when $h_\ell$ is zero. Existence and uniqueness of $j_\graph$ are \Cref{lem:laplacian_mass,lem:measure_mass} (from the Appendix).

\Cref{thm:BDMTV_localheight} pins down $j_\graph$ by a differential equation on $\graph_\Q$. The next lemma solves that equation along a single edge; it is the only computation with $j_\graph$ that we shall need.

\begin{lemma}\label{lem:quadratic_along_edge}
Let $\graph$ be the dual graph of a semi-stable model of $X$ at $\ell$ (\Cref{def:reduction_graph}), let $F$ be an endomorphism as in \Cref{def:corr_endomorphism}, let $\mu_F$ be its measure (\Cref{def:corr_measure}) and let $g\colon\graph_\Q\to\R$ be a continuous piecewise polynomial function with $\nabla^2(g)=\mu_F$ (\Cref{def:graph_laplacian}). Let $e\in E(\graph)$ be an edge, put $l\colonequals l(e)$ and identify $e$ with the interval $[0,l]$ by arclength, the point $0$ being the source $s(e)$ and the point $l$ the target $t(e)$. Assume that
\begin{enumerate}
    \item every vertex lying in the interior of $e$ in the subdivision of $\graph$ induced by the minimal regular model corresponds to a component of the special fibre of genus $0$, and
    \item the edge part of $\mu_F$ has constant density $c\in\R$ along $e$.
\end{enumerate}
Set $a\colonequals g(s(e))$ and $b\colonequals g(t(e))$. Then
\[
    g(x)\;=\;a\;+\;(b-a)\,\frac{x}{l}\;-\;\frac{c}{2}\,x\,(x-l)
    \qquad\text{for all }x\in[0,l].
\]
In particular, if $a=b=0$, then $g_e(x)=-\tfrac{c}{2}\,x\,(x-l)$, which vanishes identically if $c=0$. % \st{and vanishes exactly at the two endpoints of $e$ if $c\neq 0$}.
\end{lemma}

\begin{proof}
Subdividing $e$ into segments $[x_{i-1},x_i]$, $1\leq i\leq n$, with $0=x_0<x_1<\dots<x_n=l$, changes neither $\graph_\Q$ nor $g$, and it does not change $\mu_F$ by \Cref{rem:measure_subdivision}; we may therefore compute with the subdivision. Write $w_i$ for the interior vertex sitting at $x_i$, $1\leq i\leq n-1$. Each $w_i$ has degree $2$, the two segments meeting there being $[x_{i-1},x_i]$ and $[x_i,x_{i+1}]$.

Comparing the coefficients of the segment $[x_{i-1},x_i]$ in $\nabla^2(g)=\mu_F$ gives, by \Cref{def:graph_laplacian} and assumption~(ii), the identity $g''=-c$ on $(x_{i-1},x_i)$. As $g$ is polynomial there, $g|_{[x_{i-1},x_i]}$ has degree at most $2$ with leading coefficient $-c/2$.

Comparing the coefficients of the vertex $w_i$ gives that the sum of the outgoing slopes of $g$ at $w_i$ equals $-\tfrac12\operatorname{Tr}(F\mid V_p(X_{w_i}))$. By assumption~(i) the component $X_{w_i}$ has genus $0$, so its Jacobian is trivial, $V_p(X_{w_i})=0$, and this trace vanishes. In the coordinate $x$ the two outgoing slopes of $g$ at $w_i$ are $-g'(x_i^-)$ and $g'(x_i^+)$, so their sum being $0$ means precisely that the two one-sided derivatives of $g$ at $x_i$ agree.

Hence $h(x)\colonequals g(x)+\tfrac{c}{2}x^2$ is continuous on $[0,l]$, is affine on each segment $[x_{i-1},x_i]$, and has matching one-sided derivatives at $x_1,\dots,x_{n-1}$; therefore $h$ is affine on all of $[0,l]$ and $g_e(x)=-\tfrac{c}{2}x^2+\beta x+\gamma$ for some $\beta,\gamma\in\R$. Evaluating at $x=0$ gives $\gamma=a$, and evaluating at $x=l$ gives $b=-\tfrac{c}{2}l^2+\beta l+a$, i.e.\ $\beta=\tfrac{b-a}{l}+\tfrac{c\,l}{2}$. Substituting these two values yields the displayed formula.
\end{proof}

\subsection{Quaternion algebras}
For a comprehensive treatment of these topics, see \cite{Voight2021}.
    
\begin{definition}\label{def:quaternion_algebra}
A \emph{quaternion algebra} $B$ over a field $F$ is a four-dimensional central simple algebra over $F$. It is \emph{split} over $F$ if $B\cong M_2(F)$. If $\characteristic F\neq 2$ and $a,b\in F^\times$, we denote by $\left(\frac{a,b}{F}\right)$ the quaternion algebra $F \oplus F i \oplus F j \oplus F k$ with multiplication given by $i^2 = a$, $j^2 = b$ and $k = ij = -ji$.
\end{definition}

If $\characteristic F\neq 2$, every quaternion algebra over $F$ is isomorphic to $\left(\frac{a,b}{F}\right)$ for some $a,b\in F^\times$. For a similar description when $\characteristic F = 2$, see \cite[Def.~6.2.1]{Voight2021}.

The matrix algebra $M_2(F)$ is an example of a quaternion algebra, and by the Wedderburn--Artin theorem (\cite[Theorem 7.1.1, Corollary 7.1.2]{Voight2021}) every quaternion algebra over $F$ is either a split or a division algebra.

In the basis $\{1,i,j,k\}$ as above, one can define the following:
\begin{itemize}
    \item the standard involution $\overline{a+bi+cj+dk}=a-bi-cj-dk$,
    \item the reduced trace $\trd(\alpha) = \alpha+\overline{\alpha}$, and
    \item the reduced norm $\nrd(\alpha)=\alpha\overline{\alpha}$.
\end{itemize}

\begin{definition}\label{def:ramification}
Let $B$ be a quaternion algebra over a global field $F$. For a place $v$ of $F$, put $B_v\colonequals B\otimes_F F_v$. The \emph{ramification set} $\Ram(B)$ is the set of places $v$ of $F$ such that $B_v$ is a division algebra. If $F=\Q$, the product $D$ of the primes in $\Ram(B)$ is called the \emph{discriminant} of $B$, and $B$ is called \emph{definite} if $\infty\in\Ram(B)$.
\end{definition}

\begin{theorem}\label{thm:ramification_set}
For every quaternion algebra $B$ over $\Q$, the set $\Ram(B)$ is a finite set of places of $\Q$ of even cardinality. Conversely, every finite set of places of $\Q$ of even cardinality is $\Ram(B)$ for some quaternion algebra $B$ over $\Q$, unique up to isomorphism.
\end{theorem}
\begin{proof}
See \cite[p.~452, Proof of Main Theorem 14.6.1]{Voight2021}.
\end{proof}
    
We are particularly interested in the quaternion algebra $B=B_{\ell,\infty}$ over $\Q$ that is ramified at $\ell$ and infinity. Ramification at $\ell$ implies that the completion $B_{\ell} = B \otimes_{\Q} \Q_\ell$ is the unique quaternion division algebra over $\Q_\ell$.
Ramification at the infinite place means that $B_{\infty} = B \otimes_{\Q} \R$ is isomorphic to \emph{Hamilton's quaternions}, $\mathbb{H} = \left(\frac{-1,-1}{\R}\right)$, which is also a division algebra. For all other primes $q \neq \ell$, the algebra is split, meaning $B_q = B \otimes_{\Q} \Q_q \cong M_2(\Q_q)$. Quaternion algebras that ramify over the infinite place are called \emph{definite} quaternion algebras.

One can explicitly realise \( B_{\ell,\infty}\) as \( B_{\ell, \infty} \cong \left(\frac{a,b}{\Q}\right)\) for suitable negative integers \(a,b\) as in \cite[Example 14.2.13]{Voight2021}, e.g., \( B_{2,\infty} = \left(\frac{-1,-1}{\Q}\right) \) and \(B_{3,\infty} = \left(\frac{-1,-3}{\Q}\right) \). Only imaginary quadratic fields can embed in \(B_{\ell,\infty}\) (which also holds for all definite quaternion algebras).

\begin{definition}\label{def:order}
An \emph{order} $\orderO$ in a quaternion algebra $B$ over $\Q$ is a subring of $B$ (with $1$) that is a lattice, i.e., a finitely generated $\Z$-submodule of rank $4$. An order is \emph{maximal} if it is not properly contained in any other order of $B$. Orders in quaternion algebras over $\Q_q$ are defined in the same way, with $\Z_q$ in place of $\Z$. For an order $\orderO$ in $B$, its completion $\orderO_q\colonequals\orderO\otimes_\Z\Z_q$ is an order in $B\otimes_\Q\Q_q$.
\end{definition}

\noindent For example, $M_2(\Z)$ is a maximal order in $M_2(\Q)$.

\begin{definition}\label{def:eichler_order}
An \emph{Eichler order} is the intersection of two maximal orders, which includes the maximal orders themselves. In the split quaternion algebra over \(\Q_q\), every Eichler order is conjugate to \( \big(\begin{smallmatrix} \Z_q  & \Z_q  \\ q^e \Z_q  & \Z_q \end{smallmatrix}\big)\), the \emph{standard Eichler order of level} \(q^e \), where $e\geqslant 0$.
\end{definition}

One of the nice features of quaternion algebras is that a lot of properties are local (which means that a property holds if and only if it holds over every completion). Since being an Eichler order is a local property, we can define the level for global orders as follows. 

\begin{definition}\label{def:eichler_level}
Let $B$ be a quaternion algebra over $\Q$ of discriminant $D$, and let $M$ be a positive integer coprime to $D$. An order $\orderO$ in $B$ is an \emph{Eichler order of level} $M$ if, for every prime $q\mid M$, the completion $\orderO_q$ is isomorphic to the standard Eichler order of level $q^e$, where $q^e\parallel M$, and $\orderO_q$ is maximal for every prime $q\nmid M$.
\end{definition}

Fix an order $\orderO$ in a quaternion algebra $B$ over $\Q$.
\begin{definition}
A $\Z$-lattice $I\subset B$ is a \emph{right $\orderO$-ideal} if $I\orderO=I$. It is \emph{invertible} (equivalently, \emph{locally principal}) if for every prime $q$ we have $I \otimes_\Z \Z_q = \alpha_q \orderO_q$ for $\alpha_q \in B_q^\times$. The \emph{left order} of $I$ is \(
  O_L(I)\colonequals\{\,x\in B : xI\subseteq I\,\}.
\)
\end{definition}  

\noindent If $I$ is invertible, then $O_L(I)$ is an order in $B$ with $O_L(I)\otimes_\Z \Z_q = \alpha_q \orderO_q \alpha_q^{-1}$ for every prime $q$, i.e., $O_L(I)$ is locally conjugate to $\orderO$. In particular, if $\orderO$ is an Eichler order of level $M$, then so is $O_L(I)$.

\begin{definition}\label{def:class_set}
Two invertible right $\orderO$-ideals $I, J$ are \emph{equivalent} if $J=\alpha I$ for some $\alpha \in B^\times$. The \emph{class set} $\Cls\orderO$ is the set of equivalence classes $[I]$ of invertible right $\orderO$-ideals, and $\#\Cls\orderO$ is the \emph{class number} of $\orderO$.
\end{definition}

\noindent The class set is finite \cite[Theorem 17.1.1]{Voight2021}. Since $O_L(\alpha I) = \alpha O_L(I) \alpha^{-1}$, the left order of a class $[I]$ is well defined up to conjugation.

Now, consider a quadratic field $K/\Q$ and let $S$ be a quadratic order in \( K\).
An embedding of $S$ into $\orderO$ also gives an embedding $\phi : K \inj B$ by extending scalars. The starting embedding $\phi:S\inj \orderO$ is said to be \emph{optimal} if it doesn't extend to a larger quadratic order while staying inside the same quaternionic order, i.e., if $\phi(K) \cap \orderO = \phi(S)$.

\begin{definition}\label{def:splitting_symbol}
Let $K$ be a quadratic field, $q$ a prime, and $K_q\colonequals K\otimes_\Q\Q_q$. Define
\begin{equation}\label{eq:pInK}
\legendre{K}{q} \colonequals
\begin{cases}
1, & \text{if $q$ splits in $K$, i.e., } K_q \cong\Q_q\times\Q_q;\\
0, & \text{if $q$ ramifies in $K$;}\\
-1, & \text{if $q$ is inert in $K$.}
\end{cases}
\end{equation}
\end{definition}

\noindent If $D_K$ denotes the discriminant of $K$, then $\legendre{K}{q}$ equals the Kronecker symbol $\legendre{D_K}{q}$. In particular, if $q$ is odd and $K=\Q(\sqrt{d})$ with $d$ a squarefree integer, then $\legendre{K}{q}=\legendre{d}{q}$ is the Legendre symbol, cf.~\cite[§\,30.5.2]{Voight2021}. The symbol measures the local obstruction to embeddings: if $B$ is ramified at $q$, then $K_q$ embeds into $B\otimes_\Q\Q_q$ if and only if $\legendre{K}{q}\neq 1$. %The local factors in Eichler's formula below are expressed in terms of this symbol.

Given a group $\Gamma$ acting on $\orderO$ by ring automorphisms (in practice a subgroup $\Gamma\subseteq\orderO^\times$ acting by conjugation $x\mapsto\gamma x\gamma^{-1}$) two optimal embeddings are \emph{$\Gamma$-equivalent} if they differ by the action of some $\gamma\in\Gamma$. 

\begin{definition}\label{def:embedding_number}
Let $\orderO$  be an order in a quaternion algebra $B$ over $\Q$, and let $S$ be an order in a quadratic field $K$. We denote by $m(S, \orderO)$ the number of $\orderO^\times$-conjugacy classes of optimal embeddings $S\inj\orderO$.
\end{definition}

We now state Eichler's formula for the number of optimal embeddings.
\begin{theorem}[Eichler]\label{thm:eichler_embeddings}%, see {\cite[Theorem~30.4.7]{Voight2021}}]
    Let $\orderO \subset B$ be an Eichler order of squarefree level $M$, and $D = \disc B$, so that $\gcd(D,M)=1$. Let $K$ be an imaginary quadratic field, and let $S\subset K$ be an order that is maximal at every prime dividing $DM$. Denote the class number of $S$ by $h(S)$. Then
\[
\sum_{[I]\in\operatorname{Cls} \orderO}
   m\left(S,O_L(I)\right)
   \;=\; h(S)\,\prod_{q\mid D}\left(1-\left(\tfrac{K}{q}\right)\right)\,
         \prod_{q\mid M}\left(1+\left(\tfrac{K}{q}\right)\right).
\]
\end{theorem}
\begin{proof}
Combine \cite[Theorem~30.4.7]{Voight2021} with the local embedding numbers in \cite[\S\S\,30.5--30.6]{Voight2021}. See also \cite[Example 30.7.5]{Voight2021} for a very similar statement.
\end{proof}

\subsection{Deuring's correspondence}
Let $\ell$ be a prime and $M$ a squarefree positive integer coprime to $\ell$. For a supersingular elliptic curve $E$ over $\Fbar_\ell$ and a cyclic subgroup $G \subseteq E(\Fbar_\ell)$ of order $M$, let
\[ \End(E,G) \colonequals \{ \phi \in \End(E) : \phi(G) \subseteq G\}.\]
Two such pairs $(E, G)$ and $(E', G')$ are \emph{isomorphic} if there is an isomorphism $\phi \colon E \to E'$ with $\phi(G) = G'$.

\begin{theorem}[Deuring correspondence]\label{thm:Deuring}
    Let $\ell$ and $M$ be as above.
    \begin{enumerate}
        \item For every pair $(E, G)$ as above, $\End(E)\otimes \Q \cong B_{\ell, \infty}$, the ring $\End(E)$ is a maximal order in it, and $\End(E,G)$ is an Eichler order of level $M$ in it.
        \item Fix a pair $(E_0, G_0)$ as above and put $\orderO_0 \colonequals \End(E_0, G_0)$. There is a bijection between the isomorphism classes of pairs $(E, G)$ and $\Cls\orderO_0$ such that, if $(E,G)$ corresponds to $[I]$, then $\End(E,G)\cong O_L(I)$ and $\Aut(E, G)\cong O_L(I)^\times$.
    \end{enumerate}
\end{theorem}
\begin{proof}
    The statements about $\End(E)$ are Deuring's theorem \cite{Deuring1951}, \cite[Theorem 42.1.9]{Voight2021}, and the statement about $\End(E,G)$ is \cite[Theorem 3.7]{Arpin24} (see also \cite[Remark 42.3.10]{Voight2021}). Part~(ii) follows from the equivalence of categories \cite[Theorem 6.5]{Arpin24}, which sends $(E,G)$ to an invertible left $\orderO_0$-module whose right order is isomorphic to $\End(E,G)$.
\end{proof}

\begin{remark}\label{rem:lefteqright}
    The standard involution turns left ideal classes into right ideal classes, and right orders into left orders. A reader may find a different convention compared to ours depending on the reference (e.g.\ \cite{Arpin24}).
\end{remark}

\subsection{Automorphisms of elliptic curves}

At some points in the article we will need the classification for the possibilities of \( \Aut(E) \). To make our results hold for all $N$ and not just $N$ coprime to $6$, we will also need to sometimes deal with the subtleties that can occur because of the larger automorphism groups in characteristic $2$ and $3$.

\begin{proposition}
\label{prop:aut_classification}
Let $E$ be an elliptic curve over $\overline{\F}_\ell $, and let $j(E)$ denote its $j$-invariant. The automorphism group $\Aut(E)$ is given by:
\begin{enumerate}
    \item If $j(E) \neq 0, 1728$, then $\Aut(E) \cong \Z/2\Z$.
    \item If $j(E) = 1728$ and $\ell \neq 2, 3$, then $\Aut(E) \cong \Z/4\Z$.
    \item If $j(E) = 0$ and $\ell \neq 2, 3$, then $\Aut(E) \cong \Z/6\Z$.
    \item If $\ell = 3$ and $j(E) = 0 = 1728$, then $\Aut(E) \cong \Z/3\Z \rtimes \Z/4\Z$.
    \item If $\ell = 2$ and $j(E) = 0 = 1728$, then $\Aut(E) \cong \SL_2(\F_3)$.
\end{enumerate}
where the action of $\Z/4\Z$ on $\Z/3\Z$ is the unique nontrivial action, and the isomorphism  $\Aut(E) \cong \SL_2( \F_3)$ comes from the action of $\Aut(E)$ on $E[3](\overline{\F}_\ell )$.
\end{proposition}
\begin{proof}
This is Proposition 4.4 from \cite{Schoof87}, who references \cite[p.~182]{Tate78} and Deuring \cite[\S\,5]{Deuring41}. The same statement just mentioning the group order is also in \cite[Theorem~III.10.1]{Silverman1}.
\end{proof}

\section{$X_0(N)^*$ is stable over $\Z_\ell$ when $\ell^2 \nmid N$}\label{sec:semi-stable}

The full Atkin--Lehner group is $W_N=\{w_d : d\mid N,\ \gcd(d,N/d)=1\}\cong(\Z/2)^{\omega(N)}$, with $w_d w_{d'}=w_{d''}$ for $d''=dd'/\gcd(d,d')^2$, where $\omega(N)$ denotes the number of distinct prime factors of $N$.

\begin{setup}\label{setup:al_quotient} Let $\ell$ be a prime, let $M$ be an integer (not necessarily squarefree) coprime to $\ell$, let $N \colonequals \ell M$, and let $W \subseteq W_N$ be a subgroup of the Atkin--Lehner operators on $X_0(N)$.
\end{setup}
Note that with this notation we have $X_0(N)^* \colonequals X_0(N)/{W_N}$.

The goal of this section is to describe the dual graph of the minimal regular model of $X_0(N)^*$ over $\Z_\ell$ in terms of invertible right ideals of an Eichler order in a quaternion algebra. The minimal regular model of this curve turns out to be semi-stable and hence it is not necessary to go to a ramified extension of $\Z_\ell$ to obtain a semi-stable model. 

While our motivation is to describe the minimal regular model of $X_0(N)^*$, we will work out the theory in a slightly more general setting since it requires almost no extra work.

\begin{proposition}[quotient of semi-stable curve] \label[proposition]{prop:semi-stable_quotient}
Let $S$ be the spectrum of a discrete valuation ring $R$ with closed point $s$. 
Let $X$ be a semi-stable quasi-projective curve over $S$, endowed with 
the action of a finite group $G$. Then the quotient scheme 
$Y = X/G$ is semi-stable. More precisely, let $x \in X_s$ be a closed 
point, $y$ its image in $Y_s$. Then we have the following properties:
\begin{enumerate}
    \item if $X$ is smooth at $x$, then $Y$ is smooth at $y$;
    \item if $x$ is an ordinary double point of $X_s$, then $y$ is a 
    smooth or ordinary double point; 
    \item let $I'$ be the image of the inertia group $I$ at $x$ in 
    $\mathrm{Aut}_R(\mathcal{O}_{X,x})$. If $x$ is of 
    thickness $m$, and if $y$ is a double point, then $y$ is of 
    thickness $mn$, where $n$ is the order of $I'$. Furthermore if $x$ is split, then $y$ is split.
    \item If $x$ is an ordinary double point then $y$ is a smooth point if and only if there is an element $\iota$ in the inertia group $I$ at $x$ that swaps the two branches at $x$.
\end{enumerate}    
\end{proposition}
\begin{proof}
Items $(i), (ii)$ and $(iii)$ are~\cite[Proposition~10.3.48]{Liu2002}. Note that we slightly strengthened $(iii)$ with respect to the statement in Liu by not requiring $x$ to be split. This is okay, since the size of $I'$ does not change under \'etale base change. The fourth item follows from the proof of \cite[Proposition~10.3.48]{Liu2002}. Indeed, let $I_0$ be the subgroup of elements of $I$ that do not swap the two branches. Then the proof of \cite[Proposition~10.3.48]{Liu2002} contains a case distinction based on whether $I=I_0$, or not. In the case $I=I_0$ one has that $y$ is an ordinary double point and in the case $I \neq I_0$ one has that $y$ is a smooth point.
\end{proof}

The above allows one to describe a semi-stable model of $X_0(N)/W$ together with its minimal desingularisation in terms of $X_0(N)$ given a semi-stable model of $X_0(N)$ together with the thickness at every ordinary double point.

Note that over $\F_\ell$ there are two maps \begin{align*}
  \iota_1 \colon X_0(M)_{\F_\ell} &\to X_0(\ell M)_{\F_\ell}, &
  (E, G) &\mapsto (E,\, G + \ker \Frob_\ell),\\
  \iota_2 \colon X_0(M)_{\F_\ell} &\to X_0(\ell M)_{\F_\ell}, &
  (E, G) &\mapsto \bigl(E^{(\ell)},\, G^{(\ell)} + \ker \operatorname{Ver}_\ell\bigr),
\end{align*}
where $\Frob_\ell \colon E \to E^{(\ell)}$ is the relative Frobenius and
$\operatorname{Ver}_\ell \colon E^{(\ell)} \to E$ is the Verschiebung.

\begin{theorem}[Semi-stable model of $X_0(N)$] \label[theorem]{thm:semi-stable_X0} Let $\ell, M$ and $N$ be as in \cref{setup:al_quotient}. Then:
\begin{enumerate}
\item The curve model $X_0(N)_{\Z[1/M]}$ is a semi-stable curve over $\Spec\Z[1/M]$ that is smooth outside of the supersingular points in characteristic $\ell$.
\item $X_0(N)_{\F_\ell}$ consists of two irreducible components. The map $\iota_1$ above gives an isomorphism with the first component and $\iota_2$ one with the second component. These two components meet transversally along the supersingular points, where a supersingular point $x \in X_0(M)(\Fbar_\ell)$ on the first copy is identified with $\Frob_\ell  x$ on the second copy of $X_0(M)(\Fbar_\ell)$. An element $w_d$ sends each component to itself if and only if $\ell \nmid d$.
\item If $(E,G)$ represents a supersingular point $x\in X_0(N)(\Fbar_\ell)$, then the thickness of $x$ equals $\frac 1 2 \#\Aut(E, G)$.
\end{enumerate} 
\end{theorem}

\begin{proof}
This follows from~\cite[Theorem~VI.6.9]{DeRa} and~\cite[Theorem~13.4.7]{KatzMazur}. Note for~(ii) that from the moduli description, $w_d$ swaps kernel of Frobenius and Verschiebung if and only if $\ell \mid d$. Semi-stability follows from~(ii) and because $X_0(M)_{\Fbar_\ell}$ is smooth and because the gluing is transversal in the intersection points. 
\end{proof}
\begin{remark}\label{rem:al_action}
    That $w_d$ sends each component to itself if and only if $\ell \nmid d$ can be made more precise. Indeed, if $\ell \nmid d$, then the action of $w_d$ on $X_0(N)_{\F_\ell}$ restricts to the action of $w_d$ on each component $X_0(M)_{\F_\ell}$. If $\ell \mid d$ then $w_d$ on $X_0(N)_{\F_\ell}$ acts by first swapping the components and then acting by $w_{d/\ell}$ on each component.
\end{remark}

\begin{theorem}[stable model of Atkin--Lehner quotients] \label{thm:semi-stable_star}
Let $\ell, M, N$ and $W$ be as in \Cref{setup:al_quotient}. Then 
\begin{enumerate}
\item The quotient $(X_0(N)/W)_{\Z[1/M]}$ is a semi-stable curve over $\Spec\Z[1/M]$ that is smooth outside of the supersingular points in characteristic $\ell$.

\item Define $W' \colonequals \left\{w_d \in W \middle | \, \ell \nmid d\right\}$. \begin{enumerate}
    \item If $W'\neq W$, then $(X_0(N)/W)_{\F_\ell }$ consists of a single component whose desingularisation is isomorphic to $(X_0(M)/W')_{\F_\ell }$.

    \item If $W'=W$, then $(X_0(N)/W)_{\F_\ell }$ has two components, both isomorphic to $(X_0(M)/W')_{\F_\ell }$, and $(X_0(N)/W)_{\F_\ell }$ is obtained from these two components by gluing every supersingular point $x$ on the first component to $\Frob_\ell x$ on the second component.
\end{enumerate}

\item Let $(E,G)$ represent a supersingular point $x\in X_0(N)(\Fbar_\ell )$ and let $W_x$ be the stabiliser of $x$ inside $W$. \begin{enumerate}
    \item If there is an integer $d \mid M$ such that $w_{\ell d} \in W_x$, then $X_0(N)/W$ is smooth at the image of $x$.
    \item If such an integer $d$ does not exist, then $X_0(N)/W$ has an ordinary double point at the image of $x$ whose thickness is $\#W_x \ \frac 1 2 \#\Aut(E, G)$.
\end{enumerate} 
\item If $g((X_0(N)/W)_\Q) \geq 2$, then it is stable. 
\end{enumerate} 
\end{theorem}
\begin{proof}
(i) follows from combining~\Cref{thm:semi-stable_X0} with \Cref{prop:semi-stable_quotient}.\\
(ii) The description of the individual components follows because $w_d$ interchanges the two components if and only if $\ell \mid d$. That is, $W'$ is exactly the subgroup of $W$ that sends each component to itself. Furthermore $W'$ is acting faithfully on the generic fibre of each component, which shows birationality of each component with $(X_0(M)/W')_{\F_\ell}$. The description in case $(b)$ on how  $(X_0(N)/W)_{\Fbar_\ell }$ is obtained by gluing together two copies of $(X_0(M)/W')_{\F_\ell }$ follows from the analogous description of $X_0(N)_{\F_\ell}$ in terms of $X_0(M)_{\F_\ell}$.

(iii)\,(a) smoothness: an element $w_{\ell d}\in W_x$ swaps the two branches at the double point (it exchanges the two copies of $X_0(M)$), and a branch-swapping involution folds the double point to a smooth point by \Cref{prop:semi-stable_quotient} (iv).

(iii)\,(b): by~\Cref{prop:semi-stable_quotient}, the thickness is $\#W_x\cdot\tfrac12\#\Aut(E,G)$: this is $mn$ with $m=\tfrac12\#\Aut(E,G)$ and $n=\#W_x$ (note that in case (b) no element of $W_x$ swaps branches) and that $W_x \inj I'$ because a nontrivial $w_d$ acting trivially on the local ring would act trivially on the whole $X_0(M)$ component, forcing $d = 1$.

(iv) If there is only one component in the special fibre, it is geometrically connected. And since the arithmetic genus of this component is at least $2$, it cannot be (geometrically) isomorphic to $\P^1$. Note that the assumption $g((X_0(N)/W)_\Q)  \geq 2$ is needed (rather than merely $\geq 1$) as fibres of stable curves need to have arithmetic genus at least $2$ by definition.
If there are two components, and both are of genus $1$, they have an intersection point, so the model is stable. If both components are $\P^1$, then there are at least $3$ intersection points (because $2\cdot 0 + k - 1\geqslant 2$ where $k$ is the number of intersection points).
\end{proof}

\begin{lemma}\label{lem:nonsplit}
  Let $\ell, M, N$ and $W$ be as in \Cref{setup:al_quotient} and assume furthermore that $w_\ell \in W$. Then all ordinary double points of $(X_0(N)/W)_{\F_\ell}$ are $\F_\ell$-rational and non-split. 
\end{lemma}
\begin{proof}
As $(X_0(N)/W)_{\F_\ell}$ is smooth away from the supersingular points, it suffices to analyse the singularities at the supersingular points of $(X_0(N)/W)_{\F_\ell}$. Let $W' \subseteq W$ be the subgroup of all Atkin--Lehner operators $w_d$ with $\ell \nmid d$. Then $X_0(N)/W \cong (X_0(N)/W')/w_\ell$. Let $X_1$ and $X_2$ be the two components of $(X_0(N)/W')_{\F_{\ell}}$, and let $\iota_1$ and $\iota_2$ denote the inclusion maps of $X_1$ and $X_2$ into $(X_0(N)/W')_{\F_{\ell}}$. Note that $X_1$ and $X_2$ are both isomorphic to $(X_0(M)/W')_{\F_{\ell}}$. Let $f : X_1 \to (X_0(N)/W)_{\F_\ell}$ be the composition of $\iota_1$ with the quotient map $(X_0(N)/W')_{\F_{\ell}} \to (X_0(N)/W)_{\F_{\ell}}$. Note that $f$ is an isomorphism away from the supersingular points, and $X_1$ is smooth. So $X_1$ is the normalisation of the curve $(X_0(N)/W)_{\F_{\ell}}$, and hence it suffices to show that the fibres of $f$ above the ordinary double points consist of two $\Frob_\ell$-conjugate points. Now let $x \in X_1(\Fbar_\ell)$ be a supersingular point. By \cref{rem:al_action}, $w_\ell$ acts on $X_0(N)_{\F_\ell}$ by swapping the two components and acting trivially on each of them; as $W$ is abelian, this action descends to $(X_0(N)/W')_{\F_{\ell}}$, where $w_\ell$ just swaps the two components $X_1$ and $X_2$. We now make a case distinction depending on whether $\Frob_\ell x = x$ or not.

The case $\Frob_\ell x = x$. The map $w_\ell$ swaps the two components $X_1$ and $X_2$. Additionally $x$ on the first component is glued to $\Frob x = x$ on the second component. So $\iota_1(x)$ is a fixed point of $w_\ell$, and hence $f(x)$ is a smooth point of $(X_0(N)/W)_{\F_\ell}$ (by \Cref{thm:semi-stable_star}~(iii)(a)). So, there is nothing to prove in this case.

The case $\Frob_\ell x \neq x$. This time $\iota_1(x)$ is not a fixed point of $w_\ell$. However, we do have that $\iota_1(x)=w_\ell(\iota_1(\Frob_\ell x))$. In particular, $f(x)=f(\Frob_\ell x)=\Frob_\ell f(x)$ so $f(x)$ is $\F_\ell$-rational. And $f(x)$ is non-split since it has a pair of Galois conjugate points lying above it in $X_1$.
\end{proof}

\begin{remark}A semi-stable model of $X_0(\ell)^+_{\F_\ell}$ was already described in the Master's thesis of Noordman \cite[Theorem 4.29]{Noordman2017}. From that description it is also immediately clear that all of the ordinary double points of $X_0(\ell)^+_{\F_\ell}$ are non-split.\end{remark}

\begin{corollary}[Minimal regular model of Atkin--Lehner quotients]\label{cor:minimal_regular_AL_quot}
Assume $w_\ell \in W$ and $g((X_0(N)/W)_\Q)  \geq 2$, and let $\mathcal X$ be the minimal regular model of $X_0(N)/W$ over $\Z[1/M]$.  Then $\mathcal X_{\F_{\ell^2}}$ is the curve obtained from $(X_0(N)/W)_{\F_{\ell^2}}$ by replacing each ordinary double point $x = (E,G)\in (X_0(N)/W)_{\F_{\ell^2}}$ by a chain of $\P^1_{\F_{\ell^2}}$'s of length $l := \#W_x \ \frac 1 2 \#\Aut(E, G)-1$. For $1 \leq i \leq l$ the action of Frobenius interchanges the $i$-th $\P^1_{\F_{\ell^2}}$ with the $(l+1-i)$-th $\P^1_{\F_{\ell^2}}$. 
\end{corollary}
\begin{remark}
    Here $W_x$ denotes the stabiliser in $W$ of any preimage of $x$ in $X_0(N)$. Since $W$ is abelian, all preimages have the same stabiliser, $W_x$ is well defined.
\end{remark}
\begin{proof}The special fibre of $\mathcal X$ does not change if we work over $\Z_\ell$ instead of $\Z[1/M]$. Let $\orderO_K$ denote the ring of Witt vectors of $\F_{\ell^2}$, then every ordinary double point of $(X_0(N)/W)_{\orderO_K}$ is split. The minimal regular model commutes with base change along $\orderO_K \to \Z_\ell$ \Cite[Lemma 10.3.33]{Liu2002} so we can just work over $\orderO_K$ instead. That we need to replace $x$ with a chain of $\P^1$ of length $l := \#W_x \ \frac 1 2 \#\Aut(E, G)-1$ now follows from \Cref{lem:minimal_regular_from_stable_model}. The statement of the action of Frobenius follows since the double point was non-split. Indeed, if $l=1$ there is nothing to prove. So now let $l>1$, and let $C_0$ denote the strict transform of  $(X_0(N)/W)_{\F_{\ell^2}}$ in $\mathcal X_{\F_{\ell^2}}$, because $x$ is non-split it will have two Frobenius conjugate points $x_0$ and $x_1$ lying above it on $C_0$. These points $x_0$ and $x_1$ are exactly the points where the chain of $\P^1$'s connects to $C_0$. And as Frobenius interchanges $x_0$ and $x_1$ it has no other option than to reverse the chain of $\P^1$'s. 
\begin{figure}[ht]
  \centering
\begin{tikzpicture}[thick, line cap=round,
    pt/.style={circle, fill, inner sep=1.3pt}]
 
  % ---------- Left: nodal curve X_0(N)^* ----------
  \begin{scope}
    % parametrisation of a nodal cubic, loop on top, node at origin
    \draw[domain=-1.45:1.45, samples=120, smooth]
      plot ({1.2*\x*(\x*\x-1)}, {-1.2*(\x*\x-1)});
    \node[pt] at (0,0) {};
    \draw[->, shorten >=6pt] (0,-1.1) node[below] {$x$} -- (0,0);
    \node at (0,-2.9) {$X_0(N)^*_{\F_{\ell^2}}$};
  \end{scope}
 
  % ---------- Right: C_0 with a chain of P^1's ----------
  \begin{scope}[xshift=6.5cm]
    % the component C_0
    % (symmetric under x -> -x)
    \draw[name path=C0] (-2.1,-1.0) .. controls (-1.8,2.4) and (1.8,2.4)
      .. (2.1,-1.0);
    \node[anchor=south] at (0,1.5) {$C_0$};
 
    % chain of three P^1's, mirror symmetric: \ - /
    % the middle one is horizontal, the outer ones make angle
    % 180-\ang resp. \ang with the x-axis
    \def\ang{35}
    \def\ovh{0.3}  % overhang of every P^1 beyond its intersection points
    \coordinate (M) at (0,-0.5);              % centre of the middle P^1
    \coordinate (P) at ($(M)+(-0.95,0)$);     % middle meets left P^1
    \coordinate (Q) at ($(M)+(0.95,0)$);      % middle meets right P^1
    \draw ($(P)!-\ovh cm!(Q)$) -- ($(Q)!-\ovh cm!(P)$);
    % invisible long helper lines, only used to find x_0 and x_1
    \path[name path=L1] (P) -- ($(P)+(180-\ang:3)$);
    \path[name path=L3] (Q) -- ($(Q)+(\ang:3)$);
    \node[above right] at ($(P)+(180-\ang:0.6)$) {$\P^1$};
    \node[below]       at (M)                    {$\P^1$};
    \node[above left]  at ($(Q)+(\ang:0.6)$)     {$\P^1$};
 
    % the points x_0 = C_0 cap L1 and x_1 = C_0 cap L3
    \path[name intersections={of=C0 and L1, by=x0}];
    \path[name intersections={of=C0 and L3, by=x1}];
    % outer P^1's, same overhang \ovh at both ends
    \draw ($(P)!-\ovh cm!(x0)$) -- ($(x0)!-\ovh cm!(P)$);
    \draw ($(Q)!-\ovh cm!(x1)$) -- ($(x1)!-\ovh cm!(Q)$);
    \node[pt] at (x0) {};
    \node[pt] at (x1) {};
    \draw[->, shorten >=6pt] ([xshift=-1.1cm]x0) node[left] {$x_0$} -- (x0);
    \draw[->, shorten >=6pt] ([xshift=1.1cm]x1) node[right] {$x_1$} -- (x1);
 
    % Frobenius
    \draw[<->] (-1.6,-1.6) to[bend right=30]
      node[below] {$\Frob_\ell$} (1.6,-1.6);
 
    \node at (0,-2.9) {$\mathcal{X}_{\F_{\ell^2}}$};
  \end{scope}
 
\end{tikzpicture}
  \caption{A picture of a hypothetical case where $X_0(N)^*_{\F_{\ell^2}}$ has a single ordinary double point of thickness 4. }
  \label{fig:split-node-dising}
\end{figure}
\end{proof}

The description of the semi-stable model in \Cref{thm:semi-stable_star} directly gives the following description of the dual graph of the semi-stable model.
\begin{definition}\label{def:dual_graph_quotient}
    Let $\ell,M,N$ and $W$ be as in \Cref{setup:al_quotient}. Then define the metric graph $\graph_W$ as follows:
    \begin{itemize}
        \item if there is an Atkin--Lehner operator in $W$ of the form $w_{\ell d}$ with $d|M$, then $\graph_W$ has a single vertex $v_0$. The edges of $\graph_W$ correspond to $W$-orbits of supersingular points that are not fixed by any Atkin--Lehner operator in $W$ of the form $w_{\ell d}$. Let $x \in X_0(N)(\Fbar_\ell )$ be such a point represented by a pair $(E,G)$. Then the edge associated to its $W$-orbit is a loop attached to $v_0$ of length $\#W_x \ \frac 1 2 \#\Aut(E, G)$.
        \item otherwise $\graph_W$ has two vertices $v_0, v_1$ with one edge of length $$\#W_x \ \frac 1 2 \#\Aut(E, G)$$ connecting them for each $W$-orbit of supersingular points.
    \end{itemize}
    We use the notation $\graph_0(N)$ and $\graph_0(N)^*$ when $\#W=1$ and $W=W_N$ the full Atkin--Lehner group, respectively.
\end{definition}

\begin{corollary}\label{cor:dual_graph_quotient}
    The dual graph of the special fibre over $\F_\ell $ of $X_0(N)/W$ is $\graph_W$.
\end{corollary}

Recall from \Cref{def:reduction_graph} that the length of an edge of the dual graph of a semi-stable model over $\Z_\ell$ is the thickness of the corresponding double point, so it is a positive integer.
\begin{definition}\label{def:flower_graph}
A metric graph $\graph$ is called a \emph{flower graph} if it has a unique vertex $v_0$, and every edge $e$ of $\graph$ is a loop at $v_0$ of integral length.
\end{definition}
If $N=\ell M$ is squarefree, then $w_\ell\in W_N$, so by the first case of \Cref{def:dual_graph_quotient} and \Cref{cor:dual_graph_quotient} the dual graph $\graph_0(N)^*$ of $X_0(N)^*$ at $\ell$ is a flower graph.
\subsection{Hecke Operators on Atkin--Lehner quotients and their dual graphs}

In order to define Hecke operators on Atkin--Lehner quotients we slightly generalise \Cref{setup:al_quotient}.
\begin{setup}\label{setup:al_quotient_hecke} Let $\ell,q$ be distinct primes and $M$ be an integer (not necessarily squarefree) coprime to $\ell$ and $q$. Let $N \colonequals \ell M$ and let $W \subseteq W_N$ be a subgroup of the Atkin--Lehner operators on $X_0(N)$. We also view $W$ as acting on $X_0(qN)$.
\end{setup}

In the above setup we have two different degeneracy maps $\alpha,\beta : X_0(qN) \to X_0(N)$, with $\alpha(E,G) = (E, G[N])$ and $\beta(E,G) = (E/G[q], G/G[q])$. For every $d \mid N$ (such that $\gcd(d, N/d)=1$), the Atkin--Lehner operator $w_d$ commutes with both $\alpha$ and $\beta$. This implies that $\alpha$ and $\beta$ also induce well defined morphisms $\alpha, \beta : X_0(qN)/W \to X_0(N)/W$.

\begin{definition}
    Let $\ell,q, M, N, W$ be as in \Cref{setup:al_quotient_hecke}, and let $Z \subseteq (X_0(N)/W)^2$ be the image of $X_0(qN)/W$ under $(\alpha,\beta):X_0(qN)/W \to (X_0(N)/W)^2$. Then the Hecke operator $T_q$ is defined to be the correspondence induced by $Z$.
\end{definition}

Now in order to be able to compute the local heights on $X_0(N)/W$ with respect to the correspondence $T_q$, or more generally $f(T_q)$ for some polynomial $f \in \Z[x]$, we need to be able to compute how $T_q$ acts on the homology of the dual graph of $X_0(N)/W$. Note that it is already known how $T_q$ acts on the homology of the dual graph of $X_0(N)$ (see, e.g., \cite[p.~445]{Ribet1990}, or \cite[Theorem 4.3]{KohelHecke}). We will do this using the following lemma which more generally describes how to compute the action of correspondences on quotients.

\begin{lemma}\label{lem:quotient-correspondence}
    Let $K$ be a finite extension of $\Q_\ell$ and $X$ and $Y$ be nice curves over $K$. Let $f_1,f_2 : X \to Y$ be finite morphisms, and let $G$ be a finite group acting faithfully on both $X$ and $Y$ in a way that commutes both with $f_1$ and $f_2$. Let $g_1, g_2 : X/G \to Y/G$ be the morphisms induced by $f_1$ and $f_2$ and $\pi_X :X \to X/G$ and $\pi_Y: Y \to Y/G$ be the quotient maps. Then
    \begin{enumerate}

        \item $\pi_Y^* : H_1(\graph_{Y/G}) \to H_1(\graph_Y)$ is injective and $\pi_{Y,*}$ is injective when restricted to $\pi_Y^*( H_1(\graph_{Y/G}))$.
        \item The endomorphism $f_{2,*}\circ f_1^*$ of $H_1(\graph_{Y})\otimes \Q$ restricts to a well defined endomorphism of $\pi_Y^*(H_1(\graph_{Y/G}) \otimes \Q)$.
        \item We have $f_{2,*} \circ f_1^* \circ \pi_Y^{*} = \pi_Y^{*} \circ g_{2,*} \circ g_1^*$ in $\Hom(H_1(\graph_{Y/G}),H_1(\graph_{Y}))$.
    \end{enumerate}
\end{lemma}
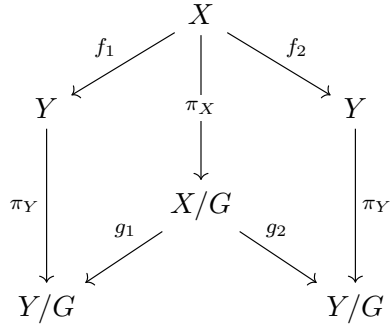
\[
\begin{tikzcd}
& X \arrow[dl, "f_1"'] \arrow[dr, "f_2"] \arrow[dd, "\pi_X" description] & \\
Y \arrow[dd, "\pi_Y"'] & & Y \arrow[dd, "\pi_Y"] \\
& X/G \arrow[dl, "g_1"' near start] \arrow[dr, "g_2" near start] & \\
Y/G & & Y/G
\end{tikzcd}
\]
Informally, the above lemma says we can just compute the correspondence $g_{2,*} \circ g_1^*$ by restricting $f_{2,*} \circ f_1^*$ to the image of $\pi_Y^*$.

\begin{proof}
    (i) The injectivity follows since $\pi_{Y,*}\circ \pi_Y^*$ is multiplication by $\#G$, and $H_1(\graph_{Y/G})$ is torsion free.
    
    (ii) This follows since $\pi_Y^*(H_1(\graph_{Y/G})) \otimes \Q=(H_1(\graph_Y)\otimes \Q)^G$ and $f_{2,*}\circ f_1^*$ commutes with the action of $G$.
    
    (iii) Since $H_1(\graph_{Y})$ is torsion free, it suffices to show this after tensoring with $\Q$. By $(ii)$ both maps have images contained in $\pi_Y^*( H_1(\graph_{Y/G})\otimes \Q)$ and so by $(i)$ it suffices to show equality after composing with $\pi_{Y,*}$. The statement follows since
    \begin{align*}     
   \pi_{Y,*} \circ f_{2,*} \circ f_1^* \circ \pi_Y^{*}&= 
    g_{2,*}\circ \pi_{X,*} \circ \pi_X^{*} \circ g_1^*=
     g_{2,*} \circ [\#G] \circ g_1^*\\
     & =    [\#G] \circ g_{2,*} \circ g_1^*
    =\pi_{Y,*} \circ \pi_Y^{*} \circ g_{2,*} \circ g_1^*.
    \end{align*}
    In the last part, we used that the actions of $G$ on both $X$ and $Y$ are faithful as this ensures that both $\pi_{X,*} \circ \pi_X^{*} $ and $\pi_{Y,*} \circ \pi_Y^{*}$ are multiplication by $\#G$.
    \end{proof}

Applying \Cref{lem:quotient-correspondence} allows one to compute the action of $T_q$ on $H_1(\graph_W)$ in terms of the action of $T_q$ on $H_1(\graph_0(N))$. To be precise:
\begin{corollary}\label{cor:computing_Tq}
    Let $\ell,q,N,M,W$ be as in \Cref{setup:al_quotient_hecke}, and $\pi:X_0(N) \to X_0(N)/W$ be the quotient map. Let $e_1,\ldots,e_n$ be a basis of $H_1(\graph_W)$ and $\pi^*(e_1),\ldots,\pi^*(e_n)$ be the corresponding basis of $\pi^*(H_1(\graph_W))  \subseteq H_1(\graph_0(N))$. Then the matrices expressing the actions of $T_q$ on the basis $e_1,\ldots,e_n$ and on the basis $\pi^*(e_1),\ldots,\pi^*(e_n)$ are the same.
\end{corollary}

The above corollary is what is actually used in our code in order to compute the action of $T_q$ on $e_1,\ldots,e_n$, as the action of $T_q$ can already be computed on $H_1(\graph_0(N))$ using Brandt matrices; see \Cref{fig:hecke_quotient} for an illustration. 

% \flowerpetal{<vertex>}{<direction angle>}{<reach>}{<half-opening angle>}
\newcommand{\flowerpetal}[4]{%
  \draw[ed] (#1) .. controls +({#2+#4}:#3) and +({#2-#4}:#3) .. (#1);}

\begin{figure}[h]
\begin{tikzpicture}[
    >={Stealth[length=3.2mm,width=2.2mm]},
    vtx/.style={circle,fill=black,inner sep=1.7pt},
    ed/.style={line width=0.7pt},
    ar/.style={->,line width=0.7pt}
  ]

  %%% left copy of X_0(N) %%%
  \node[vtx] (l1) at (0,0)   {};
  \node[vtx] (l2) at (2.2,0) {};
  \draw[ed] (l1) to[bend left=22]  (l2);
  \draw[ed] (l1) to[bend right=22] (l2);
  \node at (1.1,1.15) {$X_0(N)$};

  %%% middle: X_0(qN), raised above the two copies %%%
  \node[vtx] (m1) at (5.3,2.2) {};
  \node[vtx] (m2) at (7.9,2.2) {};
  % one entry per edge; change the list to change the number of edges
  \foreach \b in {72,45,18,-12,-40,-72}{
    \draw[ed] (m1) to[bend left=\b] (m2);
  }
  \node at (6.6,3.7) {$X_0(qN)$};

  %%% right copy of X_0(N) %%%
  \node[vtx] (r1) at (11.0,0) {};
  \node[vtx] (r2) at (13.2,0) {};
  \draw[ed] (r1) to[bend left=22]  (r2);
  \draw[ed] (r1) to[bend right=22] (r2);
  \node at (12.1,1.15) {$X_0(N)$};

  %%% degeneracy maps
  \draw[ar] (4.7,1.9) -- (3,0.9);
  \draw[ar] (8.7,1.9) -- (10.4,0.9);

\end{tikzpicture}

\begin{tikzpicture}[
    >={Stealth[length=3.2mm,width=2.2mm]},
    vtx/.style={circle,fill=black,inner sep=1.7pt},
    ed/.style={line width=0.7pt},
    ar/.style={->,line width=0.7pt}
  ]

  %%% left copy of X_0(N)^*: two loops %%%
  \node[vtx] (l) at (1.1,0) {};
  \foreach \a in {0,180}{\flowerpetal{l}{\a}{2.6}{55}}
  \node at (1.1,1.35) {$X_0(N)^*$};

  %%% middle: X_0(qN)/W_N, six loops, raised above the two copies %%%
  \node[vtx] (m) at (6.6,2.6) {};
  \foreach \a in {0,60,...,300}{\flowerpetal{m}{\a}{2.2}{27}}
  \node at (6.6,4.5) {$X_0(qN)/W_N$};

  %%% right copy of X_0(N)^*: two loops %%%
  \node[vtx] (r) at (12.1,0) {};
  \foreach \a in {0,180}{\flowerpetal{r}{\a}{2.6}{55}}
  \node at (12.1,1.35) {$X_0(N)^*$};

  \draw[ar] (4.5,1.9) -- (3,1);
  \draw[ar] (8.8,1.9) -- (10,1);
\end{tikzpicture}
    \caption{Illustration of how \Cref{lem:quotient-correspondence} is used in \Cref{cor:computing_Tq}. Here $W_N$ acts on $X_0(qN)$ via the Atkin--Lehner involutions $w_d$ with $d \mid N$.}
    \label{fig:hecke_quotient}
\end{figure}
\section{The relation between local heights and Atkin--Lehner fixed points}\label{sec:criterion}

The main goal of this section will be to show that the existence of $p$-adic heights supported at $p$ will follow if one can bound the genus of $X_0(N)^*$ and the number of long loops in $\graph_0(N)^*$ in the definition below. We will furthermore show that the number of these long loops can be bounded in terms of the number of supersingular points fixed by Atkin--Lehner involutions, up to a small contribution from points with extra automorphisms (\Cref{thm:Bpnw}).

\begin{definition}\label{def:L_ell}
Let $N=\ell M$ be squarefree. We denote by $L_\ell=L_\ell(N)$ the number of \defn{long loops} of $\graph_0(N)^*$, i.e.\ loops $e$ with $l(e)>1$.
\end{definition}
\begin{remark}\label{rem:L_ell_is_also_ODP_count}
 As $X_0(N)^*$ is semi-stable over $\Z_\ell$, by \Cref{def:reduction_graph} (with $r=1$), $L_\ell$ is also the number of ordinary double points of $(X_0(N)^*)_{\Fbar_\ell}$ of thickness $>1$.
\end{remark}

The following lemma is an important step in the proof of \cref{thm:main1.1} as it reduces that theorem to bounding the quantities $g(X_0(N)^*)$ and  $L_\ell(N)$.
\begin{lemma}\label[lemma]{lem:atLeast_g_minus_Ll_heights}
Let $N$ be a squarefree integer, $p\nmid N$ a prime and $b \in X_0(N)^*(\Q)$ the unique cusp. Then there are at least $$g(X_0(N)^*) -1- \sum_{\ell \mid N} L_\ell(N)$$
linearly independent $p$-adic heights supported at $p$ using $b$ as a basepoint.
\end{lemma}
We prove this lemma only after \cref{lem:hecke_is_End}.
\begin{corollary}\label{cor:clean_criterion}
Let $N$ be squarefree and $p\nmid N$ a prime. If $g(X_0(N)^*)>1+\sum_{\ell\mid N}L_\ell(N)$, then there is a non-trivial $p$-adic height on $X_0(N)^*$ supported at $p$.
\end{corollary}

\begin{definition}\label{def:visible_edge}
Let $X/\Q_\ell$ be a nice curve with a stable model $\mathcal X/\Z_\ell$, and let $e$
be an edge of its dual graph $\graph$, with corresponding ordinary double point
$x_e\in\mathcal X_{\Fbar_\ell}$. We call $e$ \defn{long} if $l(e)\geq 2$. We call $e$ \defn{visible} if $x_e$ is
$\F_\ell$-rational and either
\begin{enumerate}
  \item $x_e$ is split (i.e. both branches at $x_e$ are defined over $\F_\ell$) and
        $l(e)\geq 2$, or
  \item $x_e$ is non-split and $l(e)$ is even.
\end{enumerate}
\end{definition}
Note that the visible edges form a subset of the long edges.

The motivation for this definition is that visible edges are exactly the edges that have a point of $X(\Q_\ell)$ specialising to their interior under $\mathrm{red}$ (i.e.\ the edges corresponding to the ordinary double points of the stable model that have a $\Q_\ell$-point specialising to them); see claims (1) and (2) in the proof of the lemma below. The notion of long edge is useful since it has a close relation to fixed points of Atkin--Lehner, and hence allows us to bound the number of visible edges.

\begin{lemma}\label{lem:local_height_vanishes_long_edges}
Let $X/\Q_\ell$ be a nice curve of genus at least $2$ with a stable model
$\mathcal X/\Z_\ell$ whose dual graph $\graph$ is a flower graph with central vertex
$v_0$. Let $Z$ be a trace-zero correspondence on $X$, $F$ the endomorphism of
$\mathrm H_1(\graph,\Q)$ induced by $Z$, and let $b\in X(\Q_\ell)$ be a basepoint with
$\mathrm{red}(b)=v_0$. Then
\[
  h_\ell= 0 \text{ on } X(\Q_\ell)
  \iff
      e^*F\left(\pi(e)\right)=0\qquad\text{for every visible }e\in E(\graph).
\]
In particular, if $e^*F\left(\pi(e)\right)=0$ for all edges of length $\geq 2$, then $h_\ell= 0$.
\end{lemma}

\begin{proof}
The idea is to use the description of the local height from
\Cref{thm:BDMTV_localheight}, which gives $h_\ell=\kappa\cdot j_\graph\circ\mathrm{red}$
with $\kappa\neq 0$. Since $\mathrm{red}(b)=v_0$, we have $j_\graph(v_0)=0$. This
reduces the lemma to the following three claims.
\begin{enumerate}[label=(\arabic*)]
  \item For every visible edge $e$ there is a point $z\in X(\Q_\ell)$ such that
        $\mathrm{red}(z)$ lies in the interior of $e$.
  \item Conversely, for every $z\in X(\Q_\ell)$, the point $\mathrm{red}(z)$ is either
        $v_0$ or it lies in the interior of a visible edge.
  \item Let $e$ be an edge and put $F_{ee}\colonequals e^*F(\pi(e))$. Then $j_\graph$ vanishes identically on $e$ if $F_{ee}=0$, and has no zeros in the interior of $e$ if $F_{ee}\neq 0$. %If $F_{ee}=0$, then $j_\graph$ vanishes identically on $e$. If $F_{ee}\neq 0$, then $j_\graph$ has no zeros in the interior of $e$.
\end{enumerate}

Indeed, suppose first that $F_{ee}=0$ for every visible edge $e$, and let
$z\in X(\Q_\ell)$. By Claim~(2), $\mathrm{red}(z)$ is either $v_0$, where $j_\graph$
vanishes, or lies on a visible edge, where $j_\graph$ vanishes by Claim~(3). Hence
$h_\ell(z)=0$.

Conversely, suppose that $F_{ee}\neq 0$ for some visible edge $e$. Claim~(1) gives a
$z\in X(\Q_\ell)$ with $\mathrm{red}(z)$ in the interior of $e$. Claim~(3) gives
$j_\graph(\mathrm{red}(z))\neq 0$, so $h_\ell(z)\neq 0$.

The final statement of the lemma follows because every visible edge has length at
least $2$. It remains to prove the three claims.

\emph{Preliminaries on the minimal regular model.} Let $\mathcal X'$ be the minimal
regular model of $X$. By \Cref{lem:minimal_regular_from_stable_model} it is the minimal
desingularisation of $\mathcal X$. The fact that $\graph$ is a flower graph gives the following description of
$\mathcal X'_{\Fbar_\ell}$.
\begin{itemize}
  \item As $\graph$ has a single vertex, $\mathcal X_{\Fbar_\ell}$ is irreducible. Its
        strict transform $C_0$ corresponds to $v_0$.
  \item Above each double point $x_e$ lies a chain $C_{e,1},\dots,C_{e,l(e)-1}$ of projective lines of multiplicity $1$. Consecutive members meet transversally
        in one point, and $C_{e,1}$ and $C_{e,l(e)-1}$ each meet $C_0$ in one point,
        lying over the two branches of $\mathcal X_{\Fbar_\ell}$ at $x_e$ respectively.
  \item Identifying $e$ with $[0,l(e)]$, the component $C_{e,i}$ corresponds to the
        point $x=i$ of $e$.
\end{itemize}
In particular, the interior vertices of $e$ in the subdivision induced by
$\mathcal X'$ correspond to components of genus $0$.

The action of Frobenius permutes the components of
$\mathcal X'_{\Fbar_\ell}$, preserving which of them meet, and it maps the chain above
$x_e$ to the chain above $\Frob_\ell(x_e)$. Fix an edge $e$.
\begin{enumerate}[label=(\alph*)]
  \item If $x_e$ is not $\F_\ell$-rational, then $\Frob_\ell(C_{e,i})\neq  C_{e,i}$ for all
        $i$.
  \item If $x_e$ is $\F_\ell$-rational and split, then $\Frob_\ell$ fixes both branches at
        $x_e$. Hence it fixes both ends of the chain above $x_e$, and so every
        member: $\Frob_\ell(C_{e,i})=C_{e,i}$ for all $i$.
  \item If $x_e$ is $\F_\ell$-rational and non-split, then $\Frob_\ell$ swaps the two
        branches at $x_e$. Hence it swaps the two ends of the chain, so it reverses
        the chain: $\Frob_\ell(C_{e,i})=C_{e,l(e)-i}$ for all $i$. In particular,
        $C_{e,i}$ is $\Frob_\ell$-stable if and only if $l(e)$ is even and $i=l(e)/2$.
\end{enumerate}
Consequently, the edges $e$ for which some member of the chain above $x_e$ is
$\Frob_\ell$-stable are exactly the visible edges.

\emph{Proof of Claim~(1).} Let $e$ be visible. By $(a),(b)$ and $(c)$ above, there is a
$\Frob_\ell$-stable component $C=C_{e,i}$ with $0<i<l(e)$. So $C$ descends to a smooth
projective curve over $\F_\ell$ that is geometrically isomorphic to $\P^1$. Being a
conic over a finite field, it has an $\F_\ell$-point, so it is isomorphic to
$\P^1_{\F_\ell}$ and has $\ell+1\geq 3$ rational points. The curve $C$ meets the other
components of $\mathcal X'_{\Fbar_\ell}$ in exactly two points, so there is a
$y\in C(\F_\ell)$ lying on no other component. As $\mathcal X'$ is regular and $C$ has
multiplicity $1$, the morphism $\mathcal X'\to\Spec\Z_\ell$ is smooth at $y$. By
Hensel's lemma, $y$ lifts to some $z\in\mathcal X'(\Z_\ell)=X(\Q_\ell)$, and
$\mathrm{red}(z)$ is the point $x=i$ in the interior of $e$.

\emph{Proof of Claim~(2).} Let $z\in X(\Q_\ell)=\mathcal X'(\Z_\ell)$ and let
$y\in\mathcal X'(\F_\ell)$ be its reduction. As noted after \Cref{def:reduction_map},
$y$ is a smooth point of $\mathcal X'_{\F_\ell}$, so it lies on a unique component $C$
of $\mathcal X'_{\Fbar_\ell}$. Since $y$ is $\F_\ell$-rational, $C$ is $\Frob_\ell$-stable.
If $C=C_0$, then $\mathrm{red}(z)=v_0$. Otherwise $C=C_{e,i}$ for some edge $e$ and some
$0<i<l(e)$. By $(a),(b)$ and $(c)$ above, $e$ is visible, and $\mathrm{red}(z)$ is the point
$x=i$ in its interior.

\emph{Proof of Claim~(3).} By \Cref{def:corr_measure}, the edge part of $\mu_F$ has constant
density $F_{ee}/l(e)$ along $e$. Since $e$ is a loop at $v_0$ and $j_\graph(v_0)=0$,
both endpoint values of $j_\graph$ on $e$ vanish. Hence
\Cref{lem:quadratic_along_edge} applies and gives
\[
  j_\graph(x)=\frac{F_{ee}}{2\,l(e)}\,x\,\bigl(l(e)-x\bigr)
  \qquad\text{for } x\in[0,l(e)].
\]
And Claim~(3) follows.
\end{proof}

\begin{remark}\label{rem:flower_heights}
The proof of Claim~(3) in \cref{lem:local_height_vanishes_long_edges} also contains a rather explicit description of $j_\graph$ along the edge $e$. Additionally, if $z \in X(\Q_\ell)$ specialises to a non-split ordinary double point $x_e$ of $\mathcal X$, then the proof even shows that $\mathrm{red}(z)$ has to be the midpoint of the edge $e$ and hence $j_{\graph}(\mathrm{red}(z)) =F_{ee}\,l(e)/8$.
\end{remark}

We will now start bounding the long loops of $\graph_0(N)^*$, or equivalently, the number of ordinary double points of $X_0(N)^*$ of thickness $>1$. By \Cref{thm:semi-stable_star} (iii)(b) those points come from supersingular points that either have extra automorphisms, or are fixed points of some Atkin--Lehner operator.

\begin{proposition}[Contribution from extra automorphisms]\label{prop:extra_automorphism}
 Let $\ell$ be a prime, $M$ be a squarefree integer coprime to $\ell$ and $N=\ell M$. The number of supersingular points $(E,G)=x \in X_0(N)(\Fbar_\ell)$ with $\#\Aut(E,G) > 2$ is at most $(1+c_\ell) 2^{\omega(N)}$ where
 $$ c_\ell \colonequals \begin{cases}
     0\text{ if } \ell>3\\
     1 \text{ if } \ell=3\\
     \frac 5 2 \text{ if } \ell=2.
 \end{cases}$$
\end{proposition}
\begin{proof}
By \Cref{prop:aut_classification} a supersingular point $x=(E,G)$ can have $\#\Aut(E,G)>2$ only if $\#\Aut(E)>2$, that is, only if $j(E)=0$ or $j(E)=1728$. An automorphism $\gamma\in\Aut(E)$ of order $n\in\{3,4\}$ generates a subring $\Z[\gamma]\cong\Z[\zeta_n]$ of $\End(E)$, and $\gamma\in\Aut(E,G)$ if and only if $\gamma(G)=G$.

\emph{A counting bound.}
Fix such a $\gamma$ of order $n$ and let us bound the number of cyclic subgroups $G\subseteq E$ of order $M$ with $\gamma(G)=G$. For each prime $q\mid M$ we have $q\ne \ell$, so $E[q]\cong(\Z/q\Z)^2$ and $G_q\colonequals G \cap E[q]$ is a line in $E[q]$ with $\gamma(G_q)=G_q$; conversely $G=\bigoplus_{q\mid M}G_q$ is recovered from the family $(G_q)_{q\mid M}$. Now $\gamma$ acts $\F_q$-linearly on $E[q]$ with minimal polynomial $X^2+1$ (if $n=4$) or $X^2+X+1$ (if $n=3$). We claim that the action of $\gamma$ on $E[q]$ is not scalar. Indeed, if $q\nmid n$ then its eigenvalues are the two distinct primitive $n$-th roots of unity, so $\gamma$ is \emph{not} a scalar on $E[q]$. If $q \mid n$ (so $q=2,3$) then one can show that $(\gamma -1)/q$ and $(\gamma +1)/q$ have norm $q^{-1}$ or $q^{-2}$, so neither is integral, contradicting that $\gamma$ acts as $\pm 1$ mod $q$. A non-scalar element of $\GL_2(\F_q)$ leaves at most two lines invariant, so there are at most two choices for $G_q$, and hence at most $2^{\omega(M)}$ subgroups $G$ with $\gamma(G)=G$.

A point $(E,G)$ satisfies $\#\Aut(E,G)>2$ if and only if $\Aut(E,G)$ contains an element of order $3$, $4$ or $6$; since the square of an order-$6$ element has order $3$ and fixes the same $G$, this happens precisely when $G$ is invariant under a generator of one of the cyclic subgroups of order $3$ or $4$ of $\Aut(E)$. Writing $s_n$ for the number of subgroups of order $n$ of $\Aut(E)$, the number of extra-automorphism points supported on a fixed curve $E$ is therefore at most $(s_3+s_4)\,2^{\omega(M)}$.

\emph{Case $\ell>3$.} Here $\#\Aut(E)>2$ forces $j(E)=1728$ (so $\Aut(E)\cong C_4$, $s_3=0$, $s_4=1$) or $j(E)=0$ (so $\Aut(E)\cong C_6$, $s_3=1$, $s_4=0$), and these are two distinct curves. Summing, the number of extra-automorphism points is at most $(1+1)\,2^{\omega(M)}=2^{\omega(N)}$, which is the claim with $c_\ell =0$.

\emph{Case $\ell\in\{2,3\}$.} Then $0\equiv1728\pmod \ell$ and there is a unique supersingular curve $E_0$ up to isomorphism \cite[Theorem~V.4.1]{Silverman1}. If $\ell=2$ this curve has $\Aut(E_0)\cong\SL_2(\F_3)$ and $\Aut(E_0)\cong \Z/3\Z \rtimes \Z/4\Z$ if $\ell=3$ (\Cref{prop:aut_classification}). In $\SL_2(\F_3)$ the normal Sylow $2$-subgroup $Q_8$ contains all three subgroups of order $4$, so $s_4=3$, and there are $s_3=4$ Sylow $3$-subgroups.  In $\Z/3\Z \rtimes \Z/4\Z$ one has $s_4=3$ subgroups of order four and $s_3=1$ subgroup of order $3$. Hence the number of extra-automorphism points is at most $(3+4)\,2^{\omega(M)}=\tfrac72\,2^{\omega(N)}$ when $\ell=2$, and at most $(3+1)\,2^{\omega(M)}=2\cdot2^{\omega(N)}$ when $\ell=3$. This proves one can take $c_2=\frac 5 2$ and $c_3 = 1$.
\end{proof}

\begin{definition}\label{def:SS_fixed_points}
    Let $\ell$ be a prime, $M$ be a squarefree integer coprime to $\ell$ and $d \mid \ell M$. Let $d$ be larger than $1$. Denote $N=\ell M$. Then $\nu_{\ell}(N,d)$ is defined to be the number of fixed points of $w_d$ in the supersingular locus of $X_0(N)(\overline \F_\ell)$.
\end{definition}

\begin{theorem}\label{thm:Bpnw} Let $\ell, M,N$ and $W$ be as in \Cref{setup:al_quotient}, with $M$ squarefree. Then the number of ordinary double points in $(X_0(N)/W)_{\Fbar_\ell}$ of thickness $>1$ is at most 
    \begin{align*}
       B(\ell,N,W) \colonequals  \frac 2 {\#W} \left((1+c_\ell)2^{\omega(N)} + \sum_{1<d \mid M, w_d\in W}\nu_{\ell}(N,d)\right)
    \end{align*} 
    where $c_\ell$ is as in \Cref{prop:extra_automorphism}. In particular, we have $L_\ell(N) \leq B(\ell,N, W_N)$, where $L_\ell(N)$ is as in \Cref{def:L_ell}.
\end{theorem}
\begin{proof}
According to \Cref{thm:semi-stable_star} the ordinary double points of thickness $>1$ come from supersingular points $(E,G) = x \in X_0(N)(\Fbar_\ell)$ such that either $\# \Aut(E,G)>2$ or $w_d x=x$ for some $w_d \in W$ with $d \mid M$ and $d > 1$.  And furthermore, those that are fixed by at least one $w_{\ell d} \in W$ with $d \mid M$, will always be smooth points, so such supersingular points will not contribute anything.

First we look at the contribution coming from those points where $\# \Aut(E,G) > 2$ but are not fixed by any nontrivial Atkin--Lehner operator in $W$. By \Cref{prop:extra_automorphism} there are at most $(1+c_\ell)2^{\omega(N)}$ such supersingular points $(E,G)$, so they account for at most $(1+c_\ell)2^{\omega(N)}/\#W$ points of $(X_0(N)/W)_{\Fbar_\ell}$; this contributes at most $2(1+c_\ell)2^{\omega(N)}/\#W$ to the bound $B(\ell,N,W)$. 

Now for the contribution from those points that are fixed by some $w_d$ with $1 < d \mid M$ but not by any $w_{\ell d}$ with $d\mid M$. To this end, let $S \subseteq X_0(N)(\Fbar_\ell)$ be the subset of supersingular points that are fixed by at least one Atkin--Lehner operator $w_d\in W$ with $d\mid M$, $d>1$, but not by any $w_{\ell d} \in W$. As $W$ is abelian, this set is stable under the action of $W$. There will be $\#(S/W)$ points of thickness $>1$ coming from this set. We bound $\#(S/W)$ using Burnside's Lemma. 
\begin{align*}
    \#(S/W) &= \frac 1 {\#W}\left(\#S +\sum_{1<d \mid M, w_d\in W} \#S^{w_d}+\sum_{d \mid M, w_{\ell d}\in W} \#S^{w_{\ell d}}\right) \\
    &\leq \frac 1 {\#W}\left(\#S+ \sum_{1<d \mid M, w_d\in W}\nu_{\ell}(N,d)\right)
\end{align*}
In the above we used $S^{w_{\ell d}} = \emptyset$, as $S$ contains no $w_{\ell d}$ fixed points by definition.
The first part of the theorem follows since
\[
    \#S \leq \sum_{1<d \mid M, w_d\in W}\nu_{\ell}(N,d). 
\]
The fact that $L_\ell(N) \leq B(\ell,N,W_N)$ follows from \Cref{rem:L_ell_is_also_ODP_count}.
\end{proof}

\begin{lemma}\label{lem:hecke_is_End}
Let $N$ be a squarefree integer, and let $\T\subseteq \End \, J_0(N)^*$ be the Hecke algebra generated by the Hecke operators $T_q$ with $q\nmid N$. Then $\T \otimes \Q = \End \, J_0(N)^* \, \otimes \Q \cong \NS(J_0(N)^*) \otimes \Q$, and all three have dimension $g\colonequals g(X_0(N)^*)$ over $\Q$.
\end{lemma}

\begin{proof}
Clearly $\T\otimes\Q\subseteq\End \, J_0(N)^*\otimes\Q$. By Atkin--Lehner theory \cite{AtkinLehner1970}, the space $S_2(\Gamma_0(N))^{W_N}\cong H^0(X_0(N)^*,\Omega^1)$, of dimension $g$, has a basis of eigenforms for $\T$ with pairwise distinct systems of eigenvalues. Indeed, taking $W_N$-invariants causes all oldforms to occur with multiplicity 1. Hence, the image of $\T\otimes\C$ seen as acting on modular forms contains the $g$ projections onto these eigenlines, so $\dim_\Q\T\otimes\Q\geqslant g$. On the other hand, since $N$ is squarefree, $\End \, J_0(N)^*\otimes\Q\cong\prod_i K_i$ is a product of totally real number fields \cite[\S2, before Corollary~2.6]{BakerHasegawa}. Accordingly, $J_0(N)^*$ is isogenous to a product of abelian varieties $J_i$ with $\End \, J_i\otimes\Q\cong K_i$, and $[K_i:\Q]\leqslant\dim J_i$ since $K_i$ is totally real \cite[\S21]{Mumford}. Hence $\dim_\Q \End \, J_0(N)^*\otimes\Q\leqslant g$, and therefore $\T\otimes\Q=\End \, J_0(N)^*\otimes\Q$.

By \cite[Proposition 17.2]{MilneAbelianVarieties}, $\NS(J_0(N)^*)\otimes \Q$ is isomorphic to the subspace of $\End \, J_0(N)^* \otimes \Q$ fixed by the Rosati involution. We will denote the Rosati involution by $\dagger$. As $\End \, J_0(N)^* \otimes \Q$ is a product of number fields, it is commutative, and hence the Rosati involution is an algebra morphism.  Positivity rules out swapping two factors, since for an idempotent $e$ of $K_i$ with $e^\dagger\in K_j$, $j\neq i$, we get $e\,e^\dagger=0$. Furthermore, a positive involution of a totally real field is trivial. So $\dagger$ is the identity and hence $\End \, J_0(N)^* \, \otimes \Q \cong \NS(J_0(N)^*) \otimes \Q$.
\end{proof}
With \Cref{lem:hecke_is_End} and \Cref{lem:local_height_vanishes_long_edges} at hand, we can now prove \Cref{lem:atLeast_g_minus_Ll_heights}.

\begin{proof}[Proof of \Cref{lem:atLeast_g_minus_Ll_heights}]
As a shorthand, let $\NS$ denote the Néron--Severi group $\NS(J_0(N)^*)\otimes \Q$ and $\NS^{\tr=0}$ its subspace of trace-zero correspondences. The space of all $p$-adic heights with basepoint $b$ is parametrised by $\NS^{\tr=0}$. By \Cref{lem:hecke_is_End} the rank of the Néron--Severi group equals $g(X_0(N)^*)$.

Let $\ell \mid N$ be a prime, $Z \in \NS^{\tr=0}$ be a correspondence, and $\graph_0(N)^*$ be the dual graph of $X_0(N)^*$ over $\F_\ell$. Let $F_Z$ be the endomorphism of $H_1(\graph_0(N)^*)$ induced by $Z$. If $g(X_0(N)^*)\leqslant 1$ there is nothing to prove, so assume $g(X_0(N)^*)\geqslant 2$. Then $X_0(N)^*$ is stable over $\Z_\ell$ by \Cref{thm:semi-stable_star}~(iv) (as $w_\ell\in W_N$), and its dual graph $\graph_0(N)^*$ is a flower graph by \Cref{cor:dual_graph_quotient} and the paragraph following \Cref{def:flower_graph}. Moreover $\mathrm{red}(b)=v_0$, as the cusp $b$ does not reduce to a supersingular point. So by \Cref{lem:local_height_vanishes_long_edges} one has $h_\ell(X_0(N)^*(\Q_\ell))=0$ as soon as $e^*F_Z (\pi(e))=0$ for every long edge $e$ of $\graph_0(N)^*$. The lemma now follows since $e^*F_Z (\pi(e))$ depends linearly on $Z$, and hence the condition $h_\ell(X_0(N)^*(\Q_\ell))=0$ is implied by $L_\ell(N)$ linear conditions on $Z \in \NS^{\tr=0}$.
\end{proof}

\subsection{Algorithm for computing the space of $p$-adic heights supported at $p$}

 In this subsection we give the algorithm of \Cref{prop:correspondence_algorithm} and prove that it is correct. Throughout, $N$ is a squarefree integer, $g\colonequals g(X_0(N)^*)$, the Hecke algebra $\T$ is as in \Cref{lem:hecke_is_End}, and
\[
  S\colonequals S_2(\Gamma_0(N),\Q)^{W_N},
\]
so that $S\otimes\C\cong H^0(X_0(N)^*_\C,\Omega^1)$. For $t\in\T\otimes\Q$ we write $\operatorname{tr}(t)$ for the trace of $t$ acting on $S$.
 
The idea of the algorithm is as follows. By \Cref{lem:hecke_is_End}, every trace zero element of $\T\otimes\Q$ gives rise to a $p$-adic height. By \Cref{lem:local_height_vanishes_long_edges}, this height is supported at $p$ as soon as it satisfies one linear condition for every long loop of the dual graph $\graph_0(N)^*$ at every prime $\ell\mid N$. If $T_q$ generates $\T\otimes\Q$, we can write every element of $\T\otimes\Q$ as a polynomial in $T_q$, and we compute these linear conditions from Brandt matrices using \Cref{cor:computing_Tq}. To make this explicit we fix some notation.
 
\begin{notation}\label{not:brandt_orbits}
Let $\ell\mid N$ be a prime, $M\colonequals N/\ell$, let $\orderO$ be an Eichler order of level $M$ in $B_{\ell,\infty}$, and let $I_1,\ldots,I_h$ be representatives of $\Cls\orderO$. An ideal class $I_k$ corresponds to a supersingular point $x_k\in X_0(N)(\Fbar_\ell)$ by Deuring's Correspondence (\Cref{thm:Deuring}). By Deligne--Rapoport's description of $X_0(N)_{\F_\ell}$, \Cref{thm:semi-stable_X0}, the dual graph $\graph_0(N)$ consists of two vertices $v_0$ and $v_1$, and the point $x_k\in X_0(N)(\Fbar_\ell)$ corresponds to an edge $e_k$ of $\graph_0(N)$, which we orient from $v_0$ to $v_1$. Via $e_k\mapsto$ $k$-th standard basis vector we identify $\Q\,E(\graph_0(N))$ with $\Q^h$.
\begin{itemize}
    \item The group $W_N$ acts on $\{1,\ldots,h\}$ via $w(x_k)=x_{w(k)}$. For $k\in\{1,\ldots,h\}$ we write $W_{x_k}\subseteq W_N$ for the stabiliser of $x_k$.
    \item Let $\epsilon_\ell\colon W_N\to\{\pm1\}$ be the character with $\epsilon_\ell(w_d)=-1$ if $\ell\mid d$ and $\epsilon_\ell(w_d)=1$ otherwise.
    \item Let $R_\ell$ be the set of $W_N$-orbits $O\subseteq\{1,\ldots,h\}$ such that $\epsilon_\ell$ is trivial on $W_{x_k}$ for one, equivalently all, $k\in O$. For $O\in R_\ell$ we fix $k(O)\in O$ and put $l_O\colonequals\#W_{x_{k(O)}}\cdot\tfrac12\#O_L(I_{k(O)})^\times$ and
    \[
      \tilde e_O\colonequals\sum_{wW_{x_{k(O)}}\in W_N/W_{x_{k(O)}}}\epsilon_\ell(w)\,e_{w(k(O))}\in\Q^h \cong \Q \, E(\graph_0(N)),
    \]
    which is well defined as $\epsilon_\ell$ is trivial on $W_{x_{k(O)}}$.
    \item Let $R_\ell^{\mathrm{vis}}\colonequals\{O\in R_\ell : 2 
\mid l_O\}$.
\end{itemize}
\end{notation}

We give some conceptual explanation for this notation. The set $R_\ell$ is in bijection with the ordinary double points of $X_0(N)^*_{\F_\ell}$ by \Cref{thm:semi-stable_star}(iii).  The integer $l_O$ equals the thickness of the point in  $X_0(N)^*_{\F_\ell}$ corresponding to $O$ as $\#O_L(I_{k(O)})^\times$ is $\#\Aut(E,G)$. The integer $l_O$ is also the length of the loop in $\graph_0(N)^*$ corresponding to $O$. The set $R_\ell^{\mathrm{vis}}$ is exactly the set of those $O\in R_\ell$ corresponding to visible edges; this follows because all ordinary double points of $X_0(N)^*_{\F_{\ell}}$ are defined over $\F_{\ell}$ and are non-split by \Cref{lem:nonsplit}. As we will see later, $\tilde e_O$ is a non-zero multiple of $\pi^*(e_O)$ where $\pi: X_0(N)\to X_0(N)^*$ is the quotient map.
 
\begin{algorithm}\label{alg:clean_correspondences}
\textbf{Input:} a squarefree integer $N$ and a prime $q\nmid N$ such that $T_q$ generates $\T\otimes\Q$.
 
\textbf{Output:} polynomials $f_1,\ldots,f_n\in\Q[x]$ as in \Cref{prop:correspondence_algorithm}.
 
\begin{enumerate}[label=\textbf{Step \arabic*.}, ref=\arabic*, leftmargin=*]
    \item\label{step:Tq_on_S} \emph{(The trace condition.)} Compute the matrix of $T_q$ acting on $S$ using modular symbols \cite[Chapter~8]{SteinBook}, and compute $\tau_i\colonequals\operatorname{tr}(T_q^i)$ for $0\leq i\leq g-1$.
    \item\label{step:brandt} \emph{(The local conditions.)} For every prime $\ell\mid N$:
    \begin{enumerate}[label=(\alph*), ref=\arabic{enumi}(\alph*)]
        \item Compute the action of the Hecke operator $T_q$ on the Brandt module $\Q[\Cls\orderO]=\Q^h$ using the ideas of \cite{KohelHecke}.
        \item Similarly compute the action of $W_N$ on $\{1,\ldots,h\}$ using the ideas of \cite{KohelHecke}.
        \item Compute the stabilisers $W_{x_k}$, the orders $\#O_L(I_k)^\times$ of the unit groups, the sets $R_\ell\supseteq R_\ell^{\mathrm{vis}}$ and the vectors $\tilde e_O$ for $O\in R_\ell$;
        \item Compute the matrix $M_T$ of $T_q$ acting on the span of all the $\tilde e_O$, $O\in R_\ell$, with respect to the basis $(\tilde e_O)_{O\in R_\ell}$.
    \end{enumerate}
    \item\label{step:linalg} \emph{(Linear algebra.)} Compute a basis $c^{(1)},\ldots,c^{(n)}$ of the space of vectors $c=(c_0,\ldots,c_{g-1})\in\Q^g$ satisfying
    \[
      \sum_{i=0}^{g-1} c_i\,\tau_i=0\qquad\text{and}\qquad\sum_{i=0}^{g-1} c_i\,(M_T^i)_{O,O}=0\quad\text{for all }\ell\mid N\text{ and }O\in R_\ell^{\mathrm{vis}},
    \]
    and output $f_j\colonequals\sum_{i=0}^{g-1}c_i^{(j)}x^i$ for $j=1,\ldots,n$.
\end{enumerate}
\end{algorithm}

Note that the most theoretically complicated steps are the computation of the actions $T_q$ and $W_N$, as well as the computation of $\#O_L(I_k)^\times$. Luckily, these are already available in \texttt{Magma}.
 
\begin{proof}[Proof of \Cref{prop:correspondence_algorithm}]
We show that the output of \Cref{alg:clean_correspondences} has the required properties. By \Cref{lem:hecke_is_End}, $\T\otimes\Q$ has dimension $g$. As it is generated by $T_q$, the elements $1,T_q,\ldots,T_q^{g-1}$ form a basis of $\T\otimes\Q$. If $g \leq 1$ then the trace 0 condition shows that there are no $p$-adic heights and hence there is nothing to prove. So we assume $g\geq2$.
 
\emph{Trace zero.} Let $t\in\T\otimes\Q$. Then $\mathrm H^1_{\dR}(X_0(N)^*_\C)$ is the direct sum of $H^0(X_0(N)^*_\C,\Omega^1)\cong S\otimes\C$ and its complex conjugate, and $t$ acts on the latter by the complex conjugate of its action on the former. As $\operatorname{tr}(t)\in\Q$, the trace of $t$ on $\mathrm H^1_{\dR}$ equals $2\operatorname{tr}(t)$, so $t$ has trace zero if and only if $\operatorname{tr}(t)=0$. For $f=\sum_ic_ix^i$ we have $\operatorname{tr}(f(T_q))=\sum_i c_i\tau_i$, so the first equation in Step~\ref{step:linalg} is equivalent to $f(T_q)$ having trace zero.
 
\emph{Supported at $p$.} Since $X_0(N)^*$ has good reduction at all primes not dividing $N$, it suffices to show that for every $\ell\mid N$ and every $f=f_j$, the local height $h_\ell$ with respect to $t\colonequals f(T_q)$ vanishes on $X_0(N)^*(\Q_\ell)$. Fix $\ell$. As $w_\ell\in W_N$, the curve $X_0(N)^*$ is stable over $\Z_\ell$ with dual graph the flower graph $\graph_0(N)^*$ by \Cref{thm:semi-stable_star}~(iv) and \Cref{cor:dual_graph_quotient}. Every endomorphism of $J_0(N)^*$ preserves the filtration on $V_p(X_0(N)^*)$, so the construction of \Cref{def:corr_endomorphism} gives a ring homomorphism $\End\,J_0(N)^*\otimes\Q\to\End\,\mathrm H_1(\graph_0(N)^*,\Q)$, $s\mapsto F_s$. In particular $F_t=f(F_{T_q})$. So by \Cref{lem:local_height_vanishes_long_edges}, it suffices to show that the diagonal entry of $f(F_{T_q})$ at every visible loop of $\graph_0(N)^*$ vanishes if and only if $\sum_{i=0}^{g-1} c_i\,(M_T^i)_{O,O}=0$ for all $\ell\mid N\text{ and }O\in R_\ell^{\mathrm{vis}}$.
 
By \Cref{def:dual_graph_quotient}, the loops of $\graph_0(N)^*$ correspond to the $W_N$-orbits of supersingular points $x\in X_0(N)(\Fbar_\ell)$ that are not fixed by any $w_{\ell d}$, and the loop corresponding to the orbit of $x=(E,G)$ has length $\#W_x\cdot\tfrac12\#\Aut(E,G)$. As the $w_{\ell d}$ are exactly the elements of $W_N$ on which $\epsilon_\ell$ is $-1$, and $\#\Aut(E,G)=\#O_L(I_k)^\times$ for $x=x_k$ by \Cref{thm:Deuring}, the loops of $\graph_0(N)^*$ correspond to the orbits $O\in R_\ell$. The loop $e_O$ corresponding to $O$ has length $l_O$, so the visible loops correspond to the $O\in R_\ell^{\mathrm{vis}}$.
 
Let $\pi\colon X_0(N)\to X_0(N)^*$ be the quotient map. By \Cref{thm:semi-stable_X0}~(ii), $w_d$ swaps $v_0$ and $v_1$ if and only if $\ell\mid d$, so $w\in W_N$ maps $e_k$ to $\epsilon_\ell(w)\,e_{w(k)}$. The element $\pi^*(e_O)$ is $W_N$-invariant and non-zero (\Cref{lem:quotient-correspondence}), and it is supported on the edges $e_k$ lying over $e_O$, i.e.\ on those with $k\in O$. For such a vector the coefficient at $e_{w(k(O))}$ is $\epsilon_\ell(w)$ times the coefficient at $e_{k(O)}$. Hence $\pi^*(e_O)=\lambda_O\tilde e_O$ for some $\lambda_O\in\Q^\times$. In particular, the span of the $\tilde e_O$ is $\pi^*(\mathrm H_1(\graph_0(N)^*, \Q))$, which is stable under $T_q$ by \Cref{lem:quotient-correspondence}~(ii), so $M_T$ in Step~\ref{step:brandt}(d) is well defined. By \Cref{cor:computing_Tq}, the matrix of $F_{T_q}$ with respect to the basis $(e_O)_{O\in R_\ell}$ is $D^{-1}M_T D$, where $D$ is the diagonal matrix with entries $\lambda_O$. So the matrix for $f(F_{T_q})$ is $D^{-1}f(M_T)D$. Since $\graph_0(N)^*$ is a flower graph, every edge is a cycle, so $\mathrm H_1(\graph_0(N)^*,\Q)=\Q\,E(\graph_0(N)^*)$ and the orthogonal projection of \Cref{def:cycle_space} is the identity; hence $e_O^*F(\pi(e_O))$ is the diagonal entry of $F$ at $e_O$. Conjugating by $D$ doesn't change diagonal entries, so the diagonal entry for $f(F_{T_q})$ at $e_O$ is $\sum_ic_i(M_T^i)_{O,O}$.
\end{proof}
 
\begin{remark}\label{rem:Tq_generator}
By the proof of \Cref{lem:hecke_is_End}, $\T\otimes\Q$ is a product of number fields. Hence $T_q$ generates $\T\otimes\Q$ if and only if the characteristic polynomial of $T_q$ acting on $S$ is squarefree, which is easy to check. This condition does not hold for every $q$. For example, let $N=37\cdot43$. The elliptic curves \LMFDBLabel{37.a1} and \LMFDBLabel{43.a1} have root number $-1$, so the Atkin--Lehner eigenvalues of their newforms are $+1$. Each of them therefore contributes a line to $S$, on which $T_2$ acts by $a_2=-2$. So $T_2$ does not generate $\T\otimes\Q$.
\end{remark}

\begin{remark}\label{rem:single_gen}
The algorithm does not fundamentally require $\T$ to be generated by a single element $T_q$. If $\T\otimes\Q$ were generated by $T_{q_1}, \ldots, T_{q_k}$ as a $\Q$-algebra, one could still run the entire algorithm, but instead of the $\Q$-vector space basis $1,T_q,T_q^2, \ldots, T_q^{g-1}$ of $\T \otimes\Q$ one would need to start with a basis of the form $h_1(T_{q_1}, \ldots, T_{q_k}),\ldots, h_g(T_{q_1}, \ldots, T_{q_k})$ of $\T\otimes \Q$, where the $h_i$ are polynomials in $\Q[y_1,\ldots, y_k]$. The main reason for not stating the algorithm in this more general form is expository, as we did not want to bury the main idea of the algorithm in even heavier notation.
\end{remark}

%\section{Formulae for \texorpdfstring{$\nu_{\ell}(N,d)$}{nu	extunderscore p(N,d)}}
\section{Counting supersingular $w_d$-fixed points in characteristic $\ell$}
\label{sec:fixedpoint_formulas}

Recall that $\nu_\ell(N, d)$ is the number of supersingular points in $X_0(N)(\overline{\F}_\ell)$ fixed by $w_d$. We are interested in deriving the formulae for $\nu_\ell(N,d)$ (especially as these counts will give the main term of $B(\ell, N, W_N)$ from \Cref{thm:Bpnw}).

\subsection{From fixed points to embeddings}
We recall that \( B_{\ell,\infty} \) is the quaternion algebra ramified exactly at $\ell$ and $\infty$. In what follows, we will use the subalgebra structure of quaternion algebras. Namely, non-split quaternion algebras don't have proper non-commutative subalgebras. If \( K\subset B\) is a commutative \(F\)-subalgebra and \( K\neq F\), then \( K \) is a quadratic field extension of \(F\), and its centraliser in \(B\) is \(K\) itself (\cite[Corollary 4.4.5]{Voight2021}). Thus two different quadratic subfields cannot commute. Since \( B_{\ell,\infty} \) is definite, any quadratic field contained in it is imaginary.

Deuring's theorem (\cite{Deuring1951}, \cite[Theorem~42.1.9]{Voight2021}) identifies \(\End(E)\otimes\Q\) with \(B_{\ell,\infty}\).
Moreover, Deuring's correspondence gives that, if \( G \) is a cyclic subgroup of order \(M\) coprime to \(\ell\), then \( \orderO = \End(E, G)\) is an Eichler order of level \(M\) in \(B_{\ell,\infty}\). After fixing one base pair $(E_0,G_0)$ and the corresponding Eichler order $\orderO_0=\End(E_0,G_0)$, supersingular pairs $\left(E,G\right)\in X_0(N)_{\Fellbar}$ bijectively correspond to right ideal classes of $\orderO_0$. In this correspondence, if \( (E, G) \longleftrightarrow [I]\), then \(\End(E,G) =O_L(I)\). Even if the order of \(G\) is divisible by \(\ell\), this holds in the following manner: if \(G\) is a cyclic subgroup scheme of order \(\ell M\), then
\(\orderO=\End(E,G)\) is an Eichler order of level \(M=N/\ell\) in \(B_{\ell,\infty}\), and the same correspondence holds.

The goal of this subsection is to prove the following.

\begin{theorem}\label[theorem]{thm:2to1}
    Let $\ell,M,N,d$ be as in \Cref{def:SS_fixed_points} and assume $N>6$.
    \begin{enumerate}
        \item If $d>2$, then $\nu_{\ell}(N,d)$ is equal to half the number of \(\orderO^\times\)-conjugacy classes of embeddings of $\Z[\sqrt{-d}]$ into an Eichler order of level $M$ in $B_{\ell,\infty}$, where $\orderO$ ranges over the corresponding Eichler orders of level $M$ in $B_{\ell,\infty}$.

         \item If $d=2$, then $\nu_{\ell}(N,2)$ is equal to half the total number of \(\orderO^\times\)-conjugacy classes of embeddings in the disjoint union of the two embedding problems
        $\Z[\sqrt{-2}]\hookrightarrow \orderO$ and $\Z[\i]\hookrightarrow \orderO$,
        where $\orderO$ ranges over the corresponding Eichler orders of level $M$ in $B_{\ell,\infty}$.
    \end{enumerate}
\end{theorem}

\begin{remark}
    The embeddings in the above theorem do not need to be optimal.
\end{remark}

We will use the next lemma only for supersingular $E$, however it is true in general.
\begin{lemma}\label{lem:wdInduceEnd}
Let $\ell,M,N,d$ be as in \Cref{def:SS_fixed_points} and $(E,G) \in X_0(N)(\Fbar_\ell)$ be a point. Then 
$(E,G)$ is a fixed point of $w_d$ if and only if there is an $h \in \End(E,G)$ and a $\phi \in \Aut(E)$ such that
\[
h^2 = \phi\circ[d].
\]
\end{lemma}

\begin{proof}
Throughout, ``subgroup'' means finite flat subgroup scheme and $|\cdot|$ its order. This is to take into account the case where $\ell \mid d$, as then $G[\ell]$ is non-reduced. We consider a point $x\in X_0(N)$ as a pair $(E, G)$ where $G$ is a cyclic subgroup of size $N$. For $d\mid N$, we can decompose $G = G[d] \oplus G[N/d]$ and then write the action of the Atkin--Lehner involution as $w_d(E,G)=w_d(E,G[d],G[N/d])=(E/G[d],\; E[d]/G[d],\; G/G[d])$.

Hence $w_d(E,G)\sim (E,G)$ is equivalent to \[ (E/G[d], E[d]/G[d], G/G[d]) \sim (E, G[d], G[N/d]),\]
and this is equivalent to
\begin{equation}\label{eq:equivalentCondition} \text{there exists an isomorphism }\, g:E/G[d] \longrightarrow E
\text{ such that } 
\begin{cases}
g\big(E[d]/G[d]\big)=G[d],\\[2pt]
g\big(G/G[d]\big)=G[N/d].
\end{cases}
\end{equation}

Now assume $w_d(E,G)\sim (E,G)$. Taking both of these equivalent descriptions, we get a map
\[
h\colon E \longrightarrow E/G[d] \overset{g}{\longrightarrow} E,
\]
where \(\ker h=G[d]\) is of size \(d\).
Since \(h\big(E[d]\big)=G[d]\), we have
\[
\ker(h^2)=E[d] = \ker [d],
\]
%which is also \(\ker(h^\vee \circ h)\), since \(h^\vee \circ h\) is
%multiplication by \(d\).
Since multiplication by $[d]$ and $h^2$ have the same kernel, we get \(h^2=\phi\circ[d]\) for some
\(\phi\in \Aut(E)\).

To prove the other direction, assume that there is an $h\in\End(E,G)$ such that $\ker(h^2) = \ker [d]$. Since $d^2 = \deg[d] = \deg(h^2)=(\deg h)^2$, it follows that $\ker h$ is of size $d$.

For a prime $q\mid d$, since $h(G[q]) \subseteq G[q]$ (note that we assume $h(G) \subseteq G$ since $h \in \End(E,G)$), it follows that $h(G[q])$ is either $0$ (a trivial group) or $h(G[q])=G[q]$. If $h(G[q]) = G[q]$, then $0 = [d](G[q]) = h^2(G[q])=h(G[q])$ (here we use $q \mid d$), which is a contradiction. Hence $h(G[q])=0$ for all primes $q\mid d$, and $h(G[d])=0$ consequently. This means our $h$ factors through projection $\pi\colon E \rightarrow E/G[d]$, so $h=g \circ \pi$. It is straightforward to check the conditions in $\eqref{eq:equivalentCondition}$ to hold for $g$, which is equivalent to $(E, G)$ being a fixed point of $w_d$; there are conditions: (a) $g(E[d]/G[d])=G[d]$ and (b) $g(G/G[d])=G[N/d]$. For (a): $g(E[d]/G[d])=h(E[d])$, and $\ker(h^2)=E[d]$ forces $h(E[d])\subseteq\ker h=G[d]$; since $|h(E[d])|=|E[d]|/|G[d]|=d^2/d=d=|G[d]|$, equality holds. For (b): $h(G[d])=0$, so $h(G)=h(G[N/d])$, which is injective of order $N/d$ and contained in $G$; as $G$ is cyclic it has a unique subgroup of that order, namely $G[N/d]$, so $g(G/G[d])=h(G)=G[N/d]$.
\end{proof}

\begin{lemma}\label{lem:differByUnit}
Let $x=(E, G)$ be a $w_d$-fixed point in $X_0(N)(\overline{\F}_\ell)$ and put $\orderO=\End(E, G)$. Suppose that $a, b\in \orderO$ are obtained from two choices of an isomorphism $w_d(x) \cong x$. Then there is a unique $u \in \orderO^\times$ such that $b=ua$. Conversely, if $a$ comes from one such choice and $u\in \orderO^\times$, then $ua$ comes from another such choice.
\end{lemma}

\begin{proof}
    By the proof of \cref{lem:wdInduceEnd}, both $a$ and $b$ factor through the same quotient $\pi : E \to E/G[d]$, and both have kernel $G[d]$. We can write $a = g_a \circ \pi$, $b=g_b \circ \pi$, where $g_a, g_b\colon E/G[d] \to E$ are isomorphisms satisfying the conditions \eqref{eq:equivalentCondition} defining an isomorphism $w_d(x)\cong x$. Hence $u=g_b \circ g_a^{-1} \in \Aut(E,G)$, i.e.~$u\in \orderO^\times$, and $b=ua$. Since $a$ is an isogeny, uniqueness follows. Conversely, composing an isomorphism $w_d(x)\cong x$ with an element \( u \in \orderO^\times=\Aut(E, G)\) gives another such isomorphism (as \(u\) preserves both \(G[d]\) and \( G[N/d]\)).
\end{proof}
\begin{remark}
In other words, the set of isomorphisms $w_d(x)\cong x$ is a torsor under $\orderO^\times$ (acting by post-composition), and the map sending such an isomorphism to its endomorphism $h$ is $\orderO^\times$-equivariant and injective; hence the resulting $h$ form a single free $\orderO^\times$-orbit.
\end{remark}

\begin{lemma}\label{lem:traces}
Let $\ell, M, N, d$ be as in \cref{def:SS_fixed_points}, and let $x=(E,G)$ be a supersingular $w_d$-fixed point of $X_0(N)_{\overline{\F}_\ell }$ and $\orderO=\End(E,G)$. After choosing an isomorphism $w_d(x)\cong x$, let $h\in \orderO$ be the endomorphism given by \cref{lem:wdInduceEnd}. Then $\nrd(h)=d$. If $T=\trd(h)$, then $d\mid T$ and $T^2 < 4d$. Consequently:
\begin{enumerate}
    \item If $d > 3$, then $T=0$ and $h^2=-[d]$.
    \item If $d=3$, then $T\in\{ 0, \pm 3\}$. In all cases there is an element $t\in \orderO$ such that $t^2=-[3]$.
    \item If $d=2$, then $T\in\{0, \pm 2\}$. More precisely, if $T=0$, then $h^2=-[2]$; if $T=2$, then $(h-1)^2=-1$; and if $T=-2$, then $(h+1)^2=-1$.
\end{enumerate}
\end{lemma}
\begin{proof}
    We first prove the general statements. By \Cref{lem:wdInduceEnd}, after choosing
an isomorphism \(w_d(x)\cong x\), we obtain an element \(h\in \orderO\) and an
automorphism \(\phi\in\Aut(E)\) such that \(h^2=\phi\circ[d]\). Since the dual
$\hat{}$ induces the standard involution, \(\operatorname{nrd}(h)=\deg h=d\).
Write \(T=\trd(h)\). Then   
\(h^2-Th+d=0\), and therefore \(\phi=h^2/d=(T/d)h-1\). Taking reduced traces gives \(\trd(\phi)=T^2/d-2\). Since \(\phi\) is an element of an order \(\End(E)\), we get \( \trd(\phi)\in\Z \). As \(d\) is squarefree,
this implies \(d\mid T\). Since \(B_{\ell,\infty}\otimes_{\Q}\R \cong \bbH\), a
non-scalar element of reduced norm \(d\) and reduced trace \(T\) satisfies
\(T^2<4d\). Indeed, writing \(h=a+v\), with \(a\in\R\) and \(v\) purely
quaternionic, one has \(T=2a\) and \(d=a^2+\|v\|^2\), and the inequality is
strict unless \(v=0\). In our situation \(h\) cannot be scalar, because then
its reduced norm would be a rational square, but \(\operatorname{nrd}(h)=d\)
and \(d>1\) is squarefree. Hence \(d\mid T\) and
\(T^2<4d\).

\emph{Case (i), \(d>3\).} These conditions force \(T=0\) and thus \(h^2=-[d]\). (Hence $\phi = h^2/d = [-1]$.)

\emph{Case (ii), $d=3$.} Then $T\in\{0, \pm 3\}$. If $T=0$, again $h^2=-[3]$. If $T=3$, then $h^2-3h+3=0$, so $(2h-3)^2=-3$. If $T=-3$, then $h^2+3h+3=0$ and $(2h+3)^2=-3$. In all cases there is $t\in \orderO$ such that $t^2=-[3]$.

\emph{Case (iii), $d=2$.} This case is analogous. Now $T\in \{0, \pm 2\}$. If $T=0$, then $h^2=-[2]$. If $T=2$, then $h^2-2h+2=0$, so $(h-1)^2=-1$. If $T=-2$, then $h^2+2h+2=0$, so $(h+1)^2=-1$.
\end{proof}

\begin{remark}
    We note here that, for \( d=3 \), despite getting \(t\in \End(E, G)\) only after potentially modifying \( h\), since \( t^2 = -1 \circ [3]\), the converse direction of \cref{lem:wdInduceEnd} applies. This means that, for another embedding of \( \sqrt{-3}\) into \( \End(E, G) \), realised by \( s\in \End(E, G)\), by \cref{lem:differByUnit}, we can conclude \( s=ut\) for some unit \( u\).
\end{remark}

\begin{remark}
We will use the following terminology. Let \(x=(E,G)\) be a \(w_d\)-fixed supersingular point, and put \(\orderO=\End(E,G)\). An \(\orderO^\times\)-conjugacy class in the relevant embedding problem is said to lie above \(x\) if it has a representative whose image is obtained from some choice of an isomorphism \(w_d(x)\cong x\), as follows.

Let \( h \in \orderO \) be the element obtained by \cref{lem:wdInduceEnd} from this choice. If \( d > 3\), then by \cref{lem:traces}\((i)\) we have \(h^2  = -d\), and the embedding \( \Z[\sqrt{-d}] \hookrightarrow \orderO\) is sending \( \sqrt{-d} \mapsto h\). If \( d=3\), then by \cref{lem:traces}\((ii)\) we replace \( h\) by \( t = h\) if \( \trd(h)=0\), by \( t=2h-3\) if \( \trd(h)=3\), and by \( t = 2h+3 \) if \( \trd(h)=-3\). In each case, \(t^2 = -3\), and the embedding \( \Z[\sqrt{-3}] \hookrightarrow \orderO\) is sending \( \sqrt{-3} \mapsto t\).

Similarly, if \( d=2 \),  we have two cases. If \(\trd(h)=0\), then \( \sqrt{-2} \mapsto h\),  while if \( \trd(h) = \pm 2\), then \( t=h\mp 1\), and \( \i \mapsto t\).

An embedding in the relevant embedding problem is said to \emph{lie above} \( x\) if it can be obtained from \( x\), using this construction. A conjugacy class of embeddings is said to \emph{lie above} \(x\) if it contains an embedding lying above \( x\).

Note that, for \( d > 2\), we always get an embedding of \(\sqrt{-d}\), while for \(d=2\), this construction leads to embeddings of either \(\sqrt{-2}\) or \( \i\).
\end{remark}

%Since different choices of an endomorphism \(h\) differ by elements of \( \Aut(E, G)\), to solve our embedding problems, we need to classify these groups.

Starting from the classification of all possibilities of $\Aut(E)$ from \Cref{prop:aut_classification} we prove a classification for the possibilities of $\Aut(E,G)$ that covers all squarefree $N=\ell M> 6$. Note that assuming $N > 6$ removes the exceptional cases cleanly: every squarefree $N>6$ has a prime factor $\ge5$, so if $\ell\le3$ that factor lies in $M$, and the hypothesis ``$\ell>3$ or $q>3\mid M$'' always holds. %Hence $\Aut(E,G)$ is cyclic of order $2,4,6$ for all $N>6$. For $N\in\{2,3,6\}$ the unit group can be $Q_8,2T,2D_6$ and ``exactly two classes'' would need separate care. The formulae we obtain are correct for all squarefree N.

\begin{proposition}
\label{prop:unit_groups_cases}
Let $E$ be an elliptic curve over $\overline{\F}_\ell $, and let $G \subseteq E(\overline{\F}_\ell )$ be a cyclic subgroup of order $M$ with $\ell\nmid M$. If either $\ell>3$ or $M$ is divisible by a prime $q>3$, then $\Aut(E,G)$ is cyclic of order $2,4$ or $6$.
\end{proposition}
\begin{proof}
Note that $-1 \in \Aut(E,G)$, so that the case $\ell>3$ as well as $j\neq 0$ with $\ell=2,3$ immediately follows from \Cref{prop:aut_classification}.
So now the case where we have $q \mid M$ with $q$ a prime $>3$. Additionally from now on $\ell=2$ or $3$, $j=0$ and $\Aut(E)$ is either $\SL_2(\F_3)$ if $\ell=2$ or $\Z/3\Z \rtimes \Z/4\Z$ if $\ell=3$. Note that $\Aut(E,G) \subseteq \Aut(E,G[q])$, so it suffices to prove the statement in the case $G=G[q]$ and $M=q$.  

We claim that $\Aut(E)$ acts faithfully on $E[q](\Fbar_\ell)$. Indeed, let $V_q(E) \colonequals \varprojlim_n E[q^n](\Fbar_\ell ) \cong \Z_q^2$ be the $q$-adic Tate module, then it is well known that \( \Aut(E) \) acts faithfully on $V_q(E)$. 
However, the reduction mod $q$ map $\GL_2(\Z_q) \to \GL_2(\F_q)$ is injective when restricted to elements of finite order coprime to $q$. So by \Cref{prop:aut_classification} we get an injection $\Aut(E) \to \Aut(E[q](\Fbar_\ell ))$ and the claim follows.

Now as $\Aut(E)$ respects the Weil pairing, after a choice of basis of the $q$ torsion, we get that  $\Aut(E) \subseteq \SL_2(\F_q)$. Since $\Aut(E,G)$ fixes $G$ we can choose a basis such that
\[
 \Aut(E,G) \subseteq U \colonequals \left\{ \begin{pmatrix} a & b \\ 0 & a^{-1} \end{pmatrix} : a \in \F_q^{\times},\; b \in \F_q \right\} \subseteq \mathrm{SL}_2(\F_q).
\]
Note that $T \colonequals \begin{pmatrix} 1 & 1 \\ 0 & 1 \end{pmatrix}$ generates a normal subgroup of order $q$ in $U$. Since $\Aut(E,G)$ has no elements of order $q$ we get an injection  $\Aut(E,G) \hookrightarrow U/\langle T \rangle \cong \F_q^\times$. This shows that $\Aut(E,G)$ is cyclic. And the proposition follows since $\Z/2\Z, \Z/4\Z$ and $\Z/6\Z$ are the only cyclic subgroups of $\Aut(E)$ containing $[-1]$.
\end{proof}

\begin{corollary}\label[corollary]{cor:cyclic_two_classes}
Let \(\orderO\) be an order in a quaternion algebra over \(\Q\), and assume \( \orderO^\times \cong C_2, C_4, C_6\). Let \( t\in \orderO\) satisfy \( t^2\in \Q^\times\), and put \( A = \{ u \in \orderO^\times: (ut)^2=t^2 \}\).

Then the elements \(ut\), with \( u \in A\), split into exactly two \( \orderO^\times \)-conjugacy classes.
\end{corollary} 

\begin{proof}
For \(u\in \orderO^\times\), the condition \((ut)^2=t^2\) is equivalent to
\(tut^{-1}=u^{-1}\). %(note: \(t^{-1}\) exists in \(B^\times_{\ell,\infty}\)). Indeed, from \(utut=t^2\), left-multiplying by \(u^{-1}\) gives \(tut=u^{-1}t^2\), hence \(tu=u^{-1}t\) and \(tut^{-1}=u^{-1}\). 
Thus \(A\) is a subgroup of the cyclic group \(\orderO^\times\). Since \( \{\pm 1\}\subseteq A \), either \( A=\{\pm1\} \)  or \( A = \orderO^\times\).

If \( A=\orderO^\times \), then for \( u,v\in \orderO^\times \), the condition for \( v^{-1} \) gives \(tv^{-1}=vt\), so
\(v(ut)v^{-1}=vu(tv^{-1})=vu(vt)=v^2ut\), since \(\orderO^\times\) is abelian. Thus
conjugating \(ut\) by \(v\) corresponds to replacing \(u\) by \(v^2u\), and the set \( \{ ut : u \in A\} \) is stable under \( \orderO^\times\)-conjugacy. The
\(\orderO^\times\)-orbit of \(ut\) is therefore
\(\{v^2ut:v\in \orderO^\times\}=(\orderO^\times)^2ut\). Hence \(ut\) and \(u't\) are
conjugate if and only if \(u'u^{-1}\in (\orderO^\times)^2\). This gives the bijection between conjugacy classes of elements \(ut\) and
\(\orderO^\times/(\orderO^\times)^2\), which has order \(2\) as \(\orderO^\times\) is cyclic of even order.

For the case \( A = \{\pm 1\} \), only \(t\) and \(-t\) occur. Suppose \(vtv^{-1}=-t\) 
for some \(v\in \orderO^\times\). Then \(vt=-tv\) so \(t\) and \( v\) generate the entire quaternion algebra. Furthermore, \(v^2t=-vtv=tv^2\), so \(v^2\) commutes with both \(v\) and \(t\), hence is
central, so \(v^2\in\Q^\times\cap \orderO^\times=\{\pm1\}\). Since \(v=\pm1\) cannot
anticommute with \(t\), we must have \(v^2=-1\). Thus \(v^{-1}=-v\), and
\(vt=-tv\) gives \(tvt^{-1}=-v=v^{-1}\), so \(v\in A\), contradicting
\(A=\{\pm1\}\). Hence \(t\) and \(-t\) are not conjugate, and again there are
exactly two classes.
\end{proof}

\begin{remark}
In other words, if \(A=\orderO^\times\), there is a bijection \[\{\orderO^\times\text{-conj.\ classes of }ut\}\to \orderO^\times/(\orderO^\times)^2, \quad [ut]\mapsto u\,(\orderO^\times)^2.\]
The map is well-defined and injective precisely because the proof shows the conjugates of $ut$ are exactly $\{v^2ut:v\in \orderO^\times\}=(\orderO^\times)^2\,ut$.
\end{remark}

\begin{lemma}\label{lem:d_gt_3}
Assume \(N>6\). Let \(x=(E,G)\) be a \(w_d\)-fixed supersingular point with \(d>3\), and put
\(\orderO=\End(E,G)\). Then exactly two \(\orderO^\times\)-conjugacy classes of embeddings
\(\Z[\sqrt{-d}]\hookrightarrow \orderO\) lie above \(x\).
\end{lemma}

\begin{proof}
Choose one isomorphism \(w_d(x)\cong x\), and let \(t\in \orderO\) be the corresponding endomorphism from \cref{lem:wdInduceEnd}. By \cref{lem:traces}\((i)\), we have
\(t^2=-d\).

By \cref{lem:differByUnit}, the endomorphisms obtained from all choices of an
isomorphism \(w_d(x)\cong x\) are precisely the elements \(ut\), with
\(u\in \orderO^\times\). Conversely, every such \(ut\) comes from such a choice.
Applying \cref{lem:traces}\((i)\) to the choice corresponding to \(ut\), we get
\((ut)^2=-d=t^2\). Thus, in the notation of \cref{cor:cyclic_two_classes}, we
have \(A=\orderO^\times\).

Since \(\orderO^\times\cong C_2,C_4\), or \(C_6\) by \cref{prop:unit_groups_cases} (as \(N>6\)),
\cref{cor:cyclic_two_classes} applies and shows that the elements \(ut\)
obtained from the choices split into exactly two \( \orderO^\times\)-conjugacy classes.
Equivalently, exactly two \( \orderO^\times\)-conjugacy classes of embeddings
\( \Z[\sqrt{-d}]\hookrightarrow \orderO\) lie above \(x\).
\end{proof}

\begin{lemma}\label{lem:d_eq_3}
Assume \(N>6\). Let \(x=(E,G)\) be a \(w_3\)-fixed supersingular point, and put \(\orderO=\End(E,G)\). Then exactly two
\(\orderO^\times\)-conjugacy classes of embeddings
\(\Z[\sqrt{-3}]\hookrightarrow \orderO\) lie above \(x\).
\end{lemma}

\begin{proof}
Choose one isomorphism \(w_3(x)\cong x\), and let \(h\in \orderO\) be the corresponding
endomorphism from \cref{lem:wdInduceEnd}. By \cref{lem:traces}\((ii)\), after
replacing \(h\) by \(h\), \(2h-3\), or \(2h+3\), according as
\(\trd(h)=0\), \(3\), or \(-3\), we obtain an element \(t\in \orderO\) with
\(t^2=-3\).

Since \(t^2=(-1)\circ[3]\), the converse direction of
\cref{lem:wdInduceEnd} shows that \(t\) itself comes from a choice of an
isomorphism \(w_3(x)\cong x\). Now let \(s\in \orderO\) be any element representing
an embedding \(\Z[\sqrt{-3}]\hookrightarrow \orderO\) lying above \(x\). Then
\(s^2=-3=t^2\), and the same converse direction of \cref{lem:wdInduceEnd}
shows that \(s\) also comes from a choice of an isomorphism \(w_3(x)\cong x\).
Therefore \cref{lem:differByUnit} gives \(s=ut\) for some \(u\in \orderO^\times\).
Moreover \(s^2=t^2\), so \(u\in A\), where
\(A=\{u\in \orderO^\times:(ut)^2=t^2\}\).

Conversely, if \(u\in A\), then \((ut)^2=t^2=-3\), so the converse direction of
\cref{lem:wdInduceEnd} shows that \(ut\) comes from a choice of an isomorphism
\(w_3(x)\cong x\). Thus the embeddings lying above \(x\) are represented
exactly by the elements \(ut\), with \(u\in A\).

Since \(\orderO^\times\cong C_2,C_4\), or \(C_6\) by \cref{prop:unit_groups_cases} (as \(N>6\)), \cref{cor:cyclic_two_classes}
applies and gives exactly two \(\orderO^\times\)-conjugacy classes.
\end{proof}

For \(d=2\), even though embeddings of both \(\sqrt{-2}\) and \( \i\) can occur, for a single fixed point, only one type occurs.
\begin{lemma}\label{lem:d_eq_2}
Assume \(N>6\). Let \( x=(E, G)\) be a \(w_2\)-fixed supersingular point, put \( \orderO=\End(E,G) \).
Above \(x\), there are either exactly two \( \orderO^\times\)-conjugacy classes of embeddings \( \Z[\sqrt{-2}]\hookrightarrow \orderO\) or exactly two \(\orderO^\times\)-conjugacy classes of embeddings \( \Z[\i] \hookrightarrow \orderO\), but not both.
\end{lemma}

\begin{proof}
By \cref{lem:traces}\((iii)\), every choice \(h\in \orderO\) has trace \(0\) or
\(\pm2\). In the first case \(h^2=-2\), while in the second case
\( (h\mp1)^2=-1\).

We first show that the trace-zero and trace-nonzero choices cannot coexist. Suppose that
\(h\) is a trace-nonzero choice, and put \(i=h\mp1\), so \(i^2=-1\).
Since \(\orderO^\times\) is cyclic of order \(2,4\), or \(6\) by \cref{prop:unit_groups_cases} (as \(N>6\)), this forces
\(\orderO^\times=\langle i\rangle\). If \(h_0\) were a trace-zero choice, then by
\cref{lem:differByUnit} the choices would be \(h_0,-h_0,ih_0,-ih_0\). For \(ih_0\), if \(\trd(ih_0)=\pm2\), then \(ih_0\mp1\) is an order-four
unit, hence equals \(\pm i\), so \(h_0\in\Q(i)\), contradicting \(h_0^2=-2\).
Thus \(ih_0\), and hence also \(-ih_0\), have trace zero. Therefore all choices
would have trace zero, contradicting the existence of \(h\). Hence all
choices have the same trace type.

If all choices have trace zero, choose one of them, say \(h_0\). Then the choices
are precisely the elements \(uh_0\), with \(u\in \orderO^\times\), and all satisfy
\((uh_0)^2=-2=h_0^2\). Thus \cref{cor:cyclic_two_classes} gives exactly two
\(\orderO^\times\)-conjugacy classes of embeddings
\(\Z[\sqrt{-2}]\hookrightarrow \orderO\) lying above \(x\).

If all choices have trace \(\pm2\), choose one of them, say \(h\), and put
\(i=h\mp1\). Then \(\orderO^\times=\langle i\rangle\), and the only square roots
of \(-1\) in \(\orderO^\times\) are \(i\) and \(-i\). Both occur, from \(h\) and
\(-h\), and they are not \(\orderO^\times\)-conjugate because \(\orderO^\times\) is
abelian. Exactly two \(\orderO^\times\)-conjugacy classes of embeddings
\(\Z[i]\hookrightarrow \orderO\) lie above \(x\).
\end{proof}

\begin{proof}[Proof of \Cref{thm:2to1}]
To finish the proof of \Cref{thm:2to1}, we show that every embedding lies above its point. Let $x=(E,G)$ be a supersingular point and put $\orderO=\End(E,G)$.
\begin{itemize}
  \item If $s\in\orderO$ satisfies $s^2=-d$, then $s^2=(-1)\circ[d]$, so by the converse direction of \Cref{lem:wdInduceEnd} the point $x$ is $w_d$-fixed and $s$ is the endomorphism attached to a choice of isomorphism $w_d(x)\cong x$. Hence the embedding $\sqrt{-d}\mapsto s$ lies above $x$.
  \item If $d=2$ and $s\in\orderO$ satisfies $s^2=-1$, put $h\colonequals 1+s$. Then $h^2=2s=s\circ[2]$ with $s\in\Aut(E)$, so again $x$ is $w_2$-fixed and $h$ comes from a choice of isomorphism. As $\trd(h)=2$, the construction above gives $t=h-1=s$, so the embedding $\i\mapsto s$ lies above $x$.
\end{itemize}
In particular, $\orderO$ admits no such embeddings if $x$ is not $w_d$-fixed. Together with \Cref{lem:d_gt_3,lem:d_eq_3,lem:d_eq_2}, the number of $\orderO^\times$-conjugacy classes of embeddings into $\End(E,G)$ is therefore $2$ if $x$ is $w_d$-fixed, and $0$ otherwise. Summing over all supersingular points, equivalently by \Cref{thm:Deuring}~(ii) over $[I]\in\Cls\orderO_0$, gives
\[
  2\,\nu_\ell(N,d)=\sum_{[I]\in\Cls\orderO_0}\#\left\{O_L(I)^\times\text{-conjugacy classes of embeddings into }O_L(I)\right\}. \qedhere
\]
\end{proof}

Let $h(D)$ be the
class number of the quadratic order with discriminant $D$ (also the size of the quadratic form class group $C(D)$, see \cite[§7.B]{Cox}).

\subsection{Formulae for \texorpdfstring{$\nu_{\ell}(N,d)$}{nup(N,d)}}
We now turn the description of $w_d$-fixed supersingular points into explicit formulae for $\nu_{\ell}(N,d)$. By \Cref{thm:2to1}, the quantity $\nu_{\ell}(N,d)$ is obtained by counting embeddings of suitable quadratic orders into Eichler orders in $B_{\ell,\infty}$. The global count is given by Eichler's trace formula as a product of local factors (and class numbers). For $d \equiv 3 \pmod 4$, one must distinguish between the optimal and non-optimal embeddings, i.e., between the orders of discriminant $-4d$ and $-d$ in $\Q(\sqrt{-d})$. Among the local factors, the only additional computation needed is the one at $2$. All the local factors we need can be computed using results from §30.5 (for maximal orders) and §\,30.6 (for non-maximal Eichler orders) of Voight's book \cite{Voight2021}. However, as the author notes at the beginning of §\,30.6, ``the presentation (\ldots) begins to run off the rails''.  What follows is a self-contained proof of this additional computation in the more restricted setup we need, so that we can stay on track.
We still follow notation from \cite[§\,30]{Voight2021}. Throughout the subsection, $K=\Q(\sqrt{-d})$ and $N>6$, as in \Cref{thm:2to1}. %\orderO_\Delta order of discr \Delta in K

\begin{lemma}\label{lem:local_factor_2_merged}
Let \(\ell \) be an odd prime, and let \(\orderO\) be an Eichler order in \(B_{\ell,\infty}\)
of squarefree even level. Put \(\orderO_2 \colonequals \orderO\otimes\Z_2\). Since
\(B_{\ell,\infty}\otimes\Q_2\simeq M_2(\Q_2)\), after conjugation we may assume
\[
\orderO_2=
\left\{
\begin{pmatrix} a & b \\ 2c & e \end{pmatrix}
: a,b,c,e\in\Z_2
\right\}.
\]
Let \(d\) be an odd squarefree integer, let \(K=\Q(\sqrt{-d})\), and let
\(S=\Z[\sqrt{-d}]\) be the quadratic order of discriminant \(-4d\) in \(K\).
Put \(S_2\colonequals S\otimes\Z_2=\Z_2[\sqrt{-d}]\).  Then the number of
\(\orderO_2^\times\)-conjugacy classes of optimal embeddings is
\[ m\left(S_2,\orderO_2;\orderO_2^\times\right)
=
\begin{cases}
1, & d\equiv 1 \pmod 4,\\
2, & d\equiv 3 \pmod 4.
\end{cases}
\]
\end{lemma}

\begin{proof}
An embedding is determined by \(E\colonequals\iota(\sqrt{-d})\in \orderO_2\), which must
satisfy \(E^2=-dI\). Since \(\tr(E)=0\), we
may write
\(
E=\begin{pmatrix} a & b \\ 2c & -a \end{pmatrix}\), \(a,b,c\in\Z_2.
\)
Since \(E^2=(a^2+2bc)I\),
\begin{equation}\label{eq:Aeq}
a^2+2bc=-d.
\end{equation}
Since \(d\) is odd, \eqref{eq:Aeq} forces \(a\) to be odd.

If \(d\equiv 3\pmod 4\), then the maximal order is
\(\Z_2 \left[\frac{1+\sqrt{-d}}{2}\right]\). Our embedding extends to this
maximal order if and only if \(\frac12(I+E)\in \orderO_2\). Since
\[
\frac12(I+E)=
\begin{pmatrix}
(1+a)/2 & b/2\\
c & (1-a)/2
\end{pmatrix}
\]
and \(a\) is odd, this happens exactly when \(b\) and \(c\) are both even.
Hence, for \(d\equiv 3\pmod 4\),
\begin{equation}\label{eq:optimal}
\iota \text{ is optimal}\iff b,c \text{ are not both even.}
\end{equation}

Any
\(U=\left(\begin{smallmatrix}\alpha&\beta\\2\gamma&\delta\end{smallmatrix}\right)\in \orderO_2^\times\)
has \(\det U\equiv\alpha\delta\pmod 2\) a unit, so \(\alpha,\delta\) are units and
\(\bar U=\left(\begin{smallmatrix}1&\bar\beta\\0&1\end{smallmatrix}\right)\)
in \(M_2(\F_2)\). Since
\(\bar E=\left(\begin{smallmatrix}1&\bar b\\0&1\end{smallmatrix}\right)\),
and unipotent upper-triangular matrices over \( \F_2\) commute, \(\bar U^{-1}\bar E\bar U=\bar E\). Thus \(\bar b\in\F_2\) is an invariant of
the \(\orderO_2^\times\)-conjugacy class of \(E\).
We can realise our two embeddings as
\(E_1\colonequals\left(\begin{smallmatrix}1&1\\-(d+1)&-1\end{smallmatrix}\right)\) and
\(E_2\colonequals\left(\begin{smallmatrix}1&-(d+1)/2\\2&-1\end{smallmatrix}\right)\). Since \( d\equiv 3 \pmod{4}\), these two matrices are not \(\orderO_2^\times\)-conjugates, and the goal is now to show that every other embedding can be \(\orderO_2^\times\)-conjugated into one of these two.
To this end, we use the following elements of \(\orderO_2^\times\) and their
conjugation actions on \(\left(\begin{smallmatrix} a & b \\ 2c & -a\end{smallmatrix}\right) \):
\[
\begin{pmatrix}1&0\\2w&1\end{pmatrix}:\ a\mapsto a+2bw,\quad
\begin{pmatrix}1&u\\0&1\end{pmatrix}:\ a\mapsto a-2cu,\quad
\begin{pmatrix}1&0\\0&t\end{pmatrix}:\ (b,c)\mapsto(tb,t^{-1}c),
\]
the first leaving \(b\) fixed, the second leaving \(c\) fixed, the third leaving
\(a\) fixed, where \(w,u\in\Z_2\) and \(t\in\Z_2^\times\).
\begin{comment}This comment is intentionally left here.
Explicitly,
\[
\begin{pmatrix}1&0\\2w&1\end{pmatrix}^{-1}
E
\begin{pmatrix}1&0\\2w&1\end{pmatrix}
=
\begin{pmatrix}
a+2bw & b\\
2(c-2aw-2bw^2) & -a-2bw
\end{pmatrix}.
\]
\[
\begin{pmatrix}1&u\\0&1\end{pmatrix}^{-1}
E
\begin{pmatrix}1&u\\0&1\end{pmatrix}
=
\begin{pmatrix}
a-2cu & b+2au-2cu^2\\
2c & -a+2cu
\end{pmatrix}.
\]
\[
\begin{pmatrix}1&0\\0&t\end{pmatrix}^{-1}
E
\begin{pmatrix}1&0\\0&t\end{pmatrix}
=
\begin{pmatrix}
a & tb\\
2t^{-1}c & -a
\end{pmatrix}.
\]
\end{comment}
If \(b\) is a unit, since \(a\) is odd, we can take \(w=(1-a)/(2b)\in\Z_2\) to reduce to \(a=1\), then
choose \(t=b^{-1}\) to reduce to \(b=1\). Now \eqref{eq:Aeq} forces
\(2c=-(d+1)\), so \(
E\sim E_1.\)
If instead \(c\) is a unit, we can take \(u=(a-1)/(2c)\in\Z_2\) to reduce to
\(a=1\), then choose \(t=c\) to reduce to \(c=1\). Now \eqref{eq:Aeq} forces
\(b=-(d+1)/2\), so
\( E\sim E_2 \).

Reducing \eqref{eq:Aeq} modulo \(4\), and using
\(a^2\equiv 1\pmod 4\), gives
\( 2bc\equiv -d-1\pmod 4 \).

If \(d\equiv 3\pmod 4\), then \(2bc\equiv 0\pmod 4\), so \(bc\) is even. For an
optimal embedding, \eqref{eq:optimal} rules out \(b,c\) both even, so exactly
one of \(b,c\) is a unit. By the argument above, every optimal embedding
is conjugate to either \(E_1\) or \(E_2\). Since both are optimal and satisfy
\eqref{eq:Aeq}, there are exactly two classes.

For \( d \equiv 1 \pmod{4}\), an analogous argument (\(b\) and \( c\) are both units) shows every embedding is conjugate to \( E_1\). Alternatively, \(S_2\) is already the maximal order of \(K\otimes\Q_2\), so every embedding is optimal. Since \( 2\) ramifies in \( K\), by \cite[Lemma 30.6.16]{Voight2021}, \( m(S_2, \orderO_2; \orderO_2^\times) = 1 + \left(\frac{K}{2}\right)= 1 \).
\end{proof}

\begin{remark}
Provided $d$ is not an odd number congruent to $7$ mod $8$, the previous lemma also follows from \cite[2.3 Theorem]{Hijikata}, by taking, in the notation of that paper, $k=\Q_2, r=\Z_2, p=2, B \cong M_2(k), v = 1, s = 0, n = d, \rho = 0$, and $g=\sqrt{-d}\not\in k$. The condition on $d$ is necessary since Hijikata's theorem assumes $g\not\in k$, but if $-d\equiv 1 \pmod{8}$, then $-d$ is a square in $\Q_2$. The matrices we used in the proof are also present in \cite{Hijikata}.
\end{remark}

\begin{proposition}\label{prop:nuell_2M_d_odd}
Let $M'$ be odd and squarefree, let $\ell$ be an odd prime with $\ell\nmid M'$, and put $N=2\ell M'$.
Let $d > 1$ be an odd divisor of $N$, and set $d_0=d/\gcd(d,\ell)$.
Then
\[
\nu_{\ell}(N,d)
=
\frac12\left(1-\left(\frac{-d}{\ell}\right)\right)
\prod_{q\mid M'/d_0}\left(1+\left(\frac{-d}{q}\right)\right)\cdot \kappa(d),
\]
where
\[
\kappa(d)=
\begin{cases}
h(-4d), & d\equiv 1 \pmod 4,\\[4pt]
2h(-4d) \, + \, \left(1+\left(\frac{-d}{2}\right)\right)\,h(-d), & d\equiv 3 \pmod 4.
\end{cases}
\]
\end{proposition}
\begin{proof}
Fix an Eichler order \( \orderO_0\) of level \( M=2M'\) in \(B_{\ell,\infty}\). By Deuring's correspondence (\Cref{thm:Deuring}), the relevant endomorphism rings \( \End(E, G)\), as \( x = (E, G)\) ranges over supersingular points, are exactly the left orders \( O_L(I)\) for \( [I] \in \Cls \orderO_0\). By \Cref{thm:2to1}$(i)$,
\[ \nu_{\ell}(N,d) = \frac 12 \sum_{[I]\in\Cls \orderO_0} \#\left\{O_L(I)^\times\text{-conjugacy classes of embeddings } \Z[\sqrt{-d}]\hookrightarrow O_L(I)\right\}.\]

For a discriminant $\Delta$ of an order in $K$, write $S_\Delta$ for the quadratic order of discriminant $\Delta$. When $d\equiv 1\pmod{4}$, $\Z[\sqrt{-d}]$ is already the maximal
order of $K$, so every such embedding is an optimal embedding of the
order of discriminant $\Delta=-4d$.  When $d\equiv 3\pmod{4}$, $\Z[\sqrt{-d}]$ has index $2$ in the
maximal order $S_{-d}$, and every embedding
$\iota\colon \Z[\sqrt{-d}]\hookrightarrow O$ falls into exactly one
of two cases: either $\iota(K)\cap O = \iota(\Z[\sqrt{-d}])$, in
which case $\iota$ is an optimal embedding of $S_{-4d}$,
or $\iota(K)\cap O = \iota(S_{-d})$, in which case $\iota$
extends uniquely to an optimal embedding of $S_{-d}$.
The total count of embeddings of $\Z[\sqrt{-d}]$ is therefore the
sum of the optimal embedding counts for $\Delta\in\{-4d,-d\}$. 

By the trace formula for optimal embeddings
\cite[Thm.~30.4.7, Thm.~30.7.3]{Voight2021} (going back to Eichler, see, e.g.
\cite[II,~\S\,4,~Prop.~5]{Eichler1973}, see also \cite[Example 30.7.4]{Voight2021}), the number of optimal
embedding classes of $S_\Delta$ into an Eichler order of
level $M$ equals $h(\Delta)$ times the local factors.
Each local factor at an unramified odd prime follows by
\cite[Lemma~30.6.17(a)]{Voight2021}, while the local factor at \(\ell \) follows by \cite[Prop.~30.5.3(b)]{Voight2021}. The factor at $2$ for 
$\Delta=-d$ is \cite[Lemma~30.6.16]{Voight2021}; and the factor at $2$ for $\Delta=-4d$ is computed in 
\Cref{lem:local_factor_2_merged}.
Explicitly, the local factors are:
\begin{itemize}
\item $1-\left(\frac{-d}{\ell}\right)$ at the ramified prime $\ell$
(equal to $1$ if $\ell\mid d$);
\item $1+\left(\frac{-d}{q}\right)$ at each odd prime
$q\mid M'/d_0$;
\item at $2$: the factor from \Cref{lem:local_factor_2_merged}
for $\Delta=-4d$, and $1+\left(\frac{-d}{2}\right)$ for $\Delta=-d$.
%(where $2$ does not divide the conductor of $S_{-d}$).
\end{itemize}
Multiplying and summing over the relevant discriminants gives
\[
h(\Delta)
\left(1-\left(\tfrac{-d}{\ell}\right)\right)
\prod_{q\mid M'/d_0}\left(1+\left(\tfrac{-d}{q}\right)\right)
\times
\begin{cases}
1, & \Delta=-4d,\; d\equiv 1\pmod{4},\\[2pt]
2, & \Delta=-4d,\; d\equiv 3\pmod{4},\\[2pt]
1+\left(\frac{-d}{2}\right), & \Delta=-d.
\end{cases}
\]
Distinguishing by $\Delta=-d$ and $\Delta=-4d$ yields the stated formula.
\end{proof}

\begin{definition}\label{def:modified_classnumber}
    Let $d>0$ be a squarefree odd integer. Then the \emph{modified class number} $h'(-d)$ is defined to be 
        \begin{align*}
h'(-d) \colonequals
\begin{cases}
 h(-4d) &\text{if } d \equiv 1 \pmod 4,\\
h(-4d)+h(-d) &\text{if } d \equiv 3 \pmod 4.
\end{cases}
\end{align*}
\end{definition}

We introduce this notation because this quantity appears in many of the formulae below; it allows us to state and prove them without distinguishing between the residue classes of $d$ modulo $4$.

\begin{remark}\label{rem:eichler_analogy}
We note here the similarity of our formulae for \( \nu_{\ell}(N, d)\) with the formulae for \( \nu(N, d)\), the number of fixed points of \( w_d\) on \(X_0(N)\) in characteristic \( 0\). The formulae for \( \nu(N, d)\) can be interpreted as counting conjugacy classes of embeddings of quadratic orders into the Eichler order in the split algebra \( M_2(\Q)\), while our formulae for \( \nu_{\ell}(N, d)\) are obtained by counting conjugacy classes of embeddings of quadratic orders into Eichler orders in the division algebra \( B_{\ell, \infty}\), which changes the local factor at \(\ell \), while the factors away from \(\ell \) agree. The global factor \( 1/2\) comes from the \(2:1\) correspondence between embedding classes and fixed points (\cref{thm:2to1}).

This similarity is not accidental: both counts are instances of Eichler's local--global formula for optimal embeddings. To make this analogy precise, we take a formula for odd \(N\) from \cite{Klu77}:
\[ \nu(N,d) = h'(-d)\prod_{q\mid N/d}\left(1+\left(\frac{-d}{q}\right)\right)\]
and notice that this is exactly the number of conjugacy classes of embeddings of \(\Z[\sqrt{-d}]\) into the standard Eichler order \(\orderO_0(N)\) of level \(N\) in \(M_2(\Q)\). The proof would proceed as our proofs: divide into optimal and non-optimal embeddings, and then use Eichler's local--global embedding formula, in which each local factor depends only on the local quadratic order and on whether the ambient quaternion algebra is split or ramified at $q$; it equals $1+\left(\frac{-d}{q}\right)$ at a split prime \(q\neq \ell\) and $1-\left(\frac{-d}{\ell}\right)$ at the ramified prime \(q=\ell\).
%by \cite[Prop.~30.5.3(b)]{Voight2021}
Passing from $X_0(N)/\C$ to the supersingular locus in characteristic $\ell$ replaces the split $M_2(\Q_\ell )$ by the ramified $B_{\ell,\infty} \otimes \Q_\ell $, and thereby swaps the factor $1+\left(\frac{-d}{\ell}\right)$ for $1-\left(\frac{-d}{\ell}\right)$ at $\ell$ while leaving every other local factor unchanged.
\end{remark}

We now give brief proofs for the remaining formulae. We first collect relevant discriminants in the following definition.

\begin{definition}\label{def:relevant_discriminants}
For a squarefree divisor $d \mid N$, define the set $\mathcal D(d)$ of candidate discriminants by
\[
\mathcal D(d)=
\begin{cases}
\{-4d\}, & d \text{ odd and } d \equiv 1 \pmod 4,\\
\{-4d,-d\}, & d \text{ odd and } d \equiv 3 \pmod 4,\\
\{-4d\}, & d \text{ even with } d/2 > 1 \text{ odd},\\
\{-8,-4\}, & d=2.
\end{cases}
\]
\end{definition}

\begin{remark}
The set $\mathcal D(d)$ records the candidate discriminants arising from the fixed-point condition before imposing local conditions. For a given prime $\ell$, some of these candidates may contribute trivially because the corresponding local factor at $\ell$ vanishes.
\end{remark}

\begin{lemma}\label{lem:local_factor_odd_primes}
Let \(\Delta<0\) be one of the discriminants appearing in \(\mathcal D(d)\), and let \(O\) be an Eichler order of squarefree level \(M\) in \(B_{\ell,\infty}\). At odd primes, the local factor in Eichler's embedding formula is
\(1-\left(\frac{\Delta}{\ell}\right)\) at the ramified prime \(\ell \), if \(\ell \) is odd, and
\(1+\left(\frac{\Delta}{q}\right)\) at each odd prime \(q\mid M\) with \(q\neq \ell\).
\end{lemma}

\begin{proof}
See \cite[Prop.~30.5.3(b), Lemma~30.6.17(a), Example 30.7.5]{Voight2021}; these are the local factors in \Cref{thm:eichler_embeddings}.
\end{proof}
\begin{remark}
    We refer to these local factors as \emph{standard}.
\end{remark}

\begin{proposition}\label{prop:nuell_master}
Let $N=\ell M$ with $M$ squarefree and $(\ell,M)=1$, and let $1 < d \mid N$. Put $d_0=d/\gcd(d,\ell)$. Then $\nu_{\ell}(N,d)$ is equal to
\[
\frac{1}{2}\sum_{\Delta \in \mathcal D(d)} h(\Delta)\,\Lambda(\Delta;N,d),
\]
where $\Lambda(\Delta;N,d)$ is the product of the local factors appearing in Eichler's trace formula for optimal embeddings of the order of discriminant $\Delta$ into an Eichler order of level $M/d_0$ in $B_{\ell,\infty}$.
\end{proposition}

\begin{proof}
By \Cref{thm:2to1}, $\nu_{\ell}(N,d)$ is half the number of embeddings corresponding to $w_d$-fixed supersingular points. If $d$ is odd and $d \equiv 1 \pmod 4$, these are precisely embeddings of the maximal order (of discriminant $-4d$). If $d$ is odd and $d \equiv 3 \pmod 4$, we decompose embeddings of $\Z[\sqrt{-d}]$ into an Eichler order $\orderO$ according to the order $\iota(K)\cap \orderO$, so they contribute either to the discriminant $-4d$ or to the discriminant $-d$. At the ramified prime $\ell=2$, the former contribution may vanish after imposing local optimality; this is the only additional subtlety and will be treated in the proof of \Cref{cor:nuell_ell_eq_2}. Likewise, for even $d$ the relevant orders are those listed in \Cref{def:relevant_discriminants}. Thus the total number of embeddings is the sum of the optimal embedding counts for the orders of discriminant $\Delta \in \mathcal D(d)$, while these %optimal embedding counts
are given by the trace formula as products of local factors with $h(\Delta)$.
\end{proof}

The following three corollaries collect all our formulae for $\nu_{\ell}(N, d)$.

\begin{corollary}[Odd level]\label{cor:nuell_odd_level}
If $N=\ell M$ is odd, with $M$ squarefree and $(\ell,M)=1$, and $1 < d \mid N$, then
\[
\nu_{\ell}(N,d)=\frac{1}{2} h'(-d)\left(1-\left(\frac{-d}{\ell}\right)\right)\prod_{q \mid M/d_0}\left(1+\left(\frac{-d}{q}\right)\right),
\]
where $d_0=d/\gcd(d,\ell)$ and $h'(-d)$ is the modified class number from \Cref{def:modified_classnumber}.
\end{corollary}

\begin{proof}
If $2 \nmid N$, there is no exceptional local factor at $2$. Thus \Cref{prop:nuell_master} reduces to the standard odd-prime local factors, and the two odd cases $d \equiv 1 \pmod 4$ and $d \equiv 3 \pmod 4$ can be combined using the modified class number.
\end{proof}

\begin{corollary}[Even level, odd $\ell$]\label{cor:nuell_even_level_ell_odd}
If $N=2\ell M'$, where $M'$ is odd and squarefree, while $\ell$ is an odd prime with $\ell \nmid M'$, then the following holds.
\begin{enumerate}
\item Let $d > 1$ be an odd divisor of $N$, and put $d_0=d/\gcd(d,\ell)$. Then
\[
\nu_{\ell}(N,d)=\frac{1}{2}\left(1-\left(\frac{-d}{\ell}\right)\right)\prod_{q \mid M'/d_0}\left(1+\left(\frac{-d}{q}\right)\right)\kappa(d),
\]
where
\[
\kappa(d)=
\begin{cases}
h(-4d), & d \equiv 1 \pmod 4,\\
2h(-4d)+\left(1+\left(\frac{-d}{2}\right)\right)h(-d), & d \equiv 3 \pmod 4.
\end{cases}
\]

\item If $d=2$, then
\[
\nu_{\ell}(N,2)
=
\frac{1}{2}\left(
\left(1-\left(\frac{-8}{\ell}\right)\right)\prod_{q \mid M'}\left(1+\left(\frac{-8}{q}\right)\right)
+
\left(1-\left(\frac{-4}{\ell}\right)\right)\prod_{q \mid M'}\left(1+\left(\frac{-4}{q}\right)\right)
\right).
\]
\item If $d>2$ is even, put $d_0=d/\gcd(d,\ell)$. Then
\[
\nu_{\ell}(N,d)
=
\frac{1}{2}h(-4d)
\left(1-\left(\frac{-4d}{\ell}\right)\right)
\prod_{q \mid 2M'/d_0}\left(1+\left(\frac{-4d}{q}\right)\right).
\]
\end{enumerate}
\end{corollary}

\begin{proof}
Part $(i)$ is \Cref{prop:nuell_2M_d_odd}. For $(ii)$, by \Cref{thm:2to1}$(ii)$, $\nu_{\ell}(N,2)$ is equal to half the sum of the numbers of embeddings of $\Z[\sqrt{-2}]$ and $\Z[i]$ into an Eichler order of level $M'$ in $B_{\ell,\infty}$. These correspond to the quadratic orders of discriminants $-8$ and $-4$, respectively. Since $M'$ is odd, all local factors away from $\ell$ are standard. Thus the contribution of the discriminant $-8$ is
\(
\left(1-\left(\frac{-8}{\ell}\right)\right)\prod_{q\mid M'}\left(1+\left(\frac{-8}{q}\right)\right),
\)
and the contribution of the discriminant $-4$ is
\(
\left(1-\left(\frac{-4}{\ell}\right)\right)\prod_{q\mid M'}\left(1+\left(\frac{-4}{q}\right)\right).
\)
Since $h(-8)=h(-4)=1$, the stated formula follows.

For $(iii)$, the order $\Z[\sqrt{-d}]$ is the maximal order in $\Q(\sqrt{-d})$, of discriminant $-4d$.
Thus by \Cref{thm:2to1,prop:nuell_master}, $\nu_{\ell}(N,d)$ is equal to half the number of optimal embeddings of the order of discriminant $-4d$ into an Eichler order of level $2M'/d_0$ in $B_{\ell,\infty}$.

Since $d_0$ is even, the level $2M'/d_0$ is odd, so there is no exceptional local factor at $2$. Hence all local factors are standard: at the ramified prime $\ell$ the factor is
\(
1-\left(\frac{-4d}{\ell}\right),
\)
and at each prime $q\mid 2M'/d_0$ it is
\(
1+\left(\frac{-4d}{q}\right).
\)
Multiplying these factors and $h(-4d)$ gives the stated formula.
\end{proof}

\begin{corollary}[The case $\ell=2$]\label{cor:nuell_ell_eq_2}
Let $M$ be odd and squarefree, and put $N=2M$.
\begin{enumerate}
\item If $d$ is odd and $d \equiv 1 \pmod 4$, then
\[
\nu_2(2M,d)
=
\frac{1}{2} h(-4d)
\prod_{q \mid M/d}\left(1+\left(\frac{-d}{q}\right)\right).
\]

\item If $d$ is odd and $d \equiv 3 \pmod 4$, then
\[
\nu_2(2M,d)
=
\frac{1}{2}h(-d)\left(1-\left(\frac{-d}{2}\right)\right)
\prod_{q \mid M/d}\left(1+\left(\frac{-d}{q}\right)\right).
\]

\item If $d$ is even (with $d/2$ odd), then
\[
\nu_2(2M,d)=
\begin{cases}
\frac{1}{2}\left(\prod_{q \mid M}\left(1+\left(\frac{-8}{q}\right)\right)+\prod_{q \mid M}\left(1+\left(\frac{-4}{q}\right)\right)\right), & d=2,\\
\frac{1}{2} h(-4d)\prod_{q \mid 2M/d}\left(1+\left(\frac{-4d}{q}\right)\right), & d>2.
\end{cases}
\]
\end{enumerate}
\end{corollary}

\begin{proof}
For odd $d \equiv 1 \pmod 4$, only the discriminant $-4d$ occurs. By \Cref{lem:local_factor_odd_primes}, the odd primes contribute the standard factors, while the local factor at the ramified prime $2$ is equal to $1$, giving $(i)$.

For $d \equiv 3 \pmod 4$, we decompose these embeddings according to the
order $\iota(K)\cap\orderO$, equivalently into optimal embeddings of the
orders of discriminant $-4d$ and $-d$. When $\ell=2$, the local order at
$2$ is maximal in the division algebra $B_{2,\infty}\otimes\Q_2$, so
every embedding of the nonmaximal order $\Z_2[\sqrt{-d}]$ extends to the
maximal order. Hence the $-4d$ term contributes
trivially to the optimal count, and only the discriminant $-d$
remains, giving $(ii)$.
%This gives the vanishing when $\left(\frac{-d_0}{2}\right)=1$ and the non-zero case when $\left(\frac{-d_0}{2}\right)=-1$.

For $(iii)$, the formulae are obtained in the same way from the discriminants in $\mathcal{D}(d)$.
\end{proof}

\section{Bounds on the genus and the number of fixed points}\label{sec:bounds}
\subsection{Upper bound on the number of long loops}\label{subsec:longCycleBound}
\begin{lemma}\label{lem:classNumberUpperBound}
Let $D < -4$ be a discriminant of an imaginary quadratic order with conductor $f \in \{1, 2\}$. Then the class number $h(D)$ satisfies:
\begin{enumerate}
    \item If $f=1$, then
    \[
    h(D) \leqslant \frac{\sqrt{|D|}}{2\pi} \left(\log |D| + 5 - 2\log 6\right) \leqslant \frac{\sqrt{|D|}}{2\pi} (\log |D| + 1.4165).
    \]
    \item If $f=2$ (so $D=4d$ with $d$ the associated fundamental discriminant), then
    \[
    h(D) \leqslant \frac{\sqrt{|D|}}{4\pi} \left(\log |d| + 5 - 2\log(3/2)\right) \leqslant \frac{\sqrt{|d|}}{2\pi} (\log |d| + 4.1891).
    \]
\end{enumerate}
\end{lemma}

\begin{proof}
%\begin{enumerate} \item 
%e apply \cite[Corollary 1]{Ramare1}, which states that for primitive characters of conductor $q=|D|$, $|L(1, \chi)| \leqslant \frac{1}{2}\log q + 2.5 - \log 6$. Substituting this into the class number formula $h(D) = \frac{\sqrt{|D|}}{\pi}L(1, \chi)$ yields the upper bound.
If $f=1$, this follows directly from \cite[Corollary 1]{Ramare1} and the class number formula.
%\item
If $f=2$, then $D = 4d$ where $d$ is a fundamental discriminant. If $d\in\{-3,-4\}$, then $h(D)=1$ and (ii) holds directly, so assume $d<-4$ and let \( \chi \) be the associated primitive character (modulo \(|d|\)). By the class number formula and the Euler product formula for an imprimitive character and the primitive character inducing it (see e.g.~\cite[§\,5]{Davenport}),
$h(D)$ is related to the primitive $L$-function $L(1, \chi)$ by
\[
h(D) = \frac{\sqrt{|D|}}{\pi} \left| \left(1 - \frac{\chi(2)}{2}\right) L(1, \chi) \right|.
\]
If \( d\) is odd, then Ramar\'e \cite[Corollary 2]{Ramare2} proves that
\[
\left| \left(1 - \frac{\chi(2)}{2}\right) L(1, \chi) \right| \leqslant \frac{1}{4} (\log |d| + 5 - 2\log(3/2) ),
\]
since \( \chi\) is odd. If $d$ is even, then $\chi(2)=0$, and by \cite[Corollary 3]{Ramare1} (with the correction of the proof noted in \cite{Ramare2}),
\[
|L(1,\chi)| \leqslant \frac14 \left(\log |d| + 5 - 2\log 3\right)
\leqslant \frac14 \left(\log |d| + 5 - 2\log(3/2)\right).
\]
In both cases,
\[
h(D) \leqslant \frac{\sqrt{|D|}}{4\pi}\left(\log |d| + 5 - 2\log(3/2)\right),
\]
which is the desired bound.
%This bound is strictly stronger than the bound in $f=1$
%\end{enumerate}
\end{proof}

\begin{corollary}\label{cor:modified_class_number_bound}
Let $d>1$ be an odd squarefree integer. Then
\[
    h'(-d) \leqslant \frac{\sqrt{d}}{2\pi}\left(2\log d + 10 - 4\log 3\right).
\]
\end{corollary}

\begin{proof}
Let $c_1 = 5 - 2\log 6$ and $c_2 = 5 - 2\log(3/2)$, so $c_1 + c_2 = 10 - 4\log 3$.

If $d \equiv 1 \pmod 4$, then $-4d$ is a fundamental discriminant, so $\Z[\sqrt{-d}]$ is the maximal order of $K$ (conductor $1$) and $h'(-d)=h(-4d)$. By \Cref{lem:classNumberUpperBound}(i) (applied to the fundamental discriminant $-4d$),
\[
    h'(-d) = h(-4d) \leqslant\frac{\sqrt{4d}}{2\pi}(\log(4d) + c_1)
    = \frac{\sqrt{d}}{2\pi}(2\log d + 2\log 4 + 2c_1)
    = \frac{\sqrt{d}}{2\pi}(2\log d + c_1 + c_2),
\]
where the last equality uses $2\log 4 + 2c_1 = c_1 + c_2$ (equivalently $c_2 - c_1 = \log 16$).

If $d \equiv 3 \pmod 4$, then $h'(-d)=h(-4d)+h(-d)$. For $d=3$ the bound holds directly, so let $d>3$. Using \Cref{lem:classNumberUpperBound}(ii) for $h(-4d)$ and \Cref{lem:classNumberUpperBound}(i) for $h(-d)$,
\[
    h'(-d) \leqslant\frac{\sqrt{d}}{2\pi}(\log d + c_2) + \frac{\sqrt{d}}{2\pi}(\log d + c_1)
    = \frac{\sqrt{d}}{2\pi}(2\log d + c_1 + c_2).
\]

\noindent Since $c_1 + c_2 = 10 - 4\log 3$, the claimed bound follows in both cases.
\end{proof}

\begin{corollary}\label{cor:H_bound}
Let $d>2$ be squarefree and let $H(d)$ be as in \Cref{def:H_of_d} below. Then
\[
H(d)\leqslant \frac{\sqrt d}{2\pi}\left(2\log d + C_0\right),
\qquad\text{where } C_0 = \Czero.
\]
\end{corollary}

\begin{proof}
If $d$ is odd, this is exactly \Cref{cor:modified_class_number_bound}.

If $d$ is even, write $d=2m$ with $m$ odd squarefree. Then
\[
H(d)=h(-4d)=h(-8m),
\]
and $-8m$ is a fundamental discriminant (conductor $f=1$).
By \Cref{lem:classNumberUpperBound}(i),
\[
h(-4d)=h(-8m)
\le \frac{\sqrt{|{-8m}|}}{2\pi}\left(\log(8m)+c_1\right)
= \frac{2\sqrt d}{2\pi}\left(\log(4d)+c_1\right),
\]
where $c_1=5-2\log 6$.  Rewriting the right-hand side gives
\[
\frac{\sqrt d}{2\pi}\left(2\log d + (2\log 4+2c_1)\right)
=\frac{\sqrt d}{2\pi}\left(2\log d + (10-4\log 3)\right),
\]
since $2\log 4+2(5-2\log 6)=10-4\log 3$. This is the claimed bound.
\end{proof}

In both bounding the genus of $X_0(N)^*$ and bounding the number of loops $L_\ell$ in the dual graph of $X_0(N)^*$ we need to do a very similar computation twice. The following lemma allows us to do that computation just once.

\begin{lemma}\label{lem:divisor_sum_bound}
Let $N$ be squarefree, and let $A(d)$ be a non-negative function on the divisors $d>1$ of $N$. Suppose that, for some function $F(N)$, one has
$A(d)\leqslant \sqrt d\,F(N)$ for every divisor $d>1$ of $N$. Then
\[
\sum_{1<d\mid N} A(d)2^{\omega(N/d)}
\leqslant
\sqrt N\,F(N)\prod_{q \mid N}\left(1+\frac{2}{\sqrt q}\right).
\]
Consequently,
\[
\frac{1}{2^{\omega(N)+1}}\sum_{1<d\mid N} A(d)2^{\omega(N/d)}
\leqslant
\frac{\sqrt N\,F(N)}{2}
\prod_{q\mid N}\left(\frac12+\frac{1}{\sqrt q}\right).
\]
\end{lemma}

\begin{proof}
From $A(d)\leqslant \sqrt d F(N)$, we obtain
\[
\sum_{1<d\mid N} A(d)2^{\omega(N/d)}
\leqslant
F(N)\sum_{d\mid N}2^{\omega(N/d)}\sqrt d
=
F(N)\sqrt N\sum_{n\mid N}\frac{2^{\omega(n)}}{\sqrt n},
\]
where in the last equality we introduced $n=N/d$. For squarefree $N$ the divisor sum has an Euler product
\[
\sum_{n\mid N}\frac{2^{\omega(n)}}{\sqrt n}
=
\prod_{q \mid N}\left(1+\frac{2}{\sqrt q}\right),
\]
which proves the first inequality.
Dividing by $2^{\omega(N)+1}$ gives
\[
\frac{1}{2^{\omega(N)+1}}
\prod_{q \mid N}\left(1+\frac{2}{\sqrt q}\right)
=
\frac12\prod_{q \mid N}\left(\frac12+\frac{1}{\sqrt q}\right),
\]
which proves the second inequality.
\end{proof}

\begin{lemma}\label{lem:nuell_sum_over_ell}
Let $N > 6$ be squarefree and $1<d\mid N$. Then
\[
2\sum_{\ell \mid N/d}\nu_{\ell}(N,d) \leqslant H(d)2^{\omega(N/d)}.
\]
\end{lemma}

\begin{proof}
First suppose that $N$ is odd. By
\Cref{cor:nuell_odd_level}, for fixed $d$ a summand
\[ \nu_{\ell}(N,d)=\frac{1}{2} h'(-d)\left(1-\left(\frac{-d}{\ell}\right)\right)\prod_{q \mid M/d_0}\left(1+\left(\frac{-d}{q}\right)\right) \]
can be non-zero only when
\( \left(\frac{-d}{\ell}\right)=-1 \). But this implies that all other summands are trivial, as the same Kronecker symbol appears with the opposite sign in the right-most product.
Consequently, at most one prime $\ell\mid N/d$ contributes. If such a
prime exists, then
\( \nu_{\ell}(N,d)
\leqslant H(d)2^{\omega(N/d)-1} \).
The desired inequality follows.

The same argument applies when $d>2$ is even. Indeed, in this case
$2\mid d$, so every prime dividing $N/d$ is odd, and
\Cref{cor:nuell_even_level_ell_odd}$(iii)$ applies.

Suppose next that $d=2$. Applying the preceding argument separately
to the two discriminants $-8$ and $-4$ appearing in
\Cref{cor:nuell_even_level_ell_odd}$(ii)$, each of the two contributions
is at most
\( 2^{\omega(N/2)-1} \).
Therefore
\[
\sum_{\ell \mid N/2}\nu_{\ell}(N,2)
\leqslant
2^{\omega(N/2)}
=
H(2)2^{\omega(N/2)-1}.
\]

It remains to consider the case where $2\mid N$ and $d$ is odd. Write
\( Q=\{q:q\mid N/d,\ q\text{ odd}\}, s=\#Q \).
Thus \( \omega(N/d)=s+1 \).

Suppose first that $d\equiv1\pmod4$. If the symbols
\( \left(\frac{-d}{q}\right) \)
are equal to $1$ for all \( q\in Q\), only the term corresponding to $\ell=2$ can be
non-zero, and by \Cref{cor:nuell_ell_eq_2}$(i)$ it is equal to
\( H(d)2^{s-1} \).
If exactly one of these symbols is equal to $-1$, only the
corresponding odd prime can contribute, and its contribution is again
\( H(d)2^{s-1} \).
If at least two symbols are equal to $-1$, all contributions vanish.
Hence
\[
\sum_{\ell \mid N/d}\nu_{\ell}(N,d)
\leqslant
H(d)2^{s-1}
\leqslant
H(d)2^s.
\]

Finally, suppose that $d\equiv3\pmod4$, and put
\( a=h(-4d), b=h(-d) \), so \( H(d)=a+b\).
Also put
\( \varepsilon=\left(\frac{-d}{2}\right)\in\{-1,1\} \).

If the symbols \( \left(\frac{-d}{q}\right)\) are \( 1\) for all \( q\in Q\), only the term corresponding to \( \ell=2\)  may be
non-zero. By \Cref{cor:nuell_ell_eq_2}$(ii)$, it is equal to
\( \frac{1}{2}b(1-\varepsilon)2^s \leqslant H(d)2^s \).

If exactly one symbol is equal to $-1$, only the corresponding odd
prime may contribute. By
\Cref{cor:nuell_even_level_ell_odd}$(i)$, its contribution is
\( 2^{s-1}\left( 2a+(1+\varepsilon)b \right) \leqslant H(d)2^s \).
If at least two symbols are equal to $-1$, all contributions vanish.

Thus in every case the claim holds.
\end{proof}

\begin{lemma}\label{lem:upperBoundSumBps} Let $N > 6$ be squarefree. Then
\[
    \sum_{\ell \mid N} B(\ell ,N,W_N) \leqslant 5\omega(N)+ 2^{-\omega(N)}\frac{\sqrt N}{2\pi}\,\left(2\log N + C_0\right)\prod_{q \mid N}\left(1+\frac{2}{\sqrt q}\right),
\]
where \( C_0 = \Czero \).
\end{lemma}

\begin{proof}
Since $\#W_N=2^{\omega(N)}$, \Cref{thm:Bpnw} gives
\begin{align*}
\sum_{\ell \mid N}B(\ell ,N,W_N) &\leqslant
2\sum_{\ell \mid N}(1+c_\ell) +
\frac{2}{2^{\omega(N)}} \sum_{\ell \mid N}\sum_{1<d\mid N/\ell}\nu_{\ell}(N,d)\\
&\leqslant 5\omega(N) +
\frac{1}{2^{\omega(N)}} \sum_{1<d\mid N} H(d)2^{\omega(N/d)},
\end{align*}
where the second inequality uses $2\sum_{\ell \mid N}(1+c_\ell)\leqslant5\omega(N)$, reindexes $\sum_{\ell \mid N}\sum_{1<d\mid N/\ell}=\sum_{1<d\mid N}\sum_{\ell \mid N/d}$, and applies \Cref{lem:nuell_sum_over_ell}. The bound $2\sum_{\ell \mid N}(1+c_\ell)\leqslant5\omega(N)$ holds for every squarefree $N>6$: since $c_\ell =0$ for $\ell>3$ and at most one prime factor equals $2$ and at most one equals $3$, the left-hand side equals $2\omega(N)+2\sum_{\ell \mid N}c_\ell $, which is $\leqslant2\omega(N)+2(c_2+c_3)=2\omega(N)+7$ when $6\mid N$ (whence $\omega(N)\geqslant3$) and $\leqslant2\omega(N)+2c_2=2\omega(N)+5$ otherwise. If $\omega(N)=1$, then $N>6$ is a prime $>3$ and the left-hand side equals $2$; otherwise $\omega(N)\geqslant 2$ and in both cases the value is $\leqslant5\omega(N)$.

Put
\( F(N)=\frac{2\log N+C_0}{2\pi} \).
If $d>2$, then \Cref{cor:H_bound} and $d\leqslant N$ give
\(
H(d)\leqslant \sqrt d\,F(N) \).
For $d=2$, the same inequality follows directly from $N>6$.
Therefore \Cref{lem:divisor_sum_bound} gives
\[
\sum_{1<d\mid N}
H(d)2^{\omega(N/d)}
\leqslant
\sqrt N\,F(N)
\prod_{q \mid N}
\left(1+\frac{2}{\sqrt q}\right).
\] 
Dividing by $2^{\omega(N)}$ proves the lemma.
\end{proof}

\medskip

\subsection{Lower bound on the genus of \texorpdfstring{$X_0(N)^*$}{X0(N)*} for \texorpdfstring{$N$}{N} squarefree}\label{subsec:genusBounds}

We now establish the genus lower bound (\Cref{cor:gstarLowerBound}). %Combined with
%the bad-prime bound of \Cref{lem:upperBoundSumBps}, it reduces the target inequality
%$g_0^*(N)>B(N)+1$ to positivity of a single explicit function $\Phi(N)$
%(\Cref{lem:threshold_reduction}): the positive main genus term is
%$\psi(N)/(12\cdot2^{\omega(N)})$, while both the negative fixed-point term and $B(N)$ are
%$O(\sqrt N\log N)$; the delicate point is that the main term carries the factor
%$2^{-\omega(N)}$ whereas the $O(\sqrt N\log N)$ terms do not, so ``primorial-like'' $N$
%(with many small prime factors) are the worst case and determine the threshold
%$N_0=1\,231\,230$ of \Cref{thm:explicit_threshold}.

For a divisor $1<d\mid N$, let $\nu(N,d)$ denote the number of fixed points of $w_d$ on $X_0(N)(\C)$, and write $g_0^*(N)\colonequals g(X_0(N)^*)$.

\begin{lemma}\label{lem:genusStarQuotient}
Let $N$ be squarefree. Then the genus of $X_0(N)^*$ is \[ g_0^*(N) = 1 + \frac{g(X_0(N))-1}{2^{\omega(N)}} - \frac{1}{2^{\omega(N)+1}} \sum_{1\neq d\mid N} \nu(N, d) .\]
\end{lemma}
\begin{proof}

From Riemann--Hurwitz we get that the genus of $X_0(N)^*$ is
\[ g_0^*(N) = 1 + \frac{g(X_0(N))-1}{2^{\omega(N)}} - \frac{1}{2^{\omega(N)+1}}\sum_{x\in X_0(N)}\left(e_x - 1\right).\]
Since the cover is tame, $e_x=\#W_x$, and each point with stabiliser $W_x$ is fixed by the $\#W_x-1$ nontrivial elements of $W_x$; hence $\ds \sum_{x\in X_0(N)} (e_x-1) %\sum_{1\neq d\mid N} |\Fix(w_d)|
= \sum_{1\neq d\mid N} \nu(N, d)$.
\end{proof}

\begin{remark}\label{rem:genusX0}
    The genus of $X_0(N)$ is
\[ g(X_0(N)) = 1 + \frac{1}{12}\psi(N) - \frac{e_2(N)}{4} - \frac{e_3(N)}{3} - \frac{e_\infty}{2}.\]
where
\begin{align*}
    \ds \psi(N) &= N\prod_{q \mid N}\left(1+ \frac{1}{q}\right), \\e_2(N)&=\prod_{q \mid N} \left(1+\left(\frac{-4}{q}\right)\right) \leqslant 2^{\omega(N)},\\ 
    \ds e_3 (N) &= \prod_{q \mid N} \left(1+\left(\frac{-3}{q}\right)\right) \leqslant 2^{\omega(N)},\\
    e_\infty(N)&=2^{\omega(N)},
\end{align*}
and $\left(\frac{\cdot}{q}\right)$ denotes the Kronecker symbol.
\end{remark}

The fixed-point counts \( \nu(N,d) \) were determined by Kluit \cite{Klu77} -- the case $d = N$ is classical (Fricke), and Ogg \cite{Ogg74} gave a formula when $d > 3$ and $N/d$ odd.
\begin{lemma}[\cite{Klu77}]\label{lem:KluitFormulae}
Let $N$ be squarefree, and let $d>1$ be a divisor of $N$.
\begin{enumerate}
\item[(a)] If $N$ is odd, then
\begin{equation}\label{eq:nuNd_odd}
\nu(N,d)
=
h'(-d)\prod_{q \mid N/d}\left(1+\left(\frac{-d}{q}\right)\right).
\end{equation}

\item[(b)] If $N$ is even, then:
\begin{enumerate}
\item[(i)] for $d=2$,
\begin{equation}\label{eq:nuNd_even_d2}
\nu(N,2)
=
\prod_{q \mid \frac{N}{2}}\left(1+\left(\frac{-4}{q}\right)\right)
+
\prod_{q \mid \frac{N}{2}}\left(1+\left(\frac{-8}{q}\right)\right);
\end{equation}

\item[(ii)] for $2\mid d$ and $d>2$,
\begin{equation}\label{eq:nuNd_even_deven}
\nu(N,d)
=
h(-4d)\prod_{q \mid N/d}\left(1+\left(\frac{-4d}{q}\right)\right);
\end{equation}

\item[(iii)] for $d$ odd,
\begin{equation}\label{eq:nuNd_even_dodd}
\nu(N,d)
=
\begin{cases}
\displaystyle
h(-4d)\prod_{2<q\mid \frac{N}{d}}\left(1+\left(\frac{-4d}{q}\right)\right),
& d\equiv 1 \pmod 4,\\[8pt]
\displaystyle
2h(-4d)\prod_{2<q\mid \frac{N}{d}}\left(1+\left(\frac{-4d}{q}\right)\right)
+
h(-d)\prod_{1<q\mid \frac{N}{d}}\left(1+\left(\frac{-d}{q}\right)\right),
& d\equiv 3 \pmod 4.
\end{cases}
\end{equation}
\end{enumerate}
\end{enumerate}
Here the symbol in the product over $1<q\mid N/d$ is the Kronecker symbol, so the prime $2$ is included when it divides $N/d$.
\end{lemma}

To collect the class numbers appearing in formulae for \( \nu(N, d)\) and \( \nu_{\ell}(N, d) \), in a way suitable for the uniform upper bound, we define the following.

\begin{definition}\label{def:H_of_d}
Let $d\geqslant 2$ be squarefree. Define $H(d)=\sum_{\Delta\in\mathcal D(d)}h(\Delta)$, i.e.\
\[
H(d)\colonequals
\begin{cases}
2, & d=2,\\
h'(-d), & d>2 \text{ odd},\\
h(-4d), & d>2 \text{ even}.
\end{cases}
\]
\end{definition}

\begin{corollary}\label{cor:nu_bounded_by_H}
Let $N$ be squarefree, and let $d>1$ be a divisor of $N$. Then
\[
\nu(N,d)\leqslant 2^{\omega(N/d)}H(d).
\]
\end{corollary}

\begin{proof}
If $N$ is odd, this follows directly from \Cref{lem:KluitFormulae}, since then
\[ \nu(N,d)=h'(-d)\prod_{q \mid N/d}\left(1+\left(\frac{-d}{q}\right)\right)\]
and each local factor is at most $2$.

Assume now that $N$ is even. If $d=2$, then Kluit's formula gives
$\nu(N,2)\leqslant 2^{\omega(N/2)+1}=2^{\omega(N/2)}H(2)$.
If $d>2$ is even, the formula contains the class number $h(-4d)=H(d)$, and all local factors are again at most $2$.

It remains to consider the case where $N$ is even and $d>2$ is odd. If $d\equiv 1\pmod 4$, then
\[
\nu(N,d)
=
h(-4d)\prod_{2<q\mid N/d}\left(1+\left(\frac{-d}{q}\right)\right)
\leqslant
2^{\omega(N/d)}h(-4d)
=
2^{\omega(N/d)}H(d),
\]
since $h'(-d)=h(-4d)$ in this case. If $d\equiv 3\pmod 4$, Kluit's formula has two terms. The first is bounded by $2^{\omega(N/d)}h(-4d)$, because the factor $2$ in front of $h(-4d)$ is compensated by the omission of the prime $2$ from the product. The second is bounded by $2^{\omega(N/d)}h(-d)$. Thus the total is bounded by $2^{\omega(N/d)}h'(-d)=2^{\omega(N/d)}H(d)$.
\end{proof}

\begin{lemma}[Uniform upper bound for the total number of fixed points]\label{lem:totFixUpperBound_uniform}
Let $N\geqslant 6$ be squarefree. Then
\begin{equation}\label{ineq:sumFixEulerProd_uniform}
\sum_{1<d\mid N}\nu(N,d)
\leqslant
\frac{\sqrt N}{2\pi}\left(2\log N+C_0\right)
\prod_{q \mid N}\left(1+\frac{2}{\sqrt q}\right),
\end{equation}
where $C_0=\Czero$.
\end{lemma}

\begin{proof}
Put $F(N)\colonequals(2\log N+C_0)/(2\pi)$. By \Cref{cor:nu_bounded_by_H}, we have
\[
\sum_{1<d\mid N}\nu(N,d)
\leqslant
\sum_{1<d\mid N}2^{\omega(N/d)}H(d).
\]
By \Cref{cor:H_bound}, and since $\log d\leqslant \log N$, we have $H(d)\leqslant \sqrt d\,F(N)$ for all $d>2$ dividing $N$. For $d=2$, the same bound follows from $H(2)=2\leqslant \sqrt 2\,F(N)$, which holds since $N\geqslant 6$. Therefore $H(d)\leqslant \sqrt d\,F(N)$ for every divisor $d>1$ of $N$.

Applying \Cref{lem:divisor_sum_bound} with $A(d)=H(d)$ gives
\[
\sum_{1<d\mid N}\nu(N,d)
\leqslant
\sqrt N\,F(N)\prod_{q \mid N}\left(1+\frac{2}{\sqrt q}\right),
\]
which is exactly \eqref{ineq:sumFixEulerProd_uniform}.

Dividing by $2^{\omega(N)+1}$ and using the second part of \Cref{lem:divisor_sum_bound}, we obtain
\[
\frac{1}{2^{\omega(N)+1}}\sum_{1<d\mid N}\nu(N,d)
\leqslant
\frac{\sqrt N\,F(N)}{2}
\prod_{q \mid N}\left(\frac12+\frac{1}{\sqrt q}\right). \qedhere
\]
\end{proof}

\begin{corollary}[Genus lower bound for all squarefree $N$]\label{cor:gstarLowerBound}
Let $N \geqslant 6$ be squarefree. Then
\[
g_0^*(N) \geqslant -\frac{1}{12} + \frac{\psi(N)}{12\cdot 2^{\omega(N)}} - \frac{\sqrt{N}}{4\pi}(2\log N + C_0)\prod_{q \mid N} \left( \frac 12 + \frac{1}{\sqrt{q}}\right),
\]\noindent where $C_0=\Czero$. %$\psi(N)=N\prod_{q \mid N}\left(1+\frac1q\right)$ and
\end{corollary}
\begin{proof}
By \Cref{lem:genusStarQuotient} and \Cref{rem:genusX0},
\[
g_0^*(N)
=
1+\frac{\psi(N)}{12\cdot 2^{\omega(N)}}
-\frac{e_2(N)}{4\cdot 2^{\omega(N)}}
-\frac{e_3(N)}{3\cdot 2^{\omega(N)}}
-\frac12
-\frac{1}{2^{\omega(N)+1}}\sum_{1<d\mid N}\nu(N,d).
\]
Since $e_2(N)\leqslant 2^{\omega(N)}$, $e_3(N)\leqslant 2^{\omega(N)}$, the claim follows from \eqref{ineq:sumFixEulerProd_uniform} divided by $2^{\omega(N)+1}$, using
$2^{-\omega(N)}\prod_{q\mid N}(1+2/\sqrt q)=\prod_{q\mid N}(\tfrac12+\tfrac1{\sqrt q})$.
\end{proof}

\section{The existence of $p$-adic heights supported at $p$ for $N > 714$}\label{sec:existence}

The main goal of this section is to prove the following theorem.

\begin{theorem}\label{thm:existence_of_supported_at_p_heights} Let $N$ be a squarefree integer such that $X_0(N)^*$ has genus at least two and $p$ a prime not dividing $N$. Then there exists a non-trivial $p$-adic height supported at $p$ on $X_0(N)^*$ if and only if $N$ is not one of the following $36$ integers:

\begin{center}
\begin{tabular}{ccccccccc}
$85^{\dagger}$  & $93^{\dagger}$  & $106^{\dagger}$ & $115^{\dagger}$ & $122^{\dagger}$ & $129^{\dagger}$ & $154^{\dagger}$ & $161^{\dagger}$ & $165^{\dagger}$ \\
$166^{\S}$      & 178             & 183             & $186^{\dagger}$ & 202             & 203             & $215^{\dagger}$ & $230^{\dagger}$         & 235 \\
237             & 246             & 258             & 273             & 282             & $285^{\dagger}$ & $286^{\dagger}$ & 290                     & 310 \\
318             & 322             & 374             & $390^{\dagger}$ & 418             & 435             & 462             & 570                     & 714.
\end{tabular}
\end{center}
In particular, a $p$-adic height supported at $p$ exists on $X_0(N)^*$ for $N>714$.
\end{theorem}
\begin{remark}
    The 36 levels from the above theorem are the only cases where one needs to compute the height contributions at the primes of bad reduction in order to carry out quadratic Chabauty. For the 16 marked levels, however, this is no longer necessary as the set $X_0(N)^*(\Q)$ is already known.
    \begin{itemize}
        \item The levels marked with $\dagger$ were done in \cite{BGX21}, using elliptic curve Chabauty.
        \item The level marked with $\S$ was done in \cite{FrancescaOana}, using bielliptic quadratic Chabauty.
    \end{itemize}
    These 16 cases correspond exactly to the values in the above list for which $X_0(N)^*$ has genus $2$. The 20 remaining curves have genus at most 5.
\end{remark}

The proof of the theorem is done in three steps and occupies the rest of this section. The first step is to prove the existence of heights supported only at $p$ for $N$ larger than some explicit constant using the bounds from \Cref{sec:bounds}. However, this explicit constant is too large to just run \Cref{prop:correspondence_algorithm} for all $N$ up to this constant. The second step is to bring down this bound by doing computer computations up to this explicit bound using the explicit formulae from \Cref{sec:fixedpoint_formulas} in terms of class numbers, which is much faster than the quaternion algebra arithmetic needed for \Cref{prop:correspondence_algorithm}. The final step is to use \Cref{prop:correspondence_algorithm} to decide whether there exists a $p$-adic height supported at $p$ up to this smaller bound.   %using the bounds form \Cref{sec:bounds}. After that the proof is finished using explicit computer computations up to this bound using the explicit formulas from \Cref{sec}
The first step is done in \Cref{thm:explicit_threshold} and consists of combining the genus lower bound of \Cref{cor:gstarLowerBound} with the upper
bound of \Cref{lem:upperBoundSumBps} on the number of linear conditions that need to be satisfied in order for the local heights at the primes of bad reduction to vanish. Recall that
this is guaranteed once $g_0^*(N)>B(N)+1$, where
\[
B(N) \colonequals \sum_{\ell \mid N}B(\ell ,N,W_N)
\]
and $B(\ell ,N,W_N)$ is the bound of \Cref{thm:Bpnw}. The criterion $g>B(N)+1$ follows from \Cref{cor:clean_criterion}, since $L_\ell(N) \leqslant B(\ell, N, W_N)$ by \Cref{thm:Bpnw}.

\begin{lemma}\label{lem:threshold_reduction}
For squarefree $N>6$, keep $C_0=\Czero$, and set
%$\pi(N)\colonequals\prod_{q\mid N}\left(\tfrac12+\tfrac{1}{\sqrt q}\right)$ and
\[
\Phi(N)\colonequals\frac{\psi(N)}{12\cdot 2^{\omega(N)}}
-\frac{3}{4\pi}\sqrt N\,(2\log N+C_0)\prod_{q\mid N}\left(\frac 12 + \frac{1}{\sqrt{q}}\right)
-5\,\omega(N)-\frac{13}{12}.
\]
Then $g_0^*(N)-B(N)-1\geqslant\Phi(N)$; in particular $g_0^*(N)>B(N)+1$ whenever
$\Phi(N)>0$.
\end{lemma}
\begin{proof}
By \Cref{cor:gstarLowerBound},
\[
g_0^*(N)\geqslant-\frac1{12}+\frac{\psi(N)}{12\cdot 2^{\omega(N)}}
-\frac{\sqrt N}{4\pi}(2\log N+C_0)\prod_{q\mid N}\left(\frac 12 + \frac{1}{\sqrt{q}}\right).
\]
By \Cref{lem:upperBoundSumBps} together with the identity
$2^{-\omega(N)}\prod_{q\mid N}(1+2/\sqrt q) = \prod_{q\mid N}\left(\frac 12 + \frac{1}{\sqrt{q}}\right)$,
\[
B(N)\leqslant 5\,\omega(N)+\frac{\sqrt N}{2\pi}(2\log N+C_0) \prod_{q\mid N}\left(\frac 12 + \frac{1}{\sqrt{q}}\right).
\]
Subtracting these and $1$ gives $g_0^*(N)-B(N)-1\geqslant\Phi(N)$.
\end{proof}

The following lemma and theorem are supported by \repo{computations/general/final\_inequality.m}{Magma computations}.

\begin{lemma}\label{lem:threshold_tail}
If $N$ is squarefree and $\Phi(N) \leqslant 0$, then $\omega(N) \leqslant 9$; moreover $N<2.87\times10^{8}$. In particular, if $N \geqslant 2.87\times 10^{8}$ is squarefree and $p\nmid N$ is a prime, then there exists a $p$-adic height supported only at $p$ on $X_0(N)^*$.
\end{lemma}

\begin{proof}
Write $k=\omega(N)$, let $q_1<q_2<\cdots$ be the rational primes and
$P_k=\prod_{i\leqslant k}q_i$, so $N\geqslant P_k$. Since $\psi(N)\geqslant N$ and
$q\mapsto\tfrac12+\tfrac1{\sqrt q}$ is decreasing,
$\prod_{q\mid N}\left(\frac 12 + \frac{1}{\sqrt{q}}\right)\leqslant\bar\pi_k\colonequals\prod_{i\leqslant k}(\tfrac12+\tfrac1{\sqrt{q_i}})$;
hence $\Phi(N)\geqslant\Phi_{k}(N)$, where
\[
\Phi_{k}(x)\colonequals\frac{x}{12\cdot 2^{k}}-\frac{3\,\bar\pi_k}{4\pi}\sqrt x\,(2\log x+C_0)-5k-\frac{13}{12}
\qquad(x\geqslant1).
\]
Every such function is of the form $\Theta_{a,b,e}(x)=ax-b\sqrt x(2\log x+C_0)-e$ with
$a,b,e>0$ (such that $a < e$), and for these $\Theta_{a,b,e}(1)<0$, $\Theta_{a,b,e}(x)\to+\infty$, and
$\Theta_{a,b,e}'$ is increasing on $[1,\infty)$: indeed
\[
\frac{d}{dx}\left(\sqrt x\,(2\log x+C_0)\right)
=\frac{2\log x+C_0}{2\sqrt x}+\frac{2}{\sqrt x}
=\frac{2\log x+C_0+4}{2\sqrt x}, \,\, \text{so}\]
\(\Theta_{a,b,e}'(x)=a-b\,\frac{2\log x+C_0+4}{2\sqrt x}\),
and the subtracted term is decreasing because
\[ \frac{d}{dx}\left(\frac{2\log x+C_0+4}{2\sqrt x}\right)=-\frac{2\log x+C_0}{4x^{3/2}}<0 \,\, \text{for } x\geqslant 1,\]
as $C_0>0$. So $\Theta_{a,b,e}$ decreases then increases and has a unique zero
$\mathcal{Z}_{k}\geqslant1$, with $\Phi_{k}(x)>0\iff x>\mathcal{Z}_{k}$.

A direct computation gives
$\mathcal{Z}_{1}<\cdots<\mathcal{Z}_{9}<2.87\times10^{8}$. It remains
to exclude $k>9$, for which we show $\Phi_{k}(P_k)>0$, so that $N\geqslant P_k$ gives
$\Phi(N)\geqslant\Phi_{k}(N)>0$. We prove $\Phi_{k}(P_k)>0$ by induction on $k$, analysing how each summand in $\Phi_{k}$ changes. Put
$L_k=\frac{P_k}{12\cdot 2^{k}}$, $M_k=\frac{3\,\bar\pi_k}{4\pi}\sqrt{P_k}(2\log P_k+C_0)$
and $E_k=5k+\tfrac{13}{12}$, so $\Phi_{k}(P_k)=L_k-M_k-E_k$. Writing $q=q_{k+1}$ and
$A_k=2\log P_k+C_0$, the identities
$\bar\pi_{k+1}\sqrt{P_{k+1}}=(\tfrac{\sqrt q}{2}+1)\bar\pi_k\sqrt{P_k}$ and $A_{k+1}=A_k+2\log q$ give 
\[
L_{k+1}=\tfrac{q}{2}L_k,\quad
M_{k+1}=\gamma(q,A_k)\cdot\tfrac{q}{2}M_k,\,\, \text{where }
\gamma(q,A) = \tfrac{\sqrt{q}+2}{q}\left(1+\tfrac{2\log q}{A}\right)
\]
is the ratio of two growth factors (for $M_k$ and $L_k$). Now $\gamma$ is decreasing in $A$, and decreasing in $q$ since
$\tfrac{\partial}{\partial q}\log\gamma=\tfrac{1/2}{q+2\sqrt q}-\tfrac1q+\tfrac{2}{q(A+2\log q)}
\leqslant\tfrac{2}{qA}-\tfrac1{2q}=\tfrac1q\left(\tfrac2A-\tfrac12\right)<0$,
because $A\geqslant C_0>4$.
Hence $\gamma(q_{k+1},A_k)\leqslant\gamma(q_{10},A_{9})\leqslant0.2936<1$ for every $k\geqslant9$,
and therefore
\[
\Phi_{k+1}(P_{k+1})=\tfrac{q}{2}\left(L_k-\gamma M_k\right)-E_k-5
\;\geqslant\;\tfrac{q}{2}\left(L_k-M_k\right)-E_k-5 > 0 ,
\]
using the inductive hypothesis $L_k-M_k>E_k$. The base case
$\Phi_{10}(P_{10})>0$ is immediate.
\end{proof}
\begin{remark}\label{rem:threshold_uniform}
One could also bound $\omega(N)$ in \Cref{lem:threshold_tail} using $\omega(N)\leqslant 1.3841\,\log N/\log\log N$ \cite[Th\'eor\`eme~11]{Robin83}. However, the proof above is elementary and self contained. Additionally, it nicely shows that  the worst cases are ``primorial-like'' $N$, i.e. those with many small prime factors and for which $2^{\omega(N)}$ is largest relative to $N$.
\end{remark}

\begin{theorem}\label{thm:explicit_threshold}
Let $N_0\colonequals 2\,589\,510=2\cdot3\cdot5\cdot7\cdot11\cdot19\cdot59$. Then $N_0$ is
the largest squarefree level with $\Phi(N)\leqslant0$. In particular
$g_0^*(N)>B(N)+1$ for every squarefree $N>N_0$.
\end{theorem}

\begin{proof}
By \Cref{lem:threshold_tail} every squarefree $N$ with $\Phi(N)\leqslant0$ satisfies
$\omega(N)=k\leqslant9$ and $N\leqslant\mathcal{Z}_{k}$ (from the proof of the same lemma); evaluating $\Phi$ on all such
$N$ gives $N_0$, with $\Phi(N_0)=-15.05\ldots$. For the final assertion, if $N>N_0$ is squarefree then $\Phi(N)>0$ by maximality of
$N_0$, hence $g_0^*(N)-B(N)-1\geqslant\Phi(N)>0$ by \Cref{lem:threshold_reduction}.
\end{proof}

\begin{proof}[Proof of \Cref{thm:existence_of_supported_at_p_heights}]
Let $N$ be squarefree with $g_0^*(N)\geqslant 2$. If $N>N_0$, then $g_0^*(N)>B(N)+1$ by \Cref{thm:explicit_threshold}, so a $p$-adic height supported at $p$ exists, as recalled at the beginning of this section.

For the squarefree $N\leqslant N_0$ we compute $g_0^*(N)$ and $B(N)$ exactly, using \Cref{lem:genusStarQuotient}, \Cref{thm:Bpnw} and the formulae for $\nu_\ell(N,d)$ of \Cref{sec:fixedpoint_formulas}. This is done by the
\repo{computations/general/existence_of_relations.m}{computation} for $N\leqslant N_0$, which takes class numbers computed in \cite{classGroupTabulation} (available at \href{https://www.lmfdb.org/NumberField/QuadraticImaginaryClassGroups}{this LMFDB page}).
The largest level $N$ that satisfies $2\leqslant g_0^*(N)\leqslant B(N)+1$ is $6510$ (for $N=6510$ one has $g_0^*=39$ and $B+1=40$, see \repo{logs/general/existence_of_relations.txt}{the output}). For each of them the same script finds a prime $q<10000$, $q\nmid N$, such that $T_q$ generates $\T\otimes\Q$ (it raises an error if there is none), and runs the algorithm of \Cref{prop:correspondence_algorithm}. It finds a clean correspondence for every such level except the $36$ listed in the theorem.
\end{proof}

\begin{proof}[Proof of \Cref{thm:main1.1}]
By \Cref{lem:atLeast_g_minus_Ll_heights} and \Cref{thm:Bpnw}, the number of such heights is at least $g(X_0(N)^*)-1-B(N)$, where $B(N)=\sum_{\ell\mid N}B(\ell,N,W_N)$. As $\prod_{q\mid N}(\frac12+\frac1{\sqrt q})\leqslant(\frac12+\frac1{\sqrt2})(\frac12+\frac1{\sqrt3})<\frac32$, \Cref{lem:upperBoundSumBps} gives $B(N)=O(\sqrt N\log N)$, which proves the first assertion. For the second, \Cref{cor:gstarLowerBound} gives $g(X_0(N)^*)\geqslant N/(12\cdot2^{\omega(N)})-O(\sqrt N\log N)$. By \cite[Th\'eor\`eme~11]{Robin83}, $\omega(N)\leqslant 1.3841\,\log N/\log\log N$, so $2^{\omega(N)}\leqslant N^{0.96/\log\log N}$, which grows slower than any fixed positive power of $N$. Hence the number of such heights is at least $N/(12\cdot2^{\omega(N)})-O(\sqrt N\log N)\geqslant N^{1-\varepsilon}$ for all sufficiently large $N$.
\end{proof}

\section{Examples}\label{sec:examples}
The goal of this section is to illustrate our methods and to prove \Cref{prop:examples}, in the hope of motivating further computational work using our code. To this end, we describe how we determined the rational points on $X_0(N)^*$ for $N\in\{187,247,319\}$. 

\begin{remark}\label{rem:rank-genus-condition} In order to use the $p$-adic heights supported only at $p$ constructed in this paper, one not only requires $N$ squarefree but also $J_0(N)^*$ to have algebraic rank equal to its dimension. This condition holds for all squarefree $N<10000$ except $5077$. This can be seen using a standard argument together with a query to the LMFDB.

We know BSD holds for simple abelian varieties in $J_0(N)$ corresponding to newforms of analytic rank $\leq 1$. When $N$ is squarefree, all newforms in $S_2(\Gamma_0(N))$ that are invariant under all Atkin--Lehners have odd analytic rank. Hence for $J_0(N)^*$ to be of algebraic rank strictly larger than its dimension one needs the existence of a modular form of some level $M \mid N$ of odd analytic rank $\geq 3$. The LMFDB has complete data on this for $M<10000$, and a simple LMFDB query\footnote{\url{https://www.lmfdb.org/ModularForm/GL2/Q/holomorphic/?level_type=squarefree&level=1-10000&weight=2&analytic_rank=3-}} shows that \LMFDBLabelMF{5077.2.a.a} is the only such modular form.

Since every simple factor $A$ of $J_0(N)^*$ is of $\GL_2$-type with $\operatorname{rk}A(\Q)=\dim A$, the argument of \cite[Lemma~7]{DograLeFourn}, which uses the $p$-adic analytic subgroup theorem, shows that $$\log: J_0(N)^*(\Q) \otimes_\Z \Q_p \to H^0(X_0(N)^*_{\Q_p},\Omega^1)^\vee$$ is an isomorphism.
\end{remark}

For these three levels the dimension count $g>B+1$ of \Cref{cor:clean_criterion} fails ($g=3\leqslant 7=B+1$ for $N=187$, and $g=4\leqslant 7=B+1$ for $N=247,319$). Nevertheless, in each case the algorithm of \Cref{prop:correspondence_algorithm} produces a clean correspondence%whose $p$-adic height is supported at $p$
, see \Cref{ex:dual_graph_one_long_loop}. In all three cases the Jacobian $J_0(N)^*$ has Mordell--Weil rank equal to the genus $g$, which is the hypothesis needed for quadratic Chabauty. All scripts and their output can be found in the directories \repo{computations}{computations} and \repo{logs}{logs} of our repository.

To run the quadratic Chabauty implementation \texttt{QCModAffine} from \cite{QCMod, BDMTV2}, we need a \defn{plane model} of $X_0(N)^*$, i.e.\ a polynomial $f\in\Q[x,y]$ such that the affine curve $f(x,y)=0$ is birational to $X_0(N)^*$; \texttt{QCModAffine} moreover requires $f$ to be monic in $y$ (with $p$-integral coefficients). We find plane models with the code of \cite{AABCCKW}, following the rough strategy from that paper, but without using gonal maps, as these did not lead to plane models we could use. Briefly, we start from the canonical model $X_0(N)^*\subseteq\P^{g-1}$ and choose linear forms $L_1,\dots,L_4$ on $\P^{g-1}$ at random, with small integral coefficients. With $\tau_x\colonequals[L_1:L_2]$ and $\tau_y\colonequals[L_3:L_4]$, we consider the composite
\[
\varphi\colon X_0(N)^*\xrightarrow{\;\tau_x\times\tau_y\;}\P^1\times\P^1\xrightarrow{\;\text{Segre}\;}\P^3_{[w:x:y:z]}\xrightarrow{\;\text{projection}\;}\P^2_{[x:y:z]},
\]
where the Segre map is $([s_0:s_1],[t_0:t_1])\mapsto[s_0t_0:s_0t_1:s_1t_0:s_1t_1]$ and the projection is from the point $[1:0:0:0]$. On the chart $z\neq0$, the map $\varphi$ is given by $x=L_1/L_2$ and $y=L_3/L_4$, so if $\varphi$ is birational onto its image, then the polynomial relation $f(x,y)=0$ between these two functions is a plane model. We first make the equation monic in $y$, if necessary, as in
\cite[Section~3]{AABCCKW}. We keep such a model if, for some
prime $p\nmid N$ with $p\leqslant 70$ suitable for the
implementation, every residue disc of $X_0(N)^*(\Q_p)$ is
\defn{good} for it (in the sense of
\cite[Definition~3.1]{AABCCKW}, following
\cite[Definitions~2.8 and~2.10]{BalTui}).
Since \texttt{QCModAffine} proves completeness only in good
residue discs, this ensures that all of $X_0(N)^*(\Q)$ is
covered. The value of using the code of
\cite{AABCCKW} lies in performing preliminary model and prime
checks before running quadratic Chabauty.
The plane model we obtain in this way enables us to use \texttt{QCMod},
but that computation itself can still be time-consuming.

In the computations, we pass to \texttt{QCModAffine} a basis $f_1,\dots,f_k$ of the polynomials $f$ of degree $<g$ for which $f(T_p)$ satisfies all linear conditions of \Cref{alg:clean_correspondences} except the linear relation that ensures trace 0. The reason for not making them have trace zero is that \texttt{QCModAffine} internally uses the clean correspondences $\tr(f_1(T_p))f_i(T_p)-\tr(f_i(T_p))f_1(T_p)$ for $2\leqslant i\leqslant k$, which have trace $0$.

Every rational point we find is a cusp or a CM point, in accordance with Elkies' conjecture. For squarefree $N$ the curve $X_0(N)^*$ has exactly one cusp, since $X_0(N)$ has exactly $2^{\omega(N)}$ cusps and $W_N$ acts transitively on them. We use the code of Assaf--Hashimoto~\cite{SachiEran2026} to check that the remaining points are CM points, and to determine their discriminants.

\subsection{Dual graphs}\label{subsec:examples_dual_graphs}
By \Cref{cor:dual_graph_quotient}, the dual graph of $X_0(N)^*$ at a prime $\ell\mid N$ is a flower graph, and by \Cref{lem:local_height_vanishes_long_edges} only its long loops, i.e.\ those of length $>1$, impose conditions on the correspondence. In the figures below we draw the dual graph of the minimal regular model, which is obtained by subdividing every loop of length $l$ into $l$ edges (see \Cref{subsec:reduction}). Loops of length $1$ are not drawn. %correspond to nodes of the central component and are not drawn.

\begin{example}\label{ex:dual_graph_one_long_loop}
For $(N,\ell)=(187,17)$, $(247,13)$ and $(319,29)$, the flower graph of $X_0(N)^*$ at $\ell$ has exactly one long loop, of length $2$, so the dual graph of the minimal regular model is the one in \Cref{fig:187starMod17}. At the other prime dividing $N$ (namely $\ell=11$, $19$ and $11$, respectively) all loops have length $1$, so the special fibre of the minimal regular model
is irreducible. Hence, for each of these levels, a correspondence only has to satisfy one linear condition for its $p$-adic height to be supported at $p$. This explains why the algorithm of \Cref{prop:correspondence_algorithm} succeeds although $g\leqslant B+1$: here the bound $B$ on the number of long loops is not sharp.
\end{example}

\begin{figure}[h!]
    \centering
    \begin{tikzpicture}[
  every node/.style={circle, fill, inner sep=1.6pt}
]
  \draw (0,0) circle (1);

  \node at (-1,0) {};
  \node at (1,0) {};
\end{tikzpicture}
    \caption{Dual graph of the minimal regular model of $X_0(187)^*$ at $\ell = 17$, of $X_0(247)^*$ at $\ell=13$ and of $X_0(319)^*$ at $\ell=29$}
    \label{fig:187starMod17}
\end{figure}

\begin{example}\label{ex:dual_graph_290}
For $N=290$ and $\ell=2$, the flower graph of $X_0(290)^*$ has two long loops, of lengths $2$ and $4$. Hence the dual graph of the minimal regular model has five vertices, see \Cref{fig:290starMod2}.
\end{example}

\begin{figure}[h]
    \centering
\begin{tikzpicture}[
  every node/.style={circle, fill, inner sep=1.6pt}
]
  \coordinate (a) at (0,0);
  \coordinate (b) at (2,0);
  \coordinate (c) at (3,1);
  \coordinate (d) at (4,0);
  \coordinate (e) at (3,-1);

  % Left loop
  \draw (a) arc[start angle=180, end angle=0, radius=1];
  \draw (a) arc[start angle=180, end angle=360, radius=1];

  % Right loop (4 vertices total, including the shared vertex)
  \draw (b) arc[start angle=180, end angle=90, radius=1];
  \draw (c) arc[start angle=90, end angle=0, radius=1];
  \draw (d) arc[start angle=0, end angle=-90, radius=1];
  \draw (e) arc[start angle=-90, end angle=-180, radius=1];

  % Vertices
  \node at (a) {};
  \node at (b) {};
  \node at (c) {};
  \node at (d) {};
  \node at (e) {};
\end{tikzpicture}
\caption{Dual graph of the minimal regular model of $X_0(290)^*$ at $\ell =2$}\label{fig:290starMod2}
\end{figure}

\subsection{Rational points}\label{subsec:examples_rational_points}
\begin{example}\label{ex:X0_187}
The curve $X_0(187)^*$ has genus $3$ and exactly six rational points: the cusp and five CM points, of discriminants $-8,-19,-43,-51,-187$.

At the prime $p=37$ of good reduction we use the clean correspondence $2T_{37}^2+11T_{37}-111$ of trace zero (see \repo{computations/X0_187star/qc_p37_mws.m}{qc\_p37\_mws.m}). Quadratic Chabauty at $p=37$ leaves $54$ fake solutions, which we eliminate with the Mordell--Weil sieve as in \cite[\S3.4]{BDMTV2}, adapted to non-hyperelliptic curves, with modulus $M=37^3$ and auxiliary primes $\leqslant 4127$. We take $M=37^3$ because $\Z/37^3\Z\hookrightarrow J(\F_{4127})$. After the sieve $0$ cosets remain, so no fake point is rational and the known rational points are all of $X_0(187)^*(\Q)$.

%We recall the setup of the sieve. Let $C$ be a curve over $\Q$ of genus $\geqslant 2$ and $J$ its Jacobian, and let $\ell$ be a prime of good reduction. Writing
%\[
%  J^m(\Q_\ell) = \{ D \in J(\Q_\ell) : D \equiv 0 \pmod{p^m} \},
%\]
%the filtration
%\[
%  J(\Q_\ell) \supset J^1(\Q_\ell) \supset J^2(\Q_\ell)
%    \supset J^3(\Q_\ell) \supset \cdots
%\]
%satisfies
%\[
%  J(\Q_\ell) / J^1(\Q_\ell) \cong J(\F_\ell), \qquad
%  J^m(\Q_\ell) / J^{m+1}(\Q_\ell) \cong (\Z/p\Z)^g
%  \quad \text{for } m \ge 1.
%\]
%Since six of our fake residue discs coincide with (six) rational points mod $p=31$, but not mod $31^2$, we are forced to use ``deep'' information. This affects how we choose auxiliary primes $q\in S$. We want at least one prime $q \in S$ such that $p \mid J(\F_{q'})$ (since $p$ is the only prime factor of those higher quotients). Additional primes $q' \in S$ can be useful only if $|J(\F_{q'})|$ is divisible by $p$ or $|J(\F_{q'})|$ and $|J(\F_{q})|$ share a common factor (otherwise we get a surjection into $J(\F_{q}) \times J(\F_q')$, implying that no cosets get excluded).

The differences of known small points generate a finite index subgroup $G \subseteq J(\Q)$. Since the torsion subgroup is trivial,  we show that the index $[J(\Q) : G]$ is coprime to $37$, as in \cite[§\,5.1]{MaartenFilipCubic}, to get $G/MG = J(\Q)/MJ(\Q)$, which is enough for running the Mordell--Weil sieve.
See \repo{computations/X0_187star/qc_p37_mws.m}{qc\_p37\_mws.m} and \repo{logs/X0_187star/qc_p37_mws.txt}{its output}.
\end{example}

\begin{example}\label{ex:X0_247_319}
The curves $X_0(247)^*$ and $X_0(319)^*$ have genus $4$, and each has exactly five rational points: the cusp and four CM points, of discriminants $-3,-12,-27,-91$ for $N=247$, and $-7,-28,-88,-187$ for $N=319$.

For $N=247$ all residue discs are good for our plane model and $p=31$, and \texttt{QCModAffine}, with the clean correspondences $15T_p^2+7T_p-676$ and $5T_p^3-374T_p+372$, returns exactly the five known rational points and no fake solutions, so no sieve is needed; see \repo{computations/X0_247star/qc_p31.m}{qc\_p31.m} and \repo{logs/X0_247star/qc_p31.txt}{its output}. The same holds for $N=319$ and $p=61$, with the clean correspondences $7T_p^2-9T_p-542$ and $7T_p^3-429T_p+998$; see \repo{computations/X0_319star/qc_p61.m}{qc\_p61.m} and \repo{logs/X0_319star/qc_p61.txt}{its output}.
\end{example}

\appendix
\section{Measures on metric graphs}\label{app:measures}

Here we collect some statements that seem to be well known to the experts, but as we could not find a good reference for them, we decided to include them here with a self-contained proof.

\begin{lemma}\label{lem:laplacian_mass}
Let $g\colon\graph \to\R$ be piecewise polynomial. Then $\lvert \nabla^2(g)\rvert=0$. Conversely, a measure $\mu$ with constant $\mu_e$ for every $e$ is of the form $\nabla^2(g)$ for some continuous edge-wise quadratic $g$ if and only if $\lvert \mu\rvert=0$, and then $g$ is unique up to an additive constant.
\end{lemma}

\begin{proof}
For the first assertion, $\int_0^{l(e)}g_e''\,\mathrm{d}x_e=g_e'(l(e))-g_e'(0)$, while summing the vertex coefficients over $w\in V(\graph)$ counts each edge $e$ once with $+g_e'(0)$ (at $w=s(e)$) and once with $-g_e'(l(e))$ (at $w=t(e)$). The two sums cancel.

For the converse, let $\mathcal P$ be the space of continuous edge-wise quadratic functions $\graph_\Q\to\R$ and $\mathcal M$ the space of measures with constant edge densities. A member of $\mathcal P$ is determined by its values at the vertices together with the leading coefficient on each edge, so $\dim\mathcal P=\#V(\graph)+\#E(\graph)=\dim\mathcal M$, and $\nabla^2$ maps $\mathcal P$ to $\mathcal M$. Its kernel consists of constants by \cite[Corollary 2(i)]{BakerFaber2006} (alternatively, if $\nabla^2(g)=0$, then $g$ is affine on each edge and its outgoing slopes sum to $0$ at every vertex, so $g$ is harmonic on the connected graph $\graph$ and hence constant).
Therefore $\nabla^2(\mathcal P)$ has codimension $1$ in $\mathcal M$; by the first assertion it is contained in the codimension-$1$ subspace $\lvert \cdot\rvert=0$, so the two coincide.
\end{proof}

The next two lemmas justify the normalisations in \Cref{def:corr_measure}: the factor $\tfrac{1}{l(e)}$ makes $\mu_F$ invariant under subdivision, and the factor $\tfrac12$ makes its total mass vanish exactly for trace-zero correspondences.

\begin{lemma}\label{lem:measure_subdivision}
Let $\graph'$ be obtained from $\graph$ by subdividing an edge $e$ into consecutive, consistently oriented edges $e_1,\dots,e_n$ of lengths $l_1,\dots,l_n$ with $\sum_i l_i=l(e)$, the new vertices carrying rational components, and let $F'$ be the endomorphism induced on $\graph'$. Then $\mu_{F'}=\mu_F$ under the identification $\graph'_\Q=\graph_\Q$.
\end{lemma}

\begin{proof}
Sending a cycle $\gamma=\sum_f a_f\,f$ on $\graph$ to the cycle on $\graph'$ with the coefficient $a_e$ on each of $e_1,\dots,e_n$ identifies $\mathrm H_1(\graph,\Q)$ with $\mathrm H_1(\graph',\Q)$, compatibly with $F$, and it is an isometry for the pairing of \Cref{def:cycle_space} because $\sum_i a_e^2\,l_i=a_e^2\,l(e)$. We suppress it from the notation. For $\gamma\in\mathrm H_1(\graph,\Q)$ we then have $e_i^*(\gamma)=e^*(\gamma)$ and $e_i\cdot\gamma=l_i\,e^*(\gamma)=\tfrac{l_i}{l(e)}\,(e\cdot\gamma)$, so the orthogonal projections satisfy $\pi'(e_i)=\tfrac{l_i}{l(e)}\,\pi(e)$. Hence
\[
  \frac{1}{l(e_i)}\,e_i^*F\left(\pi'(e_i)\right)=\frac{1}{l_i}\cdot\frac{l_i}{l(e)}\,e^*F\left(\pi(e)\right)=\frac{1}{l(e)}\,e^*F\left(\pi(e)\right),
\]
so the edge part of $\mu_{F'}$ has the same density along $e_1\cup\dots\cup e_n$ as the edge part of $\mu_F$ has along $e$. The new vertices carry rational components, whose Jacobians are trivial, so they contribute nothing to the vertex part, and the vertex part is unchanged at the vertices of $\graph$.
\end{proof}

\begin{lemma}\label{lem:measure_mass}
The measure $\mu_F$ of \Cref{def:corr_measure} has total mass
\[
  \lvert \mu_F\rvert=\tfrac12\operatorname{Tr}\left(F\mid \mathrm H^1_{\mathrm{dR}}(X_{\Q_\ell})\right).
\]
In particular $\lvert \mu_F\rvert=0$ if and only if the correspondence inducing $F$ has trace $0$.
\end{lemma}

\begin{proof}
The measure $\lvert\mathrm{d}x_e\rvert$ has total mass $l(e)$, so the edge part of $\mu_F$ contributes $\sum_{e}e^*F(\pi(e))$. The endomorphism $F\circ\pi$ of $\Q\,E(\graph)$ annihilates $\mathrm H_1(\graph,\Q)^\perp$ and restricts to $F$ on $\mathrm H_1(\graph,\Q)$, whence
\[
  \sum_{e\in E(\graph)}e^*F\left(\pi(e)\right)=\operatorname{Tr}\left(F\circ\pi\mid \Q\,E(\graph)\right)=\operatorname{Tr}\left(F\mid \mathrm H_1(\graph,\Q)\right).
\]
The vertex part contributes $\tfrac12\sum_w\operatorname{Tr}(F\mid V_p(X_w))$. On the other hand, the monodromy weight filtration on $V_p(X)$ recalled before \Cref{def:corr_endomorphism} is preserved by $Z$, and its graded pieces are $\mathrm H_1(\graph,\Q)\otimes\Q_p$, $\bigoplus_{w}V_p(X_w)$ and $\mathrm H_1(\graph,\Q)^\vee\otimes\Q_p(1)$. The trace of $Z$ on $\mathrm H_1(\graph,\Q)\otimes\Q_p$ is the same as on the dual space $\mathrm H_1(\graph,\Q)^\vee\otimes\Q_p(1)$. Moreover $\operatorname{Tr}(Z\mid V_p(X))=\operatorname{Tr}(Z\mid \mathrm H^1_{\mathrm{dR}}(X_{\Q_\ell}))$, as both equal the trace of $Z$ on $\mathrm H_1(J(\C),\Q)$. Therefore
\[
  \operatorname{Tr}\left(F\mid \mathrm H^1_{\mathrm{dR}}(X_{\Q_\ell})\right)=2\operatorname{Tr}\left(F\mid \mathrm H_1(\graph,\Q)\right)+\sum_{w\in V(\graph)}\operatorname{Tr}\left(F\mid V_p(X_w)\right)=2\lvert \mu_F\rvert. \qedhere
\]
\end{proof}

\printbibliography

@article {AABCCKW,
    AUTHOR = {Ad\v{z}aga, Nikola and Arul, Vishal and Beneish, Lea and Chen, Mingjie and Chidambaram, Shiva and Keller, Timo and Wen, Boya},
     TITLE = {{Quadratic Chabauty for Atkin-Lehner Quotients of Modular Curves of Prime Level and Genus 4, 5, 6}},
   JOURNAL = {Acta Arith.},
  FJOURNAL = {Acta Arithmetica},
    VOLUME = {208},
      YEAR = {2023},
    NUMBER = {1},
     PAGES = {15--49},
 SHORTHAND = {AABC+23}
}

@incollection{BakerFaber2006,
  author    = {Baker, Matthew and Faber, Xander},
  title     = {Metrized Graphs, {L}aplacian Operators, and Electrical Networks},
  booktitle = {Quantum Graphs and Their Applications},
  editor    = {Berkolaiko, Gregory and Carlson, Robert and Fulling, Stephen A. and Kuchment, Peter},
  series    = {Contemporary Mathematics},
  volume    = {415},
  pages     = {15--33},
  year      = {2006},
  publisher = {American Mathematical Society},
  address   = {Providence, RI},
  doi       = {10.1090/conm/415/07858}
}

@article{BalakrishnanDogra1,
 author = {Balakrishnan, Jennifer S. and Dogra, Netan},
 title = {Quadratic {Chabauty} and rational points. {I}: {{\(p\)}}-adic heights},
 fjournal = {Duke Mathematical Journal},
 journal = {Duke Math. J.},
 issn = {0012-7094},
 volume = {167},
 number = {11},
 pages = {1981--2038},
 year = {2018},
 language = {English},
 doi = {10.1215/00127094-2018-0013},
 zbMATH = {6941816},
 Zbl = {1401.14123}
}

@misc{SachiEran2026,
  author        = {Assaf, Eran and Hashimoto, Sachi},
  title         = {Hyperelliptic Atkin--Lehner quotients of Shimura curves},
  year          = {2026},
  eprint        = {2606.25153},
  archivePrefix = {arXiv},
  doi           = {10.48550/arXiv.2606.25153}
}

@misc{ShimuraCurveALQuotients,
  author = {Assaf, Eran and Hashimoto, Sachi},
  title = {ShimuraCurveALQuotients},
  year = {2024},
  publisher = {GitHub},
  journal = {GitHub repository},
  howpublished = {\url{https://github.com/assaferan/ShimuraCurveALQuotients}},
  shorthand = {AH24}
}

@article{AtkinLehner1970,
  author  = {Atkin, A. O. L. and Lehner, Joseph},
  title   = {Hecke operators on {$\Gamma_0(m)$}},
  journal = {Mathematische Annalen},
  volume  = {185},
  year    = {1970},
  pages   = {134--160},
  doi     = {10.1007/BF01359701},
  mrnumber = {0268123}
}

@article{Arpin24,
  author        = {Arpin, Sarah},
  title         = {Adding level structure to supersingular elliptic curve isogeny graphs},
  journal       = {J. Th\'eor. Nombres Bordeaux},
  doi           = {10.5802/jtnb.1283},
  eprint        = {2203.03531},
  archivePrefix = {arXiv},
  year          = {2024},
  pages         = {405--443}
}

@misc{BettsDogra,
      title={The local theory of unipotent Kummer maps and refined Selmer schemes}, 
      author={L. Alexander Betts and Netan Dogra},
      year={2020},
      eprint={1909.05734},
      archivePrefix={arXiv},
      url={https://arxiv.org/abs/1909.05734}, 
}

@article{hyperellipticQC,
  author    = {Betts, L. Alexander and Duque-Rosero, Juanita and Hashimoto, Sachi and Spelier, Pim},
  title     = {Local heights on hyperelliptic curves and quadratic {C}habauty},
  journal   = {Transactions of the American Mathematical Society, Series B},
  year      = {2026},
  volume    = {13},
  pages     = {240--300},
  doi       = {10.1090/btran/242},
  publisher = {American Mathematical Society}
}

@incollection{FrancescaOana,
    AUTHOR = {Bianchi, Francesca and Padurariu, Oana},
     TITLE = {Rational points on rank 2 genus 2 bielliptic curves in the {LMFDB}},
 BOOKTITLE = {Lu{C}a{NT}: {LMFDB}, Computation, and Number Theory},
    SERIES = {Contemp. Math.},
    VOLUME = {796},
     PAGES = {215--242},
 PUBLISHER = {Amer. Math. Soc., Providence, RI},
      YEAR = {2024},
       DOI = {10.1090/conm/796/16003},
       URL = {https://doi.org/10.1090/conm/796/16003},
}

@incollection{ColemanGross,
  author    = {Coleman, Robert F. and Gross, Benedict H.},
  title     = {$p$-adic Heights on Curves},
  booktitle = {Algebraic Number Theory -- in honor of K. Iwasawa},
  publisher = {Mathematical Society of Japan},
  year      = {1989},
  volume    = {17},
  series    = {Advanced Studies in Pure Mathematics},
  pages     = {73--81},
  doi       = {10.2969/aspm/01710073},
  url       = {https://projecteuclid.org/ebooks/advanced-studies-in-pure-mathematics/Algebraic-Number-Theory--in-honor-of-K-Iwasawa/chapter/p-adic-Heights-on-Curves/10.2969/aspm/01710073}
}

@book{Cox,
  title={Primes of the Form $x^2 + ny^2$: Fermat, Class Field Theory, and Complex Multiplication},
  author={Cox, D.A.},
  isbn={9780471506546},
  lccn={89005555},
  series={Monographs and textbooks in pure and applied mathematics},
  year={1989},
  publisher={Wiley}
}

@article{classGroupTabulation,
    AUTHOR = {Mosunov, A. S. and Jacobson, Jr., M. J.},
     TITLE = {Unconditional class group tabulation of imaginary quadratic
              fields to {$|\Delta|<2^{40}$}},
   JOURNAL = {Math. Comp.},
  FJOURNAL = {Mathematics of Computation},
    VOLUME = {85},
      YEAR = {2016},
    NUMBER = {300},
     PAGES = {1983--2009},
      ISSN = {0025-5718,1088-6842},
       DOI = {10.1090/mcom3050},
       URL = {https://doi.org/10.1090/mcom3050},
}

@book{Mumford,
  author    = {Mumford, David},
  title     = {Abelian Varieties},
  series    = {Tata Institute of Fundamental Research Studies in Mathematics},
  volume    = {5},
  edition   = {2},
  publisher = {Hindustan Book Agency},
  address   = {New Delhi},
  year      = {2008},
  note      = {Corrected reprint of the 1974 edition, with appendices by C. P. Ramanujam and Yuri Manin},
  isbn      = {978-81-85931-86-9}
}

@article{MaartenFilipCubic,
  author   = {Derickx, Maarten and Najman, Filip},
  title    = {Torsion of elliptic curves over cyclic cubic fields},
  journal  = {Mathematics of Computation},
  volume   = {88},
  number   = {319},
  year     = {2019},
  pages    = {2443--2459},
  doi      = {10.1090/mcom/3408},
  mrnumber = {3957900}
}

@article{Deuring1951,
author = {Deuring, Max},
journal = {Jahresbericht der Deutschen Mathematiker-Vereinigung},
pages = {24-41},
title = {Die Anzahl der Typen von Maximalordungen einer definiten Quaternionenalgebra mit primer Grundzahl.},
url = {http://eudml.org/doc/146345},
volume = {54},
year = {1951},
}

@article{Edixhoven,
  author  = {Edixhoven, Bas},
  title   = {Minimal resolution and stable reduction of {$X_0(N)$}},
  journal = {Annales de l'Institut Fourier},
  volume  = {40},
  number  = {1},
  year    = {1990},
  pages   = {31--67},
  doi     = {10.5802/aif.1202},
  mrnumber = {92f:11080},
  zbl     = {0679.14009}
}

@Inbook{Elkies,
author="Elkies, Noam D.",
editor="Cremona, John E.
and Lario, Joan-Carles
and Quer, Jordi
and Ribet, Kenneth A.",
title="On Elliptic $K$-curves",
bookTitle="Modular Curves and Abelian Varieties",
year="2004",
publisher="Birkh{\"a}user Basel",
address="Basel",
pages="81--91"
}

@book{Fulton,
  author    = {Fulton, William},
  title     = {Intersection Theory},
  publisher = {Springer},
  year      = {1998},
  edition   = {2nd},
  isbn      = {978-1-4612-1700-8},
  doi       = {10.1007/978-1-4612-1700-8},
  url       = {https://doi.org/10.1007/978-1-4612-1700-8}
}

@article{KhareWintenberger1,
  author    = {Khare, Chandrashekhar and Wintenberger, Jean-Pierre},
  title     = {Serre's modularity conjecture (I)},
  journal   = {Inventiones Mathematicae},
  volume    = {178},
  number    = {3},
  pages     = {485--504},
  year      = {2009},
  doi       = {10.1007/s00222-009-0205-7},
  bibcode   = {2009InMat.178..485K}
}

@article{KhareWintenberger2,
  author    = {Khare, Chandrashekhar and Wintenberger, Jean-Pierre},
  title     = {Serre's modularity conjecture (II)},
  journal   = {Inventiones Mathematicae},
  volume    = {178},
  number    = {3},
  pages     = {505--586},
  year      = {2009},
  doi       = {10.1007/s00222-009-0206-6},
  bibcode   = {2009InMat.178..505K}
}

@misc{LMFDB,
  shorthand    = {LMFDB},
  author       = {The {LMFDB Collaboration}},
  title        = {The {L}-functions and modular forms database},
  howpublished = {\url{https://www.lmfdb.org}},
  year         = {2026},
  note         = {[Online; accessed 7 October 2026]},
}

@article{ANTStar,
    AUTHOR = {Ad\v zaga, Nikola and Chidambaram, Shiva and Keller, Timo and
              Padurariu, Oana},
     TITLE = {Rational points on hyperelliptic {A}tkin-{L}ehner quotients of
              modular curves and their coverings},
   JOURNAL = {Res. Number Theory},
  FJOURNAL = {Research in Number Theory},
    VOLUME = {8},
      YEAR = {2022},
    NUMBER = {4},
     PAGES = {Paper No. 87, 24},
      ISSN = {2522-0160,2363-9555},
   MRCLASS = {11G18 (11G30 14G05 14G35 14H50)},
  MRNUMBER = {4496691},
MRREVIEWER = {Nicholas\ George\ Triantafillou},
       DOI = {10.1007/s40993-022-00388-9},
       URL = {https://doi.org/10.1007/s40993-022-00388-9},
}

@article {BalTui,
    AUTHOR = {Balakrishnan, Jennifer S. and Tuitman, Jan},
     TITLE = {Explicit {C}oleman integration for curves},
   JOURNAL = {Math. Comp.},
  FJOURNAL = {Mathematics of Computation},
    VOLUME = {89},
      YEAR = {2020},
    NUMBER = {326},
     PAGES = {2965--2984},
      ISSN = {0025-5718},
   MRCLASS = {11S80 (11Y35 11Y50 14G05)},
  MRNUMBER = {4136553},
MRREVIEWER = {Daniel Barsky},
       DOI = {10.1090/mcom/3542},
       URL = {https://doi.org/10.1090/mcom/3542},
}

@article{BDMTV2, title={Quadratic Chabauty for modular curves: algorithms and examples}, volume={159}, DOI={10.1112/S0010437X23007170}, number={6}, journal={Compositio Mathematica}, author={Balakrishnan, Jennifer S. and Dogra, Netan and Müller, J. Steffen and Tuitman, Jan and Vonk, Jan}, year={2023}, pages={1111–1152}
}

@mastersthesis{Noordman2017,
  author  = {Noordman, Marc Paul},
  title   = {Gonality and Bad Reduction of Modular Curves},
  school  = {Mathematisch Instituut, Universiteit Leiden},
  year    = {2017},
  month   = {6},
  type    = {Master's thesis},
  note    = {Advisors: Sebastiaan Johan Edixhoven and Maarten Derickx},
  url     = {https://www.universiteitleiden.nl/binaries/content/assets/science/mi/scripties/master/algant/2016-2017/noordman-master_thesis.pdf}
}

@book{SteinBook,
 author = {Stein, William},
 title = {Modular forms, a computational approach. {With} an appendix by {Paul} {E}. {Gunnells}},
 fseries = {Graduate Studies in Mathematics},
 series = {Grad. Stud. Math.},
 issn = {1065-7339},
 volume = {79},
 isbn = {0-8218-3960-8},
 year = {2007},
 publisher = {Providence, RI: American Mathematical Society (AMS)},
 language = {English},
 zbMATH = {5080870},
 Zbl = {1110.11015}
}

@misc{BeneishWen2026,
  author = {Beneish, Lea and Wen, Boya},
  title  = {{A}tkin--{L}ehner Quotients of Modular Curves at Primes of Bad Reduction},
  year   = {2026},
  eprint = {2609.39725},
  archivePrefix={arXiv},
}

@article{BakerHasegawa,
 author = {Baker, Matthew and Hasegawa, Yuji},
 title = {Automorphisms of {{\(X_0^*(p)\)}}.},
 fjournal = {Journal of Number Theory},
 journal = {J. Number Theory},
 issn = {0022-314X},
 volume = {100},
 number = {1},
 pages = {72--87},
 year = {2003},
 language = {English},
 doi = {10.1016/S0022-314X(02)00120-8},
 zbMATH = {1934757},
 Zbl = {1088.11049}
}

@incollection{KohelHecke,
 author = {Kohel, David R.},
 title = {Hecke module structure of quaternions},
 booktitle = {Class field theory -- its centenary and prospect. Proceedings of the 7th MSJ International Research Institute of the Mathematical Society of Japan, Tokyo, Japan, June 3--12, 1998},
 isbn = {4-931469-11-6},
 pages = {177--195},
 year = {2001},
 publisher = {Tokyo: Mathematical Society of Japan},
 language = {English},
 zbMATH = {1771886},
 Zbl = {1040.11044}
}

@misc{QCMod,
  author = {Balakrishnan, Jennifer and Best, Alex and Bianchi, Francesca and Dogra, Netan and Lawrence, Brian and Müller, Jan and Triantafillou, Nicholas and Tuitman, Jan and Vonk, Jan},
  title = {QCMod},
  year = {2020},
  publisher = {GitHub},
  journal = {GitHub repository},
  howpublished = {\url{https://github.com/steffenmueller/QCMod}},
  shorthand = {BBBD+20}
}

@misc{ModelFinder,
  author = {Ad{\v{z}}aga, Nikola and Arul, Vishal and Beneish, Lea and Chen, Mingjie and Chidambaram, Shiva and Keller, Timo and Wen, Boya},
  title = {ModularCurvesX0plusG4-6},
  year = {2021},
  publisher = {GitHub},
  journal = {GitHub repository},
  howpublished = {\url{https://github.com/AWS2020QC/ModularCurvesX0plusG4-6}},
  shorthand = {AABC+21}
}

@article{BGX21,
author={Bars, Francesc and Gonz{\'a}lez, Josep and Xarles, Xavier},
title={Hyperelliptic parametrizations of $\mathbb{Q}$-curves},
journal={The Ramanujan Journal},
year={2021},
month={10},
day={01},
volume={56},
number={1},
pages={103-120},
issn={1572-9303},
}

@book{Davenport,
  title={Multiplicative Number Theory},
  author={Davenport, H.},
  note={revised by Montgomery, H.L.},
  isbn={9781475759273},
  lccn={80026329},
  series={Graduate Texts in Mathematics},
  url={https://books.google.hr/books?id=SFztBwAAQBAJ},
  year={2013},
  publisher={Springer New York}
}

@article{EdixhovenLido,
 author = {Edixhoven, Bas and Lido, Guido},
 title = {Geometric quadratic {Chabauty}},
 fjournal = {Journal of the Institute of Mathematics of Jussieu},
 journal = {J. Inst. Math. Jussieu},
 issn = {1474-7480},
 volume = {22},
 number = {1},
 pages = {279--333},
 year = {2023},
 language = {English},
 doi = {10.1017/S1474748021000244},
 zbMATH = {7662119}
}

@article{DHS23,
 author = {Duque-Rosero, Juanita and Hashimoto, Sachi and Spelier, Pim},
 title = {Geometric quadratic {Chabauty} and {{\(p\)}}-adic heights},
 fjournal = {Expositiones Mathematicae},
 journal = {Expo. Math.},
 issn = {0723-0869},
 volume = {41},
 number = {3},
 pages = {631--674},
 year = {2023},
 language = {English},
 doi = {10.1016/j.exmath.2023.05.003},
 zbMATH = {7745041},
 Zbl = {1546.14047}
}

@article{DograLeFourn,
 author = {Dogra, Netan and Le Fourn, Samuel},
 title = {Quadratic {Chabauty} for modular curves and modular forms of rank one},
 fjournal = {Mathematische Annalen},
 journal = {Math. Ann.},
 issn = {0025-5831},
 volume = {380},
 number = {1-2},
 pages = {393--448},
 year = {2021},
 language = {English},
 doi = {10.1007/s00208-020-02112-3},
 zbMATH = {7363345},
 Zbl = {1472.11191}
}

@InProceedings{Eichler1973,
author="Eichler, M.",
editor="Kuijk, Willem",
title="The Basis Problem for Modular Forms and the Traces of the Hecke Operators",
booktitle="Modular Functions of One Variable I",
year="1973",
publisher="Springer Berlin Heidelberg",
address="Berlin, Heidelberg",
pages="75--152",
}

@article{Hijikata,
author = {Hiroaki Hijikata},
title = {{Explicit formula of the traces of Hecke operators for  $\Gamma_{0}(N)$}},
volume = {26},
journal = {Journal of the Mathematical Society of Japan},
number = {1},
publisher = {Mathematical Society of Japan},
pages = {56 -- 82},
year = {1974},
doi = {10.2969/jmsj/02610056},
URL = {https://doi.org/10.2969/jmsj/02610056}
}

@article{KimTamagawa,
  author  = {Kim, Minhyong and Tamagawa, Akio},
  title   = {The l-component of the unipotent Albanese map},
  journal = {Mathematische Annalen},
  year    = {2007},
  volume  = {340},
  pages   = {223--235},
  doi     = {10.1007/s00208-007-0151-x}
}

@InProceedings{Klu77,
author="Kluit, P. G.",
editor="Serre, Jean-Pierre
and Zagier, Don Bernard",
title="On the normalizer of $\Gamma_0(N)$",
booktitle="Modular Functions of one Variable V",
year="1977",
publisher="Springer Berlin Heidelberg",
address="Berlin, Heidelberg",
pages="239--246",
isbn="978-3-540-37291-2"
}

@article{Magma,
  author  = {Bosma, Wieb and Cannon, John and Playoust, Catherine},
  title   = {The {M}agma algebra system. {I}. {T}he user language},
  journal = {J. Symbolic Comput.},
  volume  = {24},
  number  = {3-4},
  pages   = {235--265},
  year    = {1997}
}

@Inbook{MilneAbelianVarieties,
author="Milne, J. S.",
editor="Cornell, Gary
and Silverman, Joseph H.",
title="Abelian Varieties",
bookTitle="Arithmetic Geometry",
year="1986",
publisher="Springer New York",
address="New York, NY",
pages="103--150",
url="https://doi.org/10.1007/978-1-4613-8655-1_5"
}

@article{Ogg74,
     author = {Ogg, Andrew P.},
     title = {Hyperelliptic modular curves},
     journal = {Bulletin de la Soci\'et\'e Ma\-th\'e\-ma\-tique de France},
     pages = {449--462},
     year = {1974},
     publisher = {Soci\'et\'e math\'ematique de France},
     volume = {102},
     doi = {10.24033/bsmf.1789},
     mrnumber = {51 #514},
     zbl = {0314.10018},
     language = {en},
     url = {https://www.numdam.org/articles/10.24033/bsmf.1789/}
}

@article{Ramare1,
  title={Approximate formulae for $L(1, \chi)$},
  author={Ramar{\'e}, Olivier},
  journal={Acta Arithmetica},
  volume={100},
  number={3},
  pages={245--266},
  year={2001},
  publisher={THE INSTITUTE OF MATHEMATICS}
}

@article{Ramare2,
  title={Approximate formulae for $L(1, \chi)$, II},
  author={Ramar{\'e}, Olivier},
  journal={Acta Arithmetica},
  volume={112},
  number={2},
  pages={141--149},
  year={2004}
}

@article{Ribet1990,
  author  = {Ribet, Kenneth A.},
  title   = {On modular representations of $\mathrm{Gal}(\bar{\mathbb{Q}}/\mathbb{Q})$ arising from modular forms},
  journal = {Inventiones Mathematicae},
  year    = {1990},
  volume  = {100},
  number  = {2},
  pages   = {431--476},
  doi     = {10.1007/bf01231195}
}

@InCollection{Ribet2004,
  author={Ribet, Kenneth A.},
  series={Progr. Math.},
  volume={224},
  publisher={Birkh\"{a}user, Basel},
  booktitle={Modular curves and abelian varieties},
  year={2004},
  pages={241--261},
  title={{Abelian varieties over $\Q$ and modular forms}},
}

@article{xue2009minimal,
  title={Minimal resolution of Atkin--Lehner quotients of $X_0(N)$},
  author={Xue, Hui},
  journal={Journal of Number Theory},
  volume={129},
  number={9},
  pages={2072--2092},
  year={2009},
  publisher={Elsevier}
}

@InProceedings{DeRa,
author="Deligne, P.
and Rapoport, M.",
editor="Deligne, Pierre
and Kuijk, Willem",
title="Les Sch{\'e}mas de Modules de Courbes Elliptiques",
booktitle="Modular Functions of One Variable II",
year="1973",
publisher="Springer Berlin Heidelberg",
address="Berlin, Heidelberg",
pages="143--316"
}

@misc{LFHK,
      title={Rational points on $X_0(N)^*$ when $N$ is non-squarefree}, 
      author={Sachi Hashimoto and Timo Keller and Samuel {Le Fourn}},
      year={2025},
      eprint={2505.00680},
      archivePrefix={arXiv},
      url={https://arxiv.org/abs/2505.00680}, 
}

@book{Liu2002,
  author       = {Qing Liu},
  title        = {Algebraic Geometry and Arithmetic Curves},
  series       = {Oxford Graduate Texts in Mathematics},
  volume       = {6},
  publisher    = {Oxford University Press},
  address      = {Oxford},
  year         = {2002},
  pages        = {xvi+576},
  note         = {Translated from the French by Reinie Ern{\'e}},
}

@book{Voight2021,
  author    = {Voight, John},
  title     = {Quaternion Algebras},
  year      = {2021},
  publisher = {Springer Cham},
  series    = {Graduate Texts in Mathematics},
  volume    = {288},
  doi       = {10.1007/978-3-030-56694-4},
  isbn      = {978-3-030-56694-4}
}

@book{KatzMazur,
  title={Arithmetic Moduli of Elliptic Curves},
  author={Katz, N.M. and Mazur, B.},
  number={no. 108},
  isbn={9780691083520},
  lccn={lc83043079},
  series={Annals of Mathematics Studies},
  url={https://books.google.hr/books?id=M1IT0J_sPr8C},
  year={1985},
  publisher={Princeton University Press}
}

@misc{Isabel,
      title={Quadratic Chabauty for Atkin-Lehner quotients of modular curves via weakly holomorphic modular forms: Hodge Filtrations}, 
      author={Isabel Rendell},
      year={2025},
      eprint={2509.02291},
      archivePrefix={arXiv},
      url={https://arxiv.org/abs/2509.02291}, 
}

@article{Robin83,
  title={Estimation de la fonction de Tchebychef $\theta$ sur le k-i{\`e}me nombre premier et grandes valeurs de la fonction $\omega$ (n) nombre de diviseurs premiers de n},
  author={Robin, Guy},
  journal={Acta Arithmetica},
  volume={42},
  number={4},
  pages={367--389},
  year={1983},
  publisher={Polska Akademia Nauk. Instytut Matematyczny PAN}
}

@book{Silverman1,
  author    = {Silverman, Joseph H.},
  title     = {The Arithmetic of Elliptic Curves},
  series    = {Graduate Texts in Mathematics},
  volume    = {106},
  edition   = {2},
  publisher = {Springer},
  address   = {Dordrecht},
  year      = {2009},
  doi       = {10.1007/978-0-387-09494-6},
  isbn      = {978-0-387-09493-9}
}

@article{Schoof87,
 author = {Schoof, Ren{\'e}},
 title = {Nonsingular plane cubic curves over finite fields},
 fjournal = {Journal of Combinatorial Theory. Series A},
 journal = {J. Comb. Theory, Ser. A},
 issn = {0097-3165},
 volume = {46},
 number = {1-2},
 pages = {183--211},
 year = {1987},
 language = {English},
 doi = {10.1016/0097-3165(87)90003-3},
 zbMATH = {4027654},
 Zbl = {0632.14021}
}

@article{Tate78,
 author = {Tate, John T.},
 title = {The arithmetic of elliptic curves},
 fjournal = {Inventiones Mathematicae},
 journal = {Invent. Math.},
 issn = {0020-9910},
 volume = {23},
 pages = {179--206},
 year = {1974},
 language = {English},
 doi = {10.1007/BF01389745},
 url = {https://eudml.org/doc/142261},
 zbMATH = {3463800},
 Zbl = {0296.14018}
}

@article{Deuring41,
 author = {Deuring, Max},
 title = {Die {Typen} der {Multiplikatorenringe} elliptischer {Funk\-tio\-nen\-k{\"o}r\-per}},
 fjournal = {Abhandlungen aus dem Mathematischen Seminar der Universit{\"a}t Hamburg},
 journal = {Abh. Math. Semin. Univ. Hamb.},
 issn = {0025-5858},
 volume = {14},
 pages = {197--272},
 year = {1941},
 language = {German},
 doi = {10.1007/BF02940746},
 zbMATH = {3039780},
 Zbl = {0025.02003}
}

\end{document}